\documentclass{article}
\usepackage[utf8]{inputenc}
\usepackage[T1]{fontenc}
\usepackage{lmodern}
\usepackage{amsmath} 
\usepackage{amssymb} 
\usepackage{amsfonts}
\usepackage{amsthm}  
\usepackage{float}
\usepackage{mathtools}
\usepackage{bm}             % boldface symbols (\bm)
\usepackage{graphicx}       % embedding of pictures
\usepackage{fancyvrb}       % improved verbatim environment
\usepackage[nottoc]{tocbibind} % makes sure that bibliography and the lists
\usepackage{dcolumn}        % improved alignment of table columns
\usepackage{booktabs}       % improved horizontal lines in tables
\usepackage{paralist}       % improved enumerate and itemize
\usepackage[usenames]{xcolor}  % typesetting in color

\usepackage{systeme} 

\usepackage{pgfplots}
\usetikzlibrary{intersections, pgfplots.fillbetween}

\usepackage{tikz} % pictures
\usepackage{tikz-3dplot} % 3dpictures
\usetikzlibrary{arrows}
\usepackage{xparse}
\usetikzlibrary{3d}
\usetikzlibrary{calc,intersections,through,backgrounds}
\usetikzlibrary{arrows}
\usepackage{tikz-3dplot} % 3dpictures

\tikzset{
    >=stealth',
    punkt/.style={
           rectangle,
           rounded corners,
           draw=black, thick,
           text width=6.5em,
           minimum height=2em,
           text centered},
    pil/.style={
           ->,
           thick,
           shorten <=2pt,
           shorten >=2pt,}
}

\include{macros}
\theoremstyle{definition}
\newtheorem{definition}{Definition}
\theoremstyle{theorem}
\newtheorem{theorem}{Theorem}

\newtheorem{lemmach}{Lemma}
\theoremstyle{theorem}

\theoremstyle{remark}

\theoremstyle{remark}
\newtheorem{remark}{Remark}
\theoremstyle{example}

\theoremstyle{example}

\theoremstyle{notation}
\newtheorem{notation}{Notation}

\theoremstyle{theorem}
\newcommand{\floor}[1]{\left\lfloor #1 \right\rfloor}

\makeatletter
\renewcommand*\env@matrix[1][*\c@MaxMatrixCols c]{%
  \hskip -\arraycolsep
  \let\@ifnextchar\new@ifnextchar
  \array{#1}}
\makeatother

\title{Sections and Chapters}
\title{Proof complexity of CSP}

\author{Azza Gaysin\thanks{This work was partly supported by the project SVV-2020-260589, the project "Grant Schemes at CU" (reg. no. CZ.$02.2.69/0.0/0.0/19_073/0016935$) and by Charles University Research Centre program [UNCE/SCI/022].}}

\date{Department of Algebra, Faculty of Mathematics and Physics,
    Charles University in Prague}

\begin{document}
\allowdisplaybreaks
\maketitle
\begin{abstract}
The CSP (constraint satisfaction problems) is a class of problems deciding whether there exists a homomorphism from an instance relational structure to a target one. The CSP dichotomy is a profound result recently proved by Zhuk \cite{10.1145/3402029} and Bulatov \cite{8104069}. It establishes that for any fixed target structure, CSP is either NP-complete or $p$-time solvable. Zhuk's algorithm solves CSP in polynomial time for constraint languages having a weak near-unanimity polymorphism. 

For negative instances of $p$-time CSPs, it is reasonable to explore their proof complexity. We show that the soundness of Zhuk's algorithm can be proved in a theory of bounded arithmetic, namely in the theory $V^1$ augmented by three special universal algebra axioms. This implies that any propositional proof system that simulates both Extended Resolution and a theory that proves the three axioms admits $p$-size proofs of all
negative instances of a fixed $p$-time CSP. 
\end{abstract}
%\keywords{proof complexity, constraint satisfaction problem, bounded arithmetic}

\section{Introduction}

An important class of NP problems are the constraint satisfaction problems (CSP). We will give its definition in Subsection \ref{CSPDEF}, but a universal
formulation is as follows: in a constraint satisfaction problem CSP($\mathcal{A}$) associated with a relational structure $\mathcal{A}$, for any relational structure over the same vocabulary $\mathcal{X}$ the question is whether $\mathcal{X}$ can be homomorphically mapped into $\mathcal{A}$. The problem 
$\mathcal{X}\mapsto_{?}$ $\mathcal{A}$ is an instance of CSP($\mathcal{A}$). A celebrated theorem of
Zhuk \cite{10.1145/3402029} and Bulatov \cite{8104069} states that for each constraint language $\mathcal{A}$, CSP($\mathcal{A}$) is either NP-complete or $p$-time decidable (see \cite{barto_et_al},\cite{1} for the history of this theorem and earlier results and conjectures).

The statement that there is no homomorphism from $\mathcal{X}$ into $\mathcal{A}$ can be encoded by a propositional tautology having atoms for the potential edges of a homomorphism. The size of this tautology, to be denoted $\neg HOM(\mathcal{X},\mathcal{A})$, is polynomial in the sizes of $\mathcal{X}$ and $\mathcal{A}$. When CSP($\mathcal{A}$) is NP-complete we cannot hope to have short propositional proofs (in any proof system) of formulas $\neg HOM(\mathcal{X},\mathcal{A})$ for all unsatisfiable instances $\mathcal{X}$ of CSP($\mathcal{A}$), as that would imply that NP is closed under complementation. However, when CSP($\mathcal{A}$) is $p$-time decidable this obstacle is removed.

Zhuk's algorithm solves polynomial time CSPs and provides a tool for the investigation of their proof complexity. In fact, for a satisfiable instance $\mathcal{X}$ of CSP($\mathcal{A}$) the algorithm produces a homomorphism from $\mathcal{X}$ to $\mathcal{A}$ as a witness of an affirmative answer. For unsatisfiable instances, on the contrary, one has no witness to the algorithm's correctness other than its run. Our main result is that the soundness of Zhuk's algorithm can be proved in a theory of bounded arithmetic, namely in the theory $V^1$ augmented with three universal algebra axioms. By the soundness here we mean that all negative answers of the algorithm are correct. Every theory of bounded arithmetic corresponds to some propositional proof system in the sense that if one proves a universal statement in the theory, the propositional translations of this statement will have polynomial proofs in the proof system. Short propositional proofs of the statement $\neg HOM(\mathcal{X},\mathcal{A})$ can be considered as witnesses for negative instances of CSP($\mathcal{A}$).

To establish the result we uses a modified framework analogous to the framework we explore for our previous result in \cite{10.1093/logcom/exab028}; there we considered a simple example of relational structures that are undirected graphs (the Hell-Nešetřil dichotomy theorem). Atserias and Ochremiak in \cite{10.1145/3265985} studied the relation between universal algebra (and CSP in particular) and proof complexity. 

The paper is organized as follows. In Section \ref{PRELIMINARIES} we recall the necessary background from universal algebra, CSP theory, proof complexity, and bounded arithmetic. In Section \ref{LINEAR} we define strong subuniverses and linear algebras, and formulate Zhuk's four cases theorem representing one of the main ideas of the whole algorithm. The outline of Zhuk's algorithm is presented in Section \ref{ZHUfgergK}. Section \ref{SOUNDNggghESS} is devoted to the soundness of Zhuk's algorithm and is divided into three principal parts. In Sections \ref{dddddf} - \ref{aaoiisoiowie6} we introduce the framework, formalize most of the notions used in the algorithm, and define a new theory of bounded arithmetic. In Section \ref{kkspqu5465} we prove the soundness of consistency reductions in the theory $V^1$. Finally, in Section \ref{LLLLLINEARG} we consider the linear case of the algorithm. The main theorem is formulated in Section \ref{kkaiery465r} and the summary of the proof is presented there.

\section{Preliminaries}\label{PRELIMINARIES}
\subsection{Basic notions from universal algebra}
This subsection is based on papers \cite{10.1145/2677161.2677165}, \cite{barto_et_al}. Some definitions and results are adopted from \cite{BurrisSankappanavar1981}.

For our purpose, we will consider only finite objects. For any non-empty domain $A$ and any natural number $n$ we call a mapping $f:A^n\mapsto A$ an $n$-ary \emph{operation} on $A$. An \emph{algebra} $\mathbb{A}=(A,f_1,f_2,...)$ is a pair of a domain $A$ and basic operations $f_1,f_2,...$ of fixed arities on $A$ from some signature $\Sigma = \{f_1,f_2,...\}$. A \emph{constraint language} $\Gamma$ is a set of relations on finite domains. A \emph{relational structure} $\mathcal{A}=(A,R_1,R_2,...)$ is a pair of a domain $A$ and relations $R_1,R_2,...$ on $A$ from some constraint language $\Gamma = \{R_1,R_2,...\}$.

We say that an $m$-ary operation $f: A^m\rightarrow A$ \emph{preservers} an $n$-ary relation $R\in A^n$ (or $f$ is the \emph{polymorphism} of $R$, or $f$ is \emph{compatible} with $R$, or $R$ is \emph{invariant} under $f$) if $f(\bar{a_1},...,\bar{a_m}) \in R$ for all choices of $\bar{a_1},...,\bar{a_m} \in R$. For any constraint language $\Gamma$ and any set of operations $O$ we will denote by $Pol(\Gamma)$ the set of all operations on $A$ preserving each relation from $\Gamma$, and by $Inv(O)$ the set of all relations on $A$ invariant under each operation from $O$. 

A \emph{term} in a signature $\Sigma$ is a formal expression that uses \emph{variables and composition of symbols} from $\Sigma$. The set of all term operations of \emph{algebra} $\mathbb{A}=(A,F)$ is called the \emph{clone of term operations} of $\mathbb{A}$, denoted by $Clone(\mathbb{A})$. A well-known theorem from universal algebra establishes the connection between algebras and relational structures.

\begin{theorem}[\cite{bergman2011universal}]\label{fjjduh87}
For any algebra $\mathbb{A}$ there exists relation structure $\mathcal{A}$ such that $Clone(\mathbb{A})=Pol(\mathcal{A})$.
\end{theorem}

In general, any set of operations $O$ on $A$ is a clone if it contains all projections and is closed under superposition, i.e. for a $k$-ary operation $f\in \mathit{O}$ and $m$-ary operations $g_1,...,g_m\in \mathit{O}$ the superposition $f[g_1,...,g_k]$ is in $\mathit{O}$ as well. We define $Clone(O)$ to be the smallest clone containing $O$. The dual object for relations is the so-called relational clone -- a set of relations $\Gamma$ containing the binary equality relation and closed under primitive positive definitions (relations defined by relations from $\Gamma$, conjunction, and existential quantifiers). If we define $RelClone(\Gamma)$ to be the smallest relational clone containing $\Gamma$, then the following theorem expresses a one-to-one correspondence between relational clones and clones.

\begin{theorem}[Galois correspondence for constraint languages]\label{hhfy7ttdiu}
\item [1.] For any finite domain $A$ and any constraint language $\Gamma$ on $A$, $Inv(Pol(\Gamma))=RelClone(\Gamma)$. 
\item [2.] For any finite domain $A$ and any set of operations $\mathit{O}$ on $A$, $Pol(Inv(\mathit{O})) = Clone(\mathit{O})$.
\end{theorem}

For an algebra $\mathbb{A}$ a subset $B\subseteq A$ is a \emph{subuniverse} if it is closed under all operations of $\mathbb{A}$. Given a subuniverse $B$ we can form the subalgebra $\mathbb{B}\leq \mathbb{A}$ by restriction of all the operations of $\mathbb{A}$ to the set $B$. Given an algebra $\mathbb{A}$ for every subset $X\subseteq A$ we denote by $Sg(X)$ the minimal subalgebra of $A$ containing $X$, i.e. the subalgebra generated by $X$. If we define a closure operator $E(X)$ to be $E(X)=X\cup\{f(a_1,...,a_n): f\text{ is a basic operation on }A,\,a_1,...,a_n\in X\}$, and $E^{t}(X)$ for $t\geq 0$ by $E^0(X)=X, E^{t+1}(X) = E(E^t(X))$,
then
$$Sg(X) = X\cup E(X)\cup E^2(X)\cup...$$

An equivalence relation $\sigma$ on $\mathbb{A}$ is a \emph{congruence} if any term operation on $\mathbb{A}$ is compatible with $\sigma$. Two trivial congruences on $\mathbb{A}$ are the diagonal relation $\Delta_{A}=\{(a,a): a\in A\}$ and full relation $\nabla_{A}=A^2$. A congruence is a \emph{maximal congruence} if it is not contained in any other congruence except $\nabla_{A}$. A congruence $\sigma$ allows one to introduce a \emph{quotient, or factor, algebra} $\mathbb{A}/\sigma$. It has as the universe the set of $\sigma$-classes and the operations are defined using arbitrary representatives from these classes. Note that the congruence $\sigma$ forms a subalgebra of $\mathbb{A}^2$: applying any term operation to elements from $\sigma$ coordinatewise, due to the compatibility property, we again get an element from $\sigma$. In general, any $n$-ary relation $R$ on $\mathbb{A}$ invariant under all term operations is a subalgebra of $\mathbb{A}^n$.

A nonempty class $K$ of algebras of the same type (same signature) is called a \emph{variety} if it is closed under subalgebras $S(K)$, homomorphic images $H(K)$, and direct products $P(K)$. It is known that the smallest variety containing $K$ is equal to $HSP(K)$. For a pair of terms $s,t$ over a signature $\Sigma$, we say that a class of algebras $K$ in the signature $\Sigma$ satisfies the identity $s\approx t$ if every algebra in the class does. For any set of identities $\Xi$ of the type $\Sigma$, define $M(\Xi)$ to be the class of algebras $K$ satisfying $\Xi$. A class $K$ of algebras is an equational class if there is a set of
identities $\Xi$ such that $K=M(\Xi)$. In this case, we say that $K$ is defined, or axiomatized, by
$\Xi$. 

\begin{theorem}[Birkhoff]
$K$ is an equational class if and only if $K$ is a variety. In other words, classes of algebras defined by identities are precisely those that are closed under $H, S$, and $P$.
\end{theorem}

\subsection{CSP basics}\label{CSPDEF}
In this section, we will give two different definitions of the Constraint satisfaction problem (CSP) and will formulate the CSP dichotomy conjecture. Some definitions, examples, and results are adapted from
\cite{barto_et_al},
\cite{10.1145/3402029},  and \cite{Zivn:2012:CVC:2412078}.

\begin{definition}[CSP over finite domains \cite{10.1145/3402029}] The \emph{Constraint satisfaction problem} is a problem of deciding whether there is an assignment to a set of variables that satisfies some specified constraints. An \emph{instance of CSP problem} over finite domains is defined as a triple $\Theta = (X,D,C)$, where
\begin{itemize}
    \item $X = \{x_0,...,x_{n-1}\}$ is a finite set of variables,
    \item $D = \{D_0,...,D_{n-1}\}$ is a set of non-empty finite domains,
    \item $C=\{C_0,...,C_{t-1}\}$ is a set of constraints, 
\end{itemize}
where each variable $x_i$ can take on values in the non-empty domain $D_i$, and every constraint $C_j \in C$ is a pair $(\vec{x}_j,\rho_j)$ with $\vec{x}_j$ being a tuple of variables of some length $m_j$, called a \emph{constraint scope}, and $\rho_j$ being an $m_j$-ary relation on the product of the corresponding domains, called a \emph{constraint relation}. The question is whether there exists a solution to $\Theta$, i.e. an assignment to every variable $x_i$ such that for each constraint $C_j$ the image of the constraint scope is a member of the constraint relation.  
\end{definition}

A \emph{constraint satisfaction problem} associated with constraint language $\Gamma$, to be denoted CSP($\Gamma$), is a subclass of CSP defined by the property that any constraint relation in any instance of CSP($\Gamma$) must belong to $\Gamma$. 

The equivalent definition of CSP can be formulated in terms of homomorphisms between relational structures.

\begin{definition}[CSP \cite{BULATOV200531}] $\,$
\begin{itemize}
\item A \emph{vocabulary} is a finite set of relational symbols $R_1$,..., $R_n$, each of which has a fixed arity.
\item A \emph{relational structure} over the vocabulary $R_1$,..., $R_n$ is a tuple $\mathcal{A} = (A, R^\mathcal{A}_1,...,$ $R^\mathcal{A}_n)$ such that $A$ is a non-empty set, called the \emph{universe} of $\mathcal{A}$, and each $R^\mathcal{A}_i$ is a relation on $A$ having the same arity as the symbol $R_i$.
\item For $\mathcal{X}$, $\mathcal{A}$, being relational structures over the same vocabulary $R_1$,..., $R_n$, a \emph{homomorphism} from $\mathcal{X}$ to $\mathcal{A}$ is a mapping $\phi: \mathcal{X} \rightarrow \mathcal{A}$ from the universe $X$ to the universe $A$ such that for every $m$-ary relation $R^\mathcal{X}$ and every tuple $(x_1,...,x_m) \in R^\mathcal{X}$ we have $(\phi(x_1),...,\phi(x_m)) \in R^\mathcal{A}$.
\end{itemize}
Let $\mathcal{A}$ be a relational structure over a vocabulary $R_1$,..., $R_n$. In the \emph{constraint satisfaction problem} associated with $\mathcal{A}$, denoted by CSP($\mathcal{A}$), the question is, given a structure $\mathcal{X}$ over the same vocabulary, whether there exists a homomorphism from $\mathcal{X}$ to $\mathcal{A}$. If the answer is positive, then we call the instance $\mathcal{X}$ \emph{satisfiable} and \emph{unsatisfiable} otherwise. We call $\mathcal{A}$ the \emph{target structure} and $\mathcal{X}$ the \emph{instance (or input) one}. 
\end{definition}

The idea of translation from the homomorphism form to the constraint form is the following: consider the domain $X$ of the structure $\mathcal{X}$ as a set of variables and every tuple $(x_1,...,x_m) \in R^\mathcal{X}$ as a constraint $C=(x_1,...,x_m; R^\mathcal{A})$. For the translation back, consider the set of variables $X$ as a domain of the instance structure, the set $A$ as a domain of the target structure, and each constraint $C=(x_1,...,x_m; R^\mathcal{A})$ as a relation $R^\mathcal{X}$ on $X$.

It was conjectured years ago by Feder and Vardi \cite{1} and recently proved by Zhuk \cite{10.1145/3402029} and Bulatov \cite{8104069} that there is a dichotomy: each CSP($\mathcal{A}$) is either NP-complete or polynomial time solvable. The dichotomy depends on the following. We call an operation $\Omega$ on a set $A$ the \emph{weak-near unanimity operation} (WNU) if it satisfies 
$
\Omega(y,x,x,...,x)=\Omega(x,y,x,...,x)=...=\Omega(x,x,...,x,y)
$
for all $x,y\in A$. Furthermore, $\Omega$ is called \emph{idempotent} if $\Omega(x,...,x)=x$ for every $x\in A$, and is called \emph{special} if for all $x,y\in A$,
$
\Omega(x,...,x,\Omega(x,...,x,y))=\Omega(x,...,x,y).$ 

\begin{lemmach}[\cite{MAROTI}]\label{SPECIALWNU}
For any idempotent WNU operation $\Omega$ on a finite set, there exists a special WNU operation $\Omega' \in Clone(\Omega)$.
\end{lemmach}

\begin{theorem}[CSP dichotomy theorem \cite{10.1145/3402029}] Suppose $\Gamma$ is a finite set of relations on a set $A$. Then CSP($\Gamma$) can be solved in polynomial time if there exists a WNU operation $\Omega$ on $A$ preserving $\Gamma$; CSP($\Gamma$) is NP-complete otherwise.
\end{theorem}

In terms of complexity, instead of $\Gamma$ it is more convenient to consider richer languages since they considerably reduce the variety of languages to be studied. For example, if we consider the language $RelClone(\Gamma)$ that contains the binary equality relation and is closed under $pp$-definitions over $\Gamma$, we do not increase the complexity of the problem since CSP($RelClone(\Gamma)$) is log-space reducible to CSP($\Gamma$). Note that due to Theorem \ref{hhfy7ttdiu} all relations $pp$-definable over $\Gamma$ are invariant under all polymorphisms preserving $\Gamma$.

Apart from $pp$-definability, there are other modifications of constraint languages that do not increase their complexity (i.e. allow log-space reduction), such as $pp$-interpretability, homomorphic equivalence, and singleton expansion of a core constraint language, see \cite{barto_et_al}. The beauty of the so-called algebraic approach to CSP is that these modifications to constraint languages represent classical algebraic constructions. Indeed, homomorphic equivalence and singleton expansion put together ensure that the algebra corresponding to the constraint language is idempotent. $Pp$-interpretations correspond to taking homomorphic images, subalgebras, and products over the algebras of polymorphisms of the constraint languages, where an algebra of polymorphisms is $Pol(\Gamma)$ with elements being polymorphisms and the operation being a superposition. 

It turns out that a constraint language $\mathcal{D}$ $pp$-interpreters a constraint language $\mathcal{E}$ if and only if in $Pol(\mathcal{E})$ there exist operations satisfying all the identities that are satisfied by operations in $Pol(\mathcal{D})$ \cite{10.1145/2677161.2677165}. Since a variety of algebras is defined by its identities, the variety of algebra corresponding to the language $\mathcal{D}$ contains the variety of algebra corresponding to the language $\mathcal{E}$. Thus, $pp$-interpretability does not change the structure or the properties of the corresponding algebras. 
 
$Pp$-constructibility combines all previous modifications.

\begin{definition}[$Pp$-constructibility \cite{barto_et_al}]
A constraint language $\mathcal{D}$ over a domain $D$ $pp$-constructs a constraint language $\mathcal{E}$ over a domain $E$ if there is a sequence of constraint languages $\mathcal{D} = \mathcal{C}_1,...,\mathcal{C}_k=\mathcal{E}$ such that for each $1\leq i\leq k$
\begin{itemize}
    \item $\mathcal{C}_i$ $pp$-interprets $\mathcal{C}_{i+1}$, or
    \item $\mathcal{C}_i$ is homomorphically equivalent to $\mathcal{C}_{i+1}$, or
    \item $\mathcal{C}_i$ is a core and  $\mathcal{C}_{i+1}$ is its singleton expansion.
\end{itemize}
\end{definition}

The last theorem in this section is very useful since it allows one to work with at most binary constraints, which often simplifies representation and analysis of CSP. For the sake of clarity, we will further restrict the discussion to constraint languages with at most binary relations. It must be stressed that all results in the paper can be extrapolated to any other finite constraint languages (with possibly more tedious representation). 

\begin{theorem}\label{aakrrrfetsplgh}
For any constraint language $\Gamma$ there is a constraint language $\Gamma'$ such that all relations in $\Gamma'$ are at most binary and $\Gamma$ and $\Gamma'$ $pp$-constructs each other.
\end{theorem}

\subsection{Characterization of a CSP instance}

This subsection introduces some properties of a CSP instance that will be used in Zhuk's algorithm \cite{10.1145/3402029} and provides their interpretations in terms of constraint languages with at most binary relations. 

We say that a variable $y_i$ of a constraint $C_j=(y_1,...,y_k;R)$ is \emph{dummy} if $R$ does not depend on its $i$-th variable. A relation $R \subseteq D_0\times ...\times D_{n-1}$ is \emph{subdirect} if for every $i$ the projection of $R$ onto the $i$-th coordinate is the whole $D_i$. A CSP instance $\Theta$ with a domain set $D$ is called \emph{$1$-consistent} (or \emph{arc consistent}) if for every constraint $C_i$ of the instance the corresponding relation $R_i \subseteq D_{i_1}\times ...\times D_{i_{k}}$ is subdirect. An arbitrary instance can be turned into $1$-consistent instance with the same set of solutions by a simple algorithm \cite{barto_et_al}.

Another type of consistency is related to the notion of a path. Let $D_y$ denote the domain of the variable $y\in \{x_1,...,x_n\}$. We say that the sequence 
$y_1-C_1-y_2 -...-y_{l-1} - C_{l-1} - y_l$ is a \emph{path} in a CSP instance if $\{y_i,y_{i+1}\}$ are in the scope of $C_i$ for every $i<l$ (we do not care in what order variables $y_i,y_{i+1}$ occur in $C_i$). We say that the path \emph{connects} $b$ and $c$ if there exists $a_i\in D_{y_i}$ for every $i$ such that $a_1 = b$, $a_l=c$ and the projection of $C_i$ onto $\{y_i,y_{i+1}\}$ contains the tuple $(a_i,a_{i+1})$. We say that a CSP instance is \emph{cycle-consistent} if it is $1$-consistent and for every variable $y$ and $a\in D_y$ \emph{any} path starting and ending with $y$ connects $a$ and $a$. A CSP instance is called \emph{linked} if for every variable $y$ occurring in the scope of a constraint $C$ and for all $a,b\in D_y$ there \emph{exists} a path starting and ending with $y$ in $\Theta$ that connects $a$ and $b$.

A fragmented CSP instance can be divided into several nontrivial instances: an instance is \emph{fragmented} if the set of variables $X$ can be divided into $2$ disjoint sets $X_1$ and $X_2$ such that each of them is non-empty, and the constraint scope of any constraint of $\Theta$ either has variables only from
$X_1$, or only from $X_2$. We call an instance $\Theta =(X,D,C)$ \emph{irreducible} if any instance $\Theta' = (X',D',C')$ such that $X'\subseteq X$, $D'_x = D_x$ for every $x\in X'$, and every constraint of $\Theta'$ is a projection of a constraint from $\Theta$ on \emph{some} subset of variables from $X'$ is fragmented, or linked, or its solution set is subdirect. 

One of the important notions of Zhuk's algorithm is a weaker constraint: by weakening some constraints we make an instance weaker (i.e. possibly having more solutions). We say that a constraint $C_1 = ((y_1,...,y_t),$ $\rho_1)$ is \emph{weaker or equivalent} to a constraint $C_2 = ((z_1,...,z_s),\rho_2)$ if $\{y_1,...,y_t\}\subseteq \{z_1,...,z_s\}$ and $C_2$ implies $C_1$, i.e the solution set to $\Theta_1=(\{z_1,...,z_s\},(D_{z_1},$ $...,D_{z_s}),C_1)$ contains the solution set to $\Theta_2=(\{z_1,...,z_s\},(D_{z_1},...,D_{z_s}),C_2)$. We say that $C_1$ is \emph{weaker} than $C_2$ (denoted $C_1\leq C_2$) if $C_1$ is weaker or equivalent to $C_2$, but $C_1$ does not imply $C_2$. There can be $2$ types of weaker constraints. We say that $C_1 = ((y_1,...,y_t),\rho_1)\leq C_2=((z_1,...,z_s),\rho_2)$ with $\{y_1,...,y_t\}\subseteq \{z_1,...,z_s\}$ if one of the following conditions holds:
\begin{enumerate}
    \item  The arity of relation $\rho_1$ is less than the arity of relation $\rho_2$ and for any tuple $(a_{z_1},...,a_{z_s})\in \rho_2$, $(a_{y_1},...,a_{y_t})\in \rho_1$.
    \item The arities of relations $\rho_1$ and $\rho_2$ are equal and $\rho_2\subsetneq \rho_1$.
\end{enumerate}

All the above-mentioned properties have simple interpretations in terms of constraint languages with at most binary relations. Generally, CSP is defined as having a single common “superdomain” $D$ for all variables. However, even though domains can be all equal at the beginning, Zhuk's algorithm will create different domains for individual variables. We require each $D_i, i\in \{0,...,n-1\}$ to be $pp$-definable over the constraint language $\Gamma$, i.e. CSP($\Gamma$) is $p$-equivalent to CSP($\Gamma, D_0,...,D_{n-1}$). Any constraint for the CSP instance is either $C=(x_i;D_i)$, where $D_i$ is a restriction on the domain for the variable $x_i$, or $C=(x_i,x_j;E^{ij}_{})$. Every unary relation can be viewed as a domain and every binary relation - as an edge, where the order corresponds to the direction. So it is natural to refer to these relational structures as some sort of digraphs and to the CSP problem as a homomorphism problem between relational structures. 

In our case, an input relational structure is a classical digraph $\mathcal{X}=(V_{\mathcal{X}}, E_{\mathcal{X}})$ with $V_{\mathcal{X}}=\{x_1,...,x_n\}$. Let us call a target relational structure a \emph{digraph with domains} $\ddot{\mathcal{A}}=(V_{\ddot{\mathcal{A}}}, E^{ij}_{\ddot{\mathcal{A}}}: 0\leq i,j<n)$, where $V_{\ddot{\mathcal{A}}}=\{D_0,...,D_{n-1}\}$. The problem is in finding a homomorphism such that it sends every $x_i$ to the domain $D_i$ and every edge $(x_i,x_j)\in E_{\mathcal{X}}$ to an edge $(a,b)\in E^{ij}_{\ddot{\mathcal{A}}}$ (relations $E^{ij}_{\ddot{\mathcal{A}}}$ can differ for all $i,j$). We will denote the corresponding instance by $\Theta=(\mathcal{X}, \ddot{\mathcal{A}})$.

In this setting, a $1$-consistent CSP instance is an instance in which for every edge $(x_i,x_j)$ from $E_{\mathcal{X}}$, for any element $a\in D_i$ there is an element $b\in D_j$ such that $(a,b)\in E^{ij}_{\ddot{\mathcal{A}}}$ and vice versa. A variable $x_i$ of an edge $(x_i,x_j)\in E_{\mathcal{X}}$ is dummy if for every $b\in D_j$ such that there exists $a\in D_i$, $E^{ij}_{\ddot{\mathcal{A}}}(a,b)$, there is an edge $(a',b)\in E^{ij}_{\ddot{\mathcal{A}}}$ for every $a'\in D_i$. Note that for a $1$-consistent CSP instance this means that $E^{ij}_{\ddot{\mathcal{A}}}$ is a full relation.

Since we work with digraphs, by \emph{undirected path or cycle} in the paper are meant any path or cycle with edges not necessarily directed in the same direction. A path $y_0-C_0-y_1 -...-y_{t-1} - C_{t-1} - y_t$ is an undirected path in digraph $\mathcal{X}$ (where some variables $y_i,y_j$ can be the same). Consider this path as a separate digraph $\mathcal{P}_t$ with new (all different) vertices $s_0-C_0-s_1 -...-s_{t-1} - C_{t-1} - s_t$, and consider a homomorphism $H$ from $\mathcal{P}_t$ to $\mathcal{X}$ such that for each $i\leq t$, $H(s_i)=y_i$. We say that path $\mathcal{P}_t$ connects elements $b\in D_{y_0}$ and $c\in D_{y_t}$ if it can be homomorphically mapped to $\ddot{\mathcal{A}}$ in such a way that for each $i\leq t$ homomorphism $H': \mathcal{P}_t \to \ddot{\mathcal{A}}$ sends $s_i$ to some $a_i\in D_{y_i}$ and $H'(s_0)=b$, $H'(s_t)=c$. An instance is linked if for any $a,b\in D_y$ there exists an undirected path that connects $a$ and $b$. Cycle-consistency in these terms means that the instance is $1$-consistent and for any $a\in D_y$ and any $y\in\{x_0,...,x_{n-1}\}$ any undirected path that is a cycle connects $a$ and $a$. In other words, an instance is cycle-consistent if any undirected cycle in $\mathcal{X}$ can be homomorphically mapped onto a cycle in $\ddot{\mathcal{A}}$ for any element $a\in D_y$ and any $y\in\{x_0,...,x_{n-1}\}$ that occurs in this cycle. 

Compare as examples two CSP instances in Figure \ref{NONLININST} and Figure \ref{LININST}. The input digraph $\mathcal{X}$ is the same for both instances, $V_{\mathcal{X}}=\{x_0,x_1,x_2\}$, $E_\mathcal{X} = \{(x_0,x_1),(x_2,x_1),$ $(x_2,x_0)\}$. The first CSP instance has three constraint relations, $E^{01}_{\ddot{\mathcal{A}}} = \{(a,a),(b,c)\})$, $E^{21}_{\ddot{\mathcal{A}}} = \{(d,a),(b,c)\})$ and $E^{20}_{\ddot{\mathcal{A}}} = \{(d,a),(b,b)\}$. This instance is cycle-consistent since it is $1$-consistent (each constraint of the instance is subdirect) and for every variable $x$ and $e\in D_x$ any path starting and ending with $x$ connects $e$ and $e$. But it is not linked since, for example, there is no path connecting $a$ and $b$ in $D_0$. However, if we add one more edge $(d,c)$ to $E^{21}_{\ddot{\mathcal{A}}}$, the new instance will be linked. On contrary, the second instance in Figure \ref{LININST} is linked, but not cycle-consistent.

\begin{figure}
\begin{center}
{
\begin{tikzpicture}[scale=1, every label/.style={scale=0.7}]

\draw[step=0.5cm,gray,very thin,opacity=0] [dashed] (0.5,0) grid (5,5);

%ellipse 
\draw (1, 4) ellipse (0.5 and 1);
%\draw (2.5, 4) ellipse (0.5 and 1);
\draw (4, 4) ellipse (0.5 and 1);
\draw (2.5, 1.5) ellipse (0.5 and 1);

\fill[fill=blue!20!white, opacity=0] (1, 4) ellipse (0.5 and 1);
\fill[fill=blue!20!white, opacity=0] (4, 4) ellipse (0.5 and 1);
\fill[fill=blue!20!white, opacity=0] ((2.5, 1.5) ellipse (0.5 and 1);

\coordinate [label=${x_0\in D_0}$] () at (1,2.6);
\coordinate [label=${x_1\in D_1}$] () at (4,2.6);
\coordinate [label=${x_2\in D_2}$] () at (2.5,0.1);

%-----------------------------------------------

%coordinates of the ellipses

\fill[black] (1,4.5) circle (0.3mm);
\coordinate   [label=left:${a}$] (a) at (1,4.5);
\fill[black] (1,3.5) circle (0.3mm);
\coordinate   [label=left:${b}$] (b) at (1,3.5);

\fill[black] (4,4.5) circle (0.3mm);
\coordinate   [label=right:${a}$] (a) at (4,4.5);
\fill[black] (4,3.5) circle (0.3mm);
\coordinate   [label=right:${c}$] (c) at (4,3.5);

\fill[black] (2.5,2) circle (0.3mm);
\coordinate   [label=left:${d}$] (d) at (2.5,2);
\fill[black] (2.5,1) circle (0.3mm);
\coordinate   [label=left:${b}$] (b) at (2.5,1);

\draw [thin,->][black] (1,4.5) -- (4,4.5);
\draw [thin,->][black]  (2.5,2) -- (4,4.5);
\draw [thin,->][black]  (2.5,2) -- (1,4.5);

\draw [thin,->][black]  (1,3.5) -- (4,3.5);
\draw [thin,->][black]  (2.5,1) -- (4,3.5);
\draw [thin,->][black]  (2.5,1) -- (1,3.5);

\draw [thin,->][gray][dashed] (2.5,2) -- (4,3.5);

%\draw[thin,->] [green][dashed] (2.5,1) to (1,4);

\coordinate [label=${(x_0,x_1)\in E_\mathcal{X}}$] () at (5.5,4.5);
\coordinate [label=${(x_2,x_1)\in E_\mathcal{X}}$] () at (5.5,4);
\coordinate [label=${(x_2,x_0)\in E_\mathcal{X}}$] () at (5.5,3.5);

\end{tikzpicture}
}
\end{center}
    \caption{Cycle-consistent, non-linked instance.} \label{NONLININST}
  \end{figure}
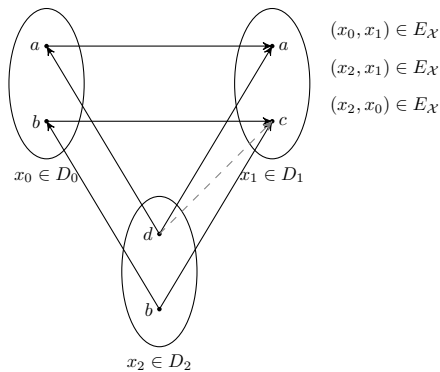

A fragmented instance in terms of digraphs and digraphs with domains is such an instance where $\mathcal{X}$ is a disconnected digraph. Finally, if an instance is not irreducible, then there exists a subgraph $\mathcal{X'}$ (a digraph formed from subsets of vertices $V_{\mathcal{X'}}\subseteq V_{\mathcal{X}}$ and edges $E_{\mathcal{X'}}\subseteq E_{\mathcal{X}}$) such that the resulting instance $\Theta=(\mathcal{X'}, \ddot{\mathcal{A}})$ is not fragmented, is not linked, and its solution set is not subdirect.

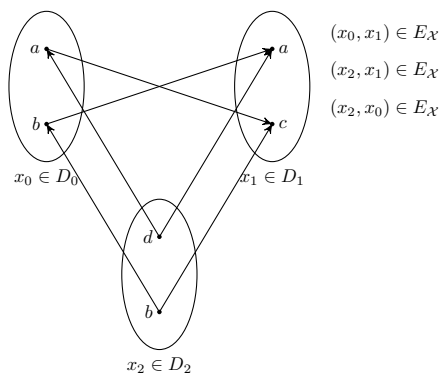
\begin{figure}[H]
\begin{center}
{
\begin{tikzpicture}[scale=1, every label/.style={scale=0.7}]

\draw[step=0.5cm,gray,very thin,opacity=0] [dashed] (0.5,0) grid (5,5);

%ellipse 
\draw (1, 4) ellipse (0.5 and 1);
%\draw (2.5, 4) ellipse (0.5 and 1);
\draw (4, 4) ellipse (0.5 and 1);
\draw (2.5, 1.5) ellipse (0.5 and 1);

\fill[fill=blue!20!white, opacity=0] (1, 4) ellipse (0.5 and 1);
\fill[fill=blue!20!white, opacity=0] (4, 4) ellipse (0.5 and 1);
\fill[fill=blue!20!white, opacity=0] ((2.5, 1.5) ellipse (0.5 and 1);

\coordinate [label=${x_0\in D_0}$] () at (1,2.6);
\coordinate [label=${x_1\in D_1}$] () at (4,2.6);
\coordinate [label=${x_2\in D_2}$] () at (2.5,0.1);

%-----------------------------------------------

%coordinates of the ellipses

\fill[black] (1,4.5) circle (0.3mm);
\coordinate   [label=left:${a}$] (a) at (1,4.5);
\fill[black] (1,3.5) circle (0.3mm);
\coordinate   [label=left:${b}$] (b) at (1,3.5);

\fill[black] (4,4.5) circle (0.3mm);
\coordinate   [label=right:${a}$] (a) at (4,4.5);
\fill[black] (4,3.5) circle (0.3mm);
\coordinate   [label=right:${c}$] (c) at (4,3.5);

\fill[black] (2.5,2) circle (0.3mm);
\coordinate   [label=left:${d}$] (d) at (2.5,2);
\fill[black] (2.5,1) circle (0.3mm);
\coordinate   [label=left:${b}$] (b) at (2.5,1);

\draw [thin,->][black]  (1,4.5) -- (4,3.5);
\draw [thin,->][black]  (2.5,2) -- (4,4.5);
\draw [thin,->][black]  (2.5,2) -- (1,4.5);

\draw [thin,->][black]  (1,3.5) -- (4,4.5);
\draw [thin,->][black]  (2.5,1) -- (4,3.5);
\draw [thin,->][black]  (2.5,1) -- (1,3.5);

%\draw[thin,->] [green][dashed] (2.5,1) to (1,4);

\coordinate [label=${(x_0,x_1)\in E_\mathcal{X}}$] () at (5.5,4.5);
\coordinate [label=${(x_2,x_1)\in E_\mathcal{X}}$] () at (5.5,4);
\coordinate [label=${(x_2,x_0)\in E_\mathcal{X}}$] () at (5.5,3.5);

\end{tikzpicture}
}
\end{center}
    \caption{Linked, not cycle-consistent instance.} \label{LININST}
  \end{figure}

Since there are two types of weaker constraints (of less arity or of richer relation of the same arity), we can weaken the CSP instance $\Theta=(\mathcal{X}, \ddot{\mathcal{A}})$ either by removing an edge $(x_i,x_j)\in E_{\mathcal{X}}$ from $\mathcal{X}$ (i.e. by reducing the arity of a constraint) or by adding edges to a relation $E^{ij}_{\ddot{\mathcal{A}}}$ (i.e. by making a richer relation of the same arity). The algorithm never increases the domains.

We conclude this subsection with Lemma \ref{NONLINKEDLEMMA} to be used further for the formalization of Zhuk's algorithm. For an instance $\Theta$ and its variable $x$ let $Linked(\Theta,x)$ denote the binary relation on $D_x$ defined as follows: $(a,b)\in Linked(\Theta,x)$ if there exists a path in $\Theta$ that connects $a$ and $b$.
\begin{lemmach}[\cite{10.1145/3402029}]\label{NONLINKEDLEMMA}
Suppose $\Theta$ is a cycle-consistent CSP instance such that every its variable $x\in X$ actually occurs in some constraint of $\Theta$. Then for every $x\in X$ there exists a path in $\Theta$ connecting all pairs $(a,b)\in Linked(\Theta,x)$ and $Linked(\Theta,x)$ is a congruence.
\end{lemmach}

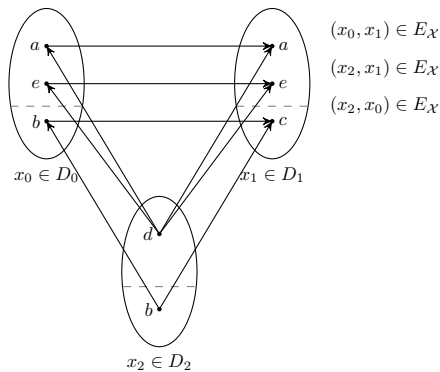
\begin{figure}
\begin{center}
{
\begin{tikzpicture}[scale=1, every label/.style={scale=0.7}]

\draw[step=0.5cm,gray,very thin,opacity=0] [dashed] (0.5,0) grid (5,5);

%ellipse 
\draw (1, 4) ellipse (0.5 and 1);
%\draw (2.5, 4) ellipse (0.5 and 1);
\draw (4, 4) ellipse (0.5 and 1);
\draw (2.5, 1.5) ellipse (0.5 and 1);

\fill[fill=blue!20!white, opacity=0] (1, 4) ellipse (0.5 and 1);
\fill[fill=blue!20!white, opacity=0] (4, 4) ellipse (0.5 and 1);
\fill[fill=blue!20!white, opacity=0] ((2.5, 1.5) ellipse (0.5 and 1);

\coordinate [label=${x_0\in D_0}$] () at (1,2.6);
\coordinate [label=${x_1\in D_1}$] () at (4,2.6);
\coordinate [label=${x_2\in D_2}$] () at (2.5,0.1);

%-----------------------------------------------

%coordinates of the ellipses

\fill[black] (1,4.5) circle (0.3mm);
\coordinate   [label=left:${a}$] (a) at (1,4.5);
\fill[black] (1,3.5) circle (0.3mm);
\coordinate   [label=left:${b}$] (b) at (1,3.5);

\fill[black] (1,4) circle (0.3mm);
\coordinate   [label=left:${e}$] (b) at (1,4);

\fill[black] (4,4.5) circle (0.3mm);
\coordinate   [label=right:${a}$] (a) at (4,4.5);
\fill[black] (4,3.5) circle (0.3mm);
\coordinate   [label=right:${c}$] (c) at (4,3.5);

\fill[black] (4,4) circle (0.3mm);
\coordinate   [label=right:${e}$] (c) at (4,4);

\fill[black] (2.5,2) circle (0.3mm);
\coordinate   [label=left:${d}$] (d) at (2.5,2);
\fill[black] (2.5,1) circle (0.3mm);
\coordinate   [label=left:${b}$] (b) at (2.5,1);

\draw [thin][gray][dashed] (0.52,3.7) -- (1.48,3.7);
\draw [thin][gray][dashed] (3.52,3.7) -- (4.48,3.7);
\draw [thin][gray][dashed](2.02,1.3) -- (2.98,1.3);

\draw [thin,->][black]  (1,4.5) -- (4,4.5);
\draw [thin,->][black]  (1,4) -- (4,4);
\draw [thin,->][black]  (2.5,2) -- (4,4.5);
\draw [thin,->][black]  (2.5,2) -- (1,4.5);

\draw [thin,->][black]  (1,3.5) -- (4,3.5);
\draw [thin,->][black]  (2.5,1) -- (4,3.5);
\draw [thin,->][black]  (2.5,1) -- (1,3.5);
\draw [thin,->][black]  (2.5,2) -- (4,4);
\draw [thin,->][black]  (2.5,2) -- (1,4);

%\draw [thin,->][red][dashed] (2.5,2) -- (4,3.5);

%\draw[thin,->] [green][dashed] (2.5,1) to (1,4);

\coordinate [label=${(x_0,x_1)\in E_\mathcal{X}}$] () at (5.5,4.5);
\coordinate [label=${(x_2,x_1)\in E_\mathcal{X}}$] () at (5.5,4);
\coordinate [label=${(x_2,x_0)\in E_\mathcal{X}}$] () at (5.5,3.5);

\end{tikzpicture}
}
\end{center}
    \caption{Division into linked components.} \label{NONLININSTMODIF}
  \end{figure}

For example, consider cycle-consistent non-linked instance $\Theta$ in Figure \ref{NONLININSTMODIF}. Binary relation $Linked(\Theta,x)$ divides each domain into two classes: $D_0$ into $\{a,e\}$ and $\{b\}$, $D_1$ into $\{a,e\}$ and $\{c\}$, and $D_2$ into $\{d\}$ and $\{b\}$.

\subsection{The theory $V^1$}\label{BAandT}

In this subsection most definitions and results are adapted from \cite{10.5555/1734064}, \cite{krajicek_1995}, \cite{krajicek2019proof}.

\emph{Second-order} (or \emph{two-sorted} first-order) theories of bounded arithmetic 
use the following set-up. The variables are of two kinds: variables $x,y,H,...$ of the first kind are called \emph{number variables} and range over the natural numbers, and variables $X,Y,H,...$ of the second kind are called \emph{set variables} and range over finite subsets of natural numbers (which can be represented as binary strings). Functions and predicate symbols can use both number and set variables, and there are \emph{number-valued} functions and \emph{set-valued} functions. Also, there are two types of quantifiers: quantifiers over number variables are called \emph{number quantifiers}, and quantifiers over set variables are called \emph{string quantifiers}. The language for the second-order theory of bounded arithmetic is an extension of the standard language for Peano Arithmetic $\mathcal{L}\mathcal{_{PA}}$, 
$$\mathcal{L}^2\mathcal{_{PA}} = \{0,1,+,\cdot,\lvert\,    \rvert,=_1,=_2,\leq, \in\}.$$ The symbols $0,1,+,\cdot,=_1$ and $\leq$ are function and predicate symbols over the number variables. The function $\lvert X\rvert$ (called the \emph{length of $X$}) is a number-valued function and it denotes the length of the corresponding string $X$ (i.e. the upper bound for the set $X$). The binary predicate $\in$ for a number and a set variables denotes set membership, and $=_2$ is the equality predicate for sets. 

\begin{notation}
We will use the abbreviation $X(t) =_{def} t \in X$, where $t$ is a number term. We thus think of $X(i)$ as of the $i$-th bit of binary string $X$ of length $\lvert X\rvert$.
\end{notation}

There is a set of axioms $2$-BASIC \cite{10.5555/1734064} that defines basic properties of symbols from $\mathcal{L}^2\mathcal{_{PA}}$. Here we present only axioms of the second sort: 

\begin{definition}[$2$-BASIC, second-sort axioms]
The set $2$-BASIC for the second-sort variables contains the following axioms: 
\begin{enumerate}
\item $X(y)\to y<|X|$.
\item $y+1 =_1 |X|\to X(y)$.
\item $(|X|=_1|Y|\wedge \forall i<|X|(X(i)\leftrightarrow X(i)))\iff X=_2Y$.
\end{enumerate}
We will skip the indices $=_1,=_2$ as there is no danger of confusion.
\end{definition}

\begin{notation}
Sometimes for a set $A$, an element $x$ and a formula $\phi$ instead of $\exists x<|A|\,A(x)\wedge \phi$ and $\forall x<|A|\,A(x)\rightarrow \phi$ we will write $\exists x\in A\, \phi$ and $\forall x\in A\, \phi$.
\end{notation}

\begin{definition}[Bounded formulas]
Let $\mathcal{L}$ be the two-sorted vocabulary. If $x$ is a number variable, $X$ is a string variable that do not occur in an $\mathcal{L}$-number term $t$, then $\exists x \leq t \phi$ stands for $\exists x( x \leq t \wedge \phi)$, $\forall x \leq t \phi$ stands for $\forall x(x \leq t \to \phi)$, $\exists X \leq t\phi$ stands for $\exists X(|X|\leq t \wedge \phi)$ and $\forall X\leq t \phi$ stands for $\forall X(|X|\leq t \to \phi)$. Quantifiers that occur in this form are said to be \emph{bounded}, and a \emph{bounded formula} is one in which every quantifier is bounded.
\end{definition}

\begin{definition}[Number Induction axioms]
If $\Phi$ is a set of two-sorted formulas, then $\Phi$-IND axioms are the formulas
\begin{equation}
  \phi(0) \wedge \forall x (\phi(x)\to \phi(x+1))\to \forall z \phi(z),
\end{equation}
where $\phi$ is any formula in $\Phi$. The formula $\phi(x)$ may have other free variables than $x$ of both sorts.
\end{definition}

\begin{definition}[Number Minimization and Maximization axioms]
The number minimization axioms (or the least number principle axioms) for a set $\Phi$ of formulas are denoted by $\Phi$-MIN and consist of the formulas
\begin{equation}
    \phi(y)\to \exists x\leq y (\phi(x)\wedge \neg \exists H<x \,\phi(z)),
\end{equation}
where $\phi$ is a formula in $\Phi$. Similarly, the number maximization axioms for $\Phi$ are denoted by $\Phi$-MAX and consist of the formulas 
\begin{equation}
    \phi(0)\to \exists x\leq y (\phi(x)\wedge \neg \exists H\leq y(x<z\wedge \phi(z))),
\end{equation}
where $\phi$ is a formula in $\Phi$. In the above definitions, $\phi$ is permitted to have free variables of both sorts, in addition to $x$.
\end{definition}

\begin{definition}[Comprehension axioms] If $\Phi$ is a set of two-sorted formulas, then $\Phi$-COMP is the set of all formulas
\begin{equation}
    \forall x\exists X\leq x\,\forall y< x \, y\in X \equiv \phi(y),
\end{equation}
where $\phi$ is any formula in $\Phi$, and $X$ does not occur free in $\phi(y)$. The formula $\phi(y)$ may have other free variables than $y$ of both sorts.
\end{definition}

Finally, we can define the theory $V^1$, which is the key theory for our work.

\begin{definition}[The theory $V^1$]
$\Sigma^{1,b}_0=\Pi^{1,b}_0$-formulas are formulas with all number quantifiers bounded and with no set-sort quantifiers. Classes $\Sigma^{1,b}_1$ and $\Pi^{1,b}_1$ are the smallest classes of $\mathcal{L}^2\mathcal{_{PA}}$-formulas such that: 
\begin{enumerate}
    \item $\Sigma^{1,b}_0\cup \Pi^{1,b}_0\subseteq \Sigma^{1,b}_1\cap \Pi^{1,b}_1$,
    \item both $\Sigma^{1,b}_1$ and $\Pi^{1,b}_1$ are closed under $\vee$ and $\wedge$,
    \item the negation of a formula $\Sigma^{1,b}_1$ is in $\Pi^{1,b}_1$ and vice versa,
    \item if $\phi \in \Sigma^{1,b}_1$, then also $\exists X\leq t \,\phi \in \Sigma^{1,b}_1$,
    \item if $\phi \in \Pi^{1,b}_1$, then also $\forall X\leq t \,\phi \in \Pi^{1,b}_1$.
\end{enumerate}
The theory $I\Sigma^{1,b}_0$ is a second-order theory and it is axiomatized by $2$-BASIC and the IND scheme for all $\Sigma^{1,b}_0$-formulas. The teory $V^0$ expands $I\Sigma^{1,b}_0$ by having also bounded comprehension axioms $\Sigma^{1,b}_0$-CA. The theory $V^0$ is a conservative extension of $I\Sigma^{1,b}_0$ with respect to $\Sigma^{1,b}_0$-consequences: if $\gamma$ is a $\Sigma^{1,b}_0$-formula and $V^0$ proves its universal closure, so does $I\Sigma^{1,b}_0$. Finally, the theory $V^1$ extends $V^0$ by accepting the IND scheme for all $\Sigma^{1,b}_1$-formulas.
\end{definition}

\subsection{Auxiliary functions, relations, and axioms in $V^1$}\label{jjjdkyfopsy}
In this subsection we will present some general auxiliary functions and relations, which help to express the bounds of the theory $V^1$.

For any two sets $A, B$, we say that a set $B$ is a \emph{subset} of $A$ if
\begin{equation}
 \begin{split}
&SS(B,A) \iff |A|=|B|\wedge \forall i<|B|\, (B(i)\rightarrow A(i)).
 \end{split}
\end{equation}
We say that a set $B$ is a \emph{proper subset} of $A$ if
\begin{equation}
 \begin{split}
&PSS(B,A) \iff |A|=|B|\wedge \forall i<|B|\, (B(i)\rightarrow A(i))\wedge \\
&\hspace{30pt}\exists j<|A|,  B(j)\wedge \exists i<|A|,\, A(i)\wedge \neg B(i).
 \end{split}
\end{equation}
If $x,y \in \mathbb{N}$, we define the \emph{pairing function} $\langle x,y\rangle$ to be the following term
\begin{equation}
\langle x,y\rangle = \frac{(x + y)(x + y + 1)}{2} + y.
\end{equation}
One can easily prove in $V^0$ that for the pairing function the following is true:

\begin{itemize}
 \item $\forall x_1,x_2,y_1,y_2\,\, (\langle x_1,y_1\rangle = \langle x_2,y_2\rangle \to x_1=x_2 \wedge y_1=y_2)$,
\item $\forall z \exists x,y\,\, (\langle x,y\rangle = z)$,
\item $\forall x,y\,\,$  $(x,y \leq \langle x,y\rangle < (x+y+1)^2)$.
\end{itemize}
We can iterate the pairing function to code triples, quadruples, and so forth for any $k$, inductively setting
\begin{equation}
\langle x_1,x_2,...,x_k\rangle = \langle ...\langle \langle x_1, x_2\rangle, x_3 \rangle ,..., x_k \rangle,
\end{equation}
where 
\begin{itemize}
\item $\forall x_1,x_2,...,x_k$  $x_1,x_2,...,x_k \leq \langle x_1,x_2,...,x_k\rangle < (x_1+x_2+...+x_k+1)^{2^k}$.
\end{itemize}
We refer to the term $\langle x_1,x_2,...,x_k\rangle$ as the \emph{tupling function}.

\begin{notation}
For any set $H$, $m\geq 2$: $H(x_1,...,x_m) =_{def} H(\langle x_1,...,x_m\rangle)$.
\end{notation}
We will use the tupling function to code a function as a set. We can then express that $H$ is a function from sets $X_1,...,X_n$ to a set $Y$ by stating
$$
\forall x_1\in X_1,...,\forall x_n\in X_n
\exists !y\in Y \, H(x_1,...,x_n,y).
$$
We will abbreviate it as $Z: X_1,...,X_n \rightarrow Y$ and $H(x_1,...,x_n) = y$. Using the pairing function (or encoding of $k$-tuples), with finite sets we can also code binary (or $k$-ary) relations. Finite functions can be represented by their digraphs. For example, to represent an $m \times n$ matrix $A$ with natural number entries we think of it as of a function from $[m] \times [n]$ into $N$. The matrix is thus encoded by the set $A(i,j,a)$, and we write $A_{ij}=a$ for the corresponding entry. 

We say that a set $H$ is a \emph{well-defined map} between two sets $A$, $|A|=n$ and $B, |B|=m$ if it satisfies the relation
\begin{equation}
 \begin{split}
&MAP(A,n,B,m,H) \iff \forall i\in A \exists j\in B\wedge H(i) =j\wedge \\& \hspace{5pt}
\forall i\in A \,\forall j_1,j_2 \in B\, (H(i)=j_1\wedge H(i)=j_2 \to j_1=j_2).
 \end{split}
\end{equation}

The counting axiom allows one to count the number of elements in a set. Given a set $X$, the \emph{census function} $\verb|#|X(n)$ for $X$ is a number function defined for $n\leq |X|$ such that $\verb|#|X(n)$ is the number of $x<n$, $x\in X$. Thus, $\texttt{\#}X(|X|)$ is the number of elements in $X$. The following relation says that $\verb|#|X$ is the census function for $X$:
\begin{equation}
    \begin{split}
        &\hspace{10pt}Census(X,\texttt{\#}X)\iff \texttt{\#}X\leq \langle |X|,|X|\rangle\wedge \texttt{\#}X(0)=0\wedge \forall x<|X|\\
        &(x\in X \rightarrow \texttt{\#}X(x+1)=\texttt{\#}X(x)+1\wedge x\notin X \rightarrow \texttt{\#}X(x+1)=\texttt{\#}X(x)).
    \end{split}
\end{equation}

\begin{lemmach}
For any set $X$, $V^1$ proves that there exists its census function.
\end{lemmach}
\begin{proof}
Given any set $X$, consider $\Sigma_1^{1,b}$-induction on $n\leq |X|$ for the formula
\begin{equation}
    \begin{split}
&\hspace{45pt}\phi(n) = \exists H \leq \langle n,n\rangle\, H(0)=0\wedge \forall \,0\leq x< n\\
&\hspace{0pt}(x\in X \rightarrow H(x+1)=H(x)+1\wedge x\notin X \rightarrow H(x+1)=H(x)).
    \end{split}
\end{equation}
\end{proof}

We will now remind the reader a few well-known number-theoretic functions and relations, mainly to fix the notation. They are all definable in a weak subtheory of $V^1$ and the stated properties are proved in \cite{10.5555/1734064},\cite{krajicek_1995}. The relation of \emph{divisibility} can be defined by the formula 
\begin{equation}
x|y \iff \exists H\leq y(xz=y).
\end{equation}
We say that $p$ is a \emph{prime number} if it satisfies the relation 
\begin{equation}
PRIME(p) \iff 1<p\wedge \forall y<p\forall z<p \,(yz\neq p).
\end{equation}
It is easily seen that $V^1$ proves that any $x>0$ is uniquely representable by a product of powers of primes. The \emph{limited subtraction} $a \dot - b = max\{0,a-b\}$  can be defined by
\begin{equation}
c=a \dot - b \longleftrightarrow ((b+c=a)\vee(a\leq b\wedge c=0)),
\end{equation} 
and the \emph{division} $\floor{a/b}$ for $b\neq 0$ can be defined as follows:
\begin{equation}
c=\floor{a/b} \longleftrightarrow  (bc\leq a\wedge a<b(c+1)).
\end{equation} 
Finally, the \emph{remainder} of $a$ after being divided by $p$ can be defined by the formula
\begin{equation}
    a \,mod \,p = a \dot - (p\cdot\floor{a/p}).
\end{equation}
We say that two numbers are \emph{congruent modulo $p$}, denoted $c_1 \equiv c_2 (mod \,p)$ if $c_1\, mod\, p = c_2\,mod\,p$. It means that if $c_1<c_2$, then
\begin{equation}
    \begin{split}
        &c_1 \dot -(p\cdot\floor{c_1/p}) = c_2 \dot -(p\cdot\floor{c_2/p})\\
        &\hspace{10pt}c_2\dot - c_1 = p(\floor{c_2/p} \dot - \floor{c_1/p}),     \end{split}
\end{equation}
i.e. the difference $c_2-c_1$ is divisible by $p$. Note that it is straightforward to show in $V^1$ that for all $x_1\equiv x_2(mod \,p)$ and $y_1 \equiv y_2(mod \,p)$,
\begin{equation}
    \begin{split}
        &(x_1+y_1) \equiv (x_2+y_2)(mod \,p)\\
        &\hspace{10pt}(x_1y_1) \equiv (x_2y_2)(mod \,p).
        \end{split}
\end{equation}

\section{Zhuk's four cases}\label{LINEAR}
One of the two main ideas of Zhuk's algorithm is based on strong subalgebras. In this section we will give the definitions of absorbing subuniverse, center and central subuniverse, and polynomially complete algebra and briefly mention their main properties. Further, we consider the notion of linear algebras as introduced in \cite{10.1145/3402029} and give two elementary examples of relational structures corresponding to linear algebras. Finally, we will formulate Zhuk's four-cases theorem. 

\subsection{Absorption, center and polynomial complete algebras}

If $\mathbb{B}=(B,F_B)$ is a subalgebra of $\mathbb{A}=(A,F_A)$, then $B$ \emph{absorbs} $\mathbb{A}$ if there exists an $n$-ary term operation $f\in Clone(F_A)$ such that $f(a_1,...,a_n)\in B$ whenever the set of indices $\{i: a_i\notin B\}$ has at most one element. $B$ \emph{binary absorbs} $A$ if there exists a binary term operation $f \in Clone(F_A)$ such that $f(a,b)\in B$ and $f(b,a)\in B$ for any $a\in A$ and $b\in B$.

If $\mathbb{A}=(A,\Omega_A)$ is a finite algebra with a special WNU operation, then $C\subseteq A$ is a \emph{center} if there exists an algebra $\mathbb{B}=(B,\Omega_B)$ with a special WNU operation of the same arity and a subdirect subalgebra $\mathbb{D}=(D,\Omega_D)$ of $\mathbb{A}\times\mathbb{B}$ such that there is no nontrivial binary absorbing subuniverse in $\mathbb{B}$ and $C=\{a\in A|\forall b\in B: (a,b)\in D\}$. Every center is a ternary absorbing subuniverse. A weaker notion, suggested by Zhuk in \cite{DBLP:journals/mvl/Zhuk21}, is a central subuniverse. A subuniverse $C$ of $\mathbb{A}$ is called \emph{central} if it is an absorbing subuniverse and for every $a\in A\backslash C$ we have $(a,a)\notin Sg(\{a\}\times C\cup C\times \{a\})$. A central subuniverse has all the good properties of a center and can be used in Zhuk's algorithm instead of the center. Both algorithms, with the center or central universe, will correctly answer whether an instance has a solution, or not.

For any set $A$ denote by $O_n(A)$ the set of all $n$-ary operations on $A$. The clone of all operations on $A$ is denoted by $O(A)=\{O_n(A)|n\geq 0\}$. An $n$-ary operation $f$ on algebra $\mathbb{A}=(A,F_A)$ is called \emph{polynomial} if there exist some $(n+t)$-ary operation $g\in Clone(F_A)$ and constants $a_1,...,a_t\in A$ such that for all $x_1,...,x_n\in A$, 
$f(x_1,...,x_n)=g(x_1,...,x_n,a_1,...,$ $a_m)$. 
Denote the clone generated by $F_A$ and all the constants on $A$ (i.e. the set of all polynomial operations on $\mathbb{A}$) by $Polynom(\mathbb{A})$. We call an algebra $\mathbb{A}=(A, F_A)$ \emph{polynomially complete} (PC) if its polynomial clone is the clone of all operations on $A$, $O(A)$. In simple words, a universal algebra $\mathbb{A}$ is polynomially complete if every function on $A$ with values in $A$ is a polynomial function. A classical result about polynomial completeness is based on the following notion. The \emph{ternary discriminator function} is the function $t$ defined by the identities
\[   
t(x,y,z) = 
     \begin{cases}
        &z, \,x=y, \\ 
        &x, \,x\neq y.
     \end{cases}
\]
Then Theorem \ref{slsl88yhdurh} gives a necessary and sufficient condition of polynomial completeness. 
\begin{theorem}[\cite{https://doi.org/10.48550/arxiv.2210.07383}]\label{slsl88yhdurh}
A finite algebra is polynomially complete if and only if it has the ternary discriminator as a polynomial operation.
\end{theorem}

\subsection{Linear algebras: properties and examples on digraphs}

\begin{definition}[Linear algebra, \cite{10.1145/3402029}]
An idempotent finite algebra $\mathbb{A}=(A,\Omega)$, where $\Omega$ is an $m$-ary idempotent special WNU operation, is called \emph{linear} if it is isomorphic to $(\mathbb{Z}_{p_1}\times ...\times \mathbb{Z}_{p_s},x_1+...+x_m)$ for prime (not necessarily distinct) numbers $p_1,...,p_s$. For every finite idempotent algebra, there exists the smallest congruence (not necessarily proper), called the \emph{minimal linear congruence}, such that the factor algebra is linear. 
\end{definition}

To understand how linear algebras appear in Zhuk's algorithm, and to establish some of their properties, we consider the notion of an affine algebra. An algebra $\mathbb{A}=(A,F)$ is called \emph{affine} if there is an abelian group $\mathbb{A'}=(A,0,-,+)$ such that the relation $R = \{(x,y,H,u):(x+y=z+u)\}$ is preserved by all operations of $\mathbb{A}$ \cite{FREEZE}. Affine algebra is polynomially equivalent (has the same polynomial clone) to a module. It means that each term operation of algebra $\mathbb{A}$ is affine with respect to the abelian group $\mathbb{A'}$, i.e. to say, for any given $n$-ary operation $f\in F$ there are endomorphisms $\alpha_1,...,\alpha_n$ of $\mathbb{A}$ and an element $a \in \ A$ such that $f$ can be expressed identically as in \cite{FREEZE}:
$$
f(x_1,...,x_n) = \sum_{i=1}^{n} \alpha_i(x_i) +a. 
$$
The following lemma establishes one important property of an affine algebra in case there is an idempotent WNU operation on $\mathbb{A}$. We will provide its proof as in \cite{10.5555/2634438.2635492}, to make some notes further. 

\begin{lemmach}[\cite{10.5555/2634438.2635492}]\label{L1}
Suppose $\mathbb{A'}=(A,0,-,+)$ is a finite abelian group, the relation $R\subseteq A^4$ is defined by $R = \{(x,y,H,u):(x+y=z+u)\}$, $R$ is preserved by an idempotent WNU $m$-ary operation $\Omega$. Then $\,\Omega(x_1,...x_m)=tx_1+...+tx_m\,$ for some $t\in \mathbb{N}$.
\end{lemmach}

\begin{proof}
Define $h(x) = \Omega(0,0,...,0,x)$. We will prove the equation
$$
\Omega(x_1,...,x_i,0,...,0) = h(x_1)+...+h(x_i)
$$
by induction on $i$. For $m=1$ it follows from the definition and properties of WNU. We know that

\[ \Omega\left( \begin{array}{cccccccc}
x_1 & x_2 & ... & x_i & x_{i+1} & 0 & ... & 0 \\ 
0 & 0 & ... & 0 & 0 & 0 & ... & 0  \\
x_1 & x_2 & ... & x_i & 0 & 0 & ... & 0  \\
0 & 0 & ... & 0 & x_{i+1} & 0 & ... & 0 
\end{array} \right) \in R\]
is in $R$, which by the inductive assumption gives 
\begin{equation}
 \begin{split}
&\Omega(x_1,...,x_i,x_{i+1},0,...,0)= \Omega(x_1,...,x_i,0,0,...,0)+h(x_{i+1}) = \\
&\hspace{50pt} = h(x_{1})+ ... + h(x_{i})+h(x_{i+1}). 
 \end{split}
\end{equation}
We thus know that $\Omega(x_1,...,x_m) = h(x_1)+...+h(x_m)$. Let $p$ be the maximal order of an element in group $\mathbb{A'}=(A,0,-,+)$. Then for any element $a$ in $A$, the order of $a$ divides $p$, and in particular $pa =0$. For every $a\in A$ we have $\underbrace{h(a)+h(a)+...+h(a)}_{m} = \Omega(a,a,...,a)=a$. Thus, for any element $a\neq 0$, $m\cdot h(a)\neq 0$, hence $m$ does not divide an order of any element in $\mathbb{A'}$ and therefore $m$ and $p$ are coprime. 
Hence $m$ has the multiplicative inverse modulo $p$ and there is some integer $t$ such that $tm = 1$, $m\cdot h(x) = h(x)/t=x$, and $h(x)=tx$ for every $x$. 
\end{proof}
If we additionally assume that $\Omega$ is special (by Lemma \ref{SPECIALWNU}), then $t=1$: 
\begin{equation}
    \begin{split}
       &\Omega(x,...,x,\Omega(x,...,x,y))=\Omega(x,...,x,y),\\
        &\underbrace{tx+...+tx}_{m-1}+t\Omega(x,...,x,y) = \underbrace{tx+...+tx}_{m-1}+ty, \\
       &t(\underbrace{tx+...+tx}_{m-1}+ty) = ty,\\
       &\underbrace{tx+...+tx}_{m-1}+ty + tx = y+tx\\       &x+ty = y+tx \implies t=1.
    \end{split}
\end{equation}

Consider any finite affine algebra $\mathbb{A}$. Due to the well-known Classification theorem \cite{fuchs2014abelian} every finite abelian group is isomorphic to a product of cyclic groups whose orders are all prime powers. Thus $\mathbb{A} = \mathbb{Z}_{{p_1}^{r_1}}\times ... \times \mathbb{Z}_{{p_s}^{r_s}}$ for some not necessarily distinct primes $p_1,...,p_s$. If $p$ is the maximal order of an element in $\mathbb{A}$, then, by the above proof, $m=1(mod \,p)$. Therefore, since every $p_i$ has to divide $p$, every $p_i$ also divides $(m-1)$. If there is an idempotent WNU operation on $\mathbb{A}$, then there exists the minimal linear congruence $\sigma$ such that $\mathbb{A}/\sigma$ is isomorphic to a linear algebra.

Finally, we will formulate and prove an important theorem used in Zhuk's algorithm.

\begin{theorem}[Affine subspaces \cite{10.1145/3402029}]\label{AffineSubspaces}
Suppose that relation $\rho \subseteq (\mathbb{Z}_{p_1})^{n_1}\times ... \times (\mathbb{Z}_{p_k})^{n_k}$ is preserved by $x_1+...+x_m$, where $p_1,...,p_k$ are distinct prime numbers dividing $m-1$ and $\mathbb{Z}_{p_i} = (\mathbb{Z}_{p_i},x_1+...+x_m)$ for every $i$. Then $\rho = L_1\times...\times L_k$, where each $L_i$ is an affine subspace of $(\mathbb{Z}_{p_i})^{n_i}$.
\end{theorem}

\begin{proof}
We first derive a ternary operation on every $\mathbb{Z}_{p_i}$.
\begin{equation}
    \begin{split}
        &f(x,y,z) = x-y+z \,(mod \,p_i) = \Omega(x,z,0,...,0)+\Omega(y,...,y,0,0) = \\
        &\hspace{40pt} =x+z+y+...+y = \Omega(x,z,y,...,y).
    \end{split}
\end{equation}
Thus, $f(x,y,z)$ preserves $\rho$. Now consider the relation $\rho\subseteq (\mathbb{Z}_{p_1})^{n_1}\times ... \times (\mathbb{Z}_{p_k})^{n_k}$ and choose any element $a\in\rho$. The set $\vec{V} = \{v|a+v\in\rho\}$ obviously contains $0$. Moreover, it is closed under $+$. Consider any $v_1,v_2\in \vec{V}$, $a+v_1, a+v_2\in \rho$. Then $v_1+v_2\in \vec{V}$ since $f(a+v_1,a,a+v_2)=a+ v_1+v_2\in \rho$. Thus, $\vec{V}$ is a linear subspace and $\rho$ is therefore an affine subspace.
\end{proof}

In the remainder of this subsection we will give two elementary examples of constraint languages corresponding to linear algebras. We will consider classical digraphs, relational structures with unique binary relation of being an edge. Due to Theorem \ref{fjjduh87}, each relational structure $\mathcal{A}$ corresponds to an algebra $\mathbb{A}$ such that $Clone(\mathbb{A})=Pol(\mathcal{A})$. We can assume that for both CSP instances there is a special WNU operation $\Omega$ of some arity $m$, which is a polymorphism for all constraint relations. 

% ILLUSTRATION EXAMPLE_1
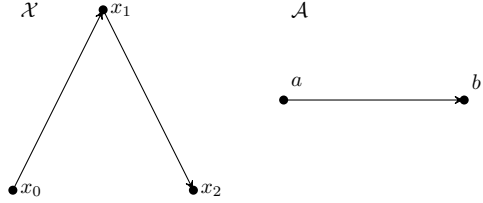
\begin{figure}[H]
\begin{center}
{ 
\begin{tikzpicture}[scale=1.2, every label/.style={scale=0.8}]

\coordinate   [label=right:${x_0}$] () at (1,1);
\coordinate   [label=right:${x_1}$] () at (2,3);
\coordinate   [label=right:${x_2}$] () at (3,1);
\coordinate   [label=right:${\mathcal{X}}$] () at (1,3);

%\draw[step=1 cm,white,very thin] [dashed] (0,0) grid (7,4);
\draw[thin,->] (1,1) -- (2,3) node[anchor=north west] {};
\draw[thin,->] (2,3) -- (3,1) node[anchor=north west] {};

%\tikzset{
%  big dot/.style={
%    circle, inner sep=0pt, 
 %   minimum size=3mm, fill=gray
% }
%}

\fill[black] (1,1) circle (0.5mm);
\fill[black] (2,3) circle (0.5mm);
\fill[black] (3,1) circle (0.5mm);

\draw[thin,->] [black]  (4,2) -- (6,2) node[anchor=north west] {};

\coordinate   [label=right:${a}$] () at (4,2.2);
\coordinate   [label=right:${b}$] () at (6,2.2);
\coordinate   [label=right:${\mathcal{A}}$] () at (4,3);

\fill[black]  (4,2) circle (0.5mm);
\fill[black]  (6,2) circle (0.5mm);

%\node[big dot] (A) at (3,0)   {};

\end{tikzpicture}
}
\end{center}
    \caption{Example 1.} \label{tggghhhjjcf}
  \end{figure}
Consider CSP($\mathcal{A}$), where $\mathcal{A}=(V_\mathcal{A}, E_\mathcal{A})$ is the digraph on two vertices and $E_\mathcal{A} = \{(a,b)\}$. An instance of CSP($\mathcal{A}$), depicted in Figure \ref{tggghhhjjcf}, is the digraph $\mathcal{X}=(V_\mathcal{X}, E_\mathcal{X})$, where $V_\mathcal{X} = \{x_0,x_1,x_2\}$ and $E_\mathcal{X} = \{(x_0,x_1), (x_1,x_2)\}$. It is obvious that there is no homomorphism from $\mathcal{X}$ to $\mathcal{A}$. Let us define a $3$-ary operation $\Omega$ on $V_\mathcal{A}$ as follows:
\begin{equation}
    \begin{split}
        &\Omega(a,a,a)=a, \,\,\,\Omega(b,b,b)=b,\\
        &\Omega(b,a,a)=\Omega(a,b,a)=\Omega(a,a,b)=b, \\
        &\Omega(a,b,b)=\Omega(b,a,b)=\Omega(b,b,a)=a.\\
    \end{split}
\end{equation}
$\Omega$ preserves $E_\mathcal{A}$ and is clearly idempotent, WNU and special: 
\begin{equation}
    \begin{split}
        &\Omega(a,a,\Omega(a,a,b))=\Omega(a,a,b)=b,\\
        &\Omega(b,b,\Omega(b,b,a))=\Omega(b,b,a)=a. \\
    \end{split}
\end{equation}
We can define an operation $+$ on $V_\mathcal{A}$ as $(a+x)=(x+a)=x$ (i.e. $a$ is zero) and $(b+b) = a$ (i.e. $b$ is an inverse element to itself). Hence $\mathbb{A}=(V_\mathcal{A},+)$ is a finite abelian group, namely $\mathbb{Z}_2$, and the algebra $(V_\mathcal{A},\Omega)$ is isomorphic to linear algebra $(\mathbb{Z}_2, x+y+z)$. 

The instance has two constraints, $E_\mathcal{X}(x_0,x_1) \subseteq \mathbb{Z}_2\times \mathbb{Z}_2$ and $E_\mathcal{X}(x_1,x_2) \subseteq \mathbb{Z}_2\times \mathbb{Z}_2$. Since $E_\mathcal{A}=\{(a,b)\}$ is an affine subspace of $\mathbb{Z}_2\times \mathbb{Z}_2$, we can express constraints as a conjunction of the linear equations
\[ E_\mathcal{X}(x_0,x_1) \iff \left\{
\begin{array}{ll}
x_0 & = a\textrm{,}\\
x_1 & = b.
\end{array} \right.\]
\[ E_\mathcal{X}(x_1,x_2) \iff \left\{
\begin{array}{ll}
x_1 & = a\textrm{,}\\
x_2 & = b.
\end{array} \right.\]
The instance can be viewed as a system of linear equations in different fields and it has no solution.

Now consider a different example in Figure \ref{kfkfkjgut}, where $\mathcal{A}=(V_\mathcal{A},E_\mathcal{A})$ is the digraph on two vertices with $E_\mathcal{A} = \{(a,b), (b,a)\}$, and the instance digraph $\mathcal{X}=(V_\mathcal{X},E_\mathcal{X})$ is the same.  

% ILLUSTRATION EXAMPLE_1
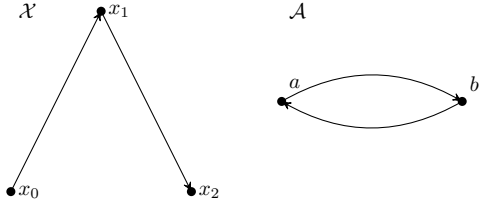
\begin{figure}[H]
\begin{center}
{ 
\begin{tikzpicture}[scale=1.2, every label/.style={scale=0.8}]

\coordinate   [label=right:${x_0}$] () at (1,1);
\coordinate   [label=right:${x_1}$] () at (2,3);
\coordinate   [label=right:${x_2}$] () at (3,1);
\coordinate   [label=right:${\mathcal{X}}$] () at (1,3);

%\draw[step=1 cm,white,very thin] [dashed] (0,0) grid (7,4);
\draw[thin,->] (1,1) -- (2,3) node[anchor=north west] {};
\draw[thin,->] (2,3) -- (3,1) node[anchor=north west] {};

%\tikzset{
%  big dot/.style={
%    circle, inner sep=0pt, 
 %   minimum size=3mm, fill=gray
% }
%}

\fill[black] (1,1) circle (0.5mm);
\fill[black] (2,3) circle (0.5mm);
\fill[black] (3,1) circle (0.5mm);

\draw[thin,->] [black]  (4,2) to[bend left] (6,2) node[anchor=north west] {};
\draw[thin,->] [black]  (6,2) to[bend left] (4,2) node[anchor=north west] {};

\coordinate   [label=right:${a}$] () at (4,2.2);
\coordinate   [label=right:${b}$] () at (6,2.2);
\coordinate   [label=right:${\mathcal{A}}$] () at (4,3);

\fill[black]  (4,2) circle (0.5mm);
\fill[black]  (6,2) circle (0.5mm);

%\node[big dot] (A) at (3,0)   {};

\end{tikzpicture}
}
\end{center}
    \caption{Example 2.}\label{kfkfkjgut}
  \end{figure}
Since the constraint relation $E_\mathcal{A}$ is still preserved by above defined $\Omega$, $(V_{\mathcal{A}},\Omega)$ is isomorphic to $(\mathbb{Z}_2, x+y+z)$. But $E_\mathcal{A}$ differs from the relation in the previous example, so we can express constraints as the linear equations 
\begin{equation}
    \begin{split}
        &E_\mathcal{X}(x_0,x_1) \iff x_0 +x_1 = b; \\
        &E_\mathcal{X}(x_1,x_2) \iff x_1+x_2 = b.
    \end{split}
\end{equation}
This system has two solutions, $S_1 = \{x_0=x_2 =a, x_1=b\}$ and $S_2 = \{x_0=x_2 =b, x_1=a\}$, and the instance is therefore satisfiable. 

\subsection{Zhuk's four-cases theorem}
Zhuk's algorithm is based on the following theorem:
\begin{theorem}[\cite{10.1145/3402029}]\label{4cases}
If $\mathbb{A}$ is a nontrivial finite idempotent algebra with WNU operation, then at least one of the following is true:
\begin{itemize}
    \item $\mathbb{A}$ has a nontrivial binary absorbing subuniverse,
    \item $\mathbb{A}$ has a nontrivial centrally absorbing subuniverse,
    \item $\mathbb{A}$ has a nontrivial PC quotient,
    \item $\mathbb{A}$ has a nontrivial affine quotient.
\end{itemize}
\end{theorem}

\section{Zhuk's algorithm}\label{ZHUfgergK}

Here we will briefly sketch the leading ideas of Zhuk's algorithm without any details. All details necessary for the formalization will be given directly in the corresponding subsections. For more information we send the reader to the original paper \cite{10.1145/3402029}. 

In this section we will consider an arbitrary constraint language (since the algorithm is designed for all finite languages). Before running the algorithm, it is necessary to make a slight modification of the constraint language. Suppose we have a finite language $\Gamma'$ that is preserved by an idempotent WNU operation $\Omega'$. By Lemma \ref{SPECIALWNU}, $\Gamma'$ is therefore also preserved by a special WNU operation $\Omega$. Let $k'$ be the maximal arity of the relations in $\Gamma'$ and denote by $\Gamma$ the set of all relations of arity at most $k'$ that are preserved by $\Omega$. Hence all $pp$-definable relations of arity at most $k'$ are in $\Gamma$, and CSP$(\Gamma')$ is an instance of CSP($\Gamma$). 

The common property of all parts of the algorithm is that any time when it reduces or restricts domains, the algorithm uses recursion. 

\subsection{Outline of the general part}
The key notion of the general part of Zhuk's algorithm is reduction, which is divided into several procedures. Consider a CSP instance of CSP($\Gamma$), $\Theta = (X,D,C)$. In this part, the algorithm gradually reduces different domains until it terminates in the linear case. At every step, it either produces a reduced domain or moves to the other type of reduction, or answers that there is no solution (if some domain is empty after one of the procedures). After outputting any reduced domain, the algorithm runs all from the beginning for the same instance $\Theta$ but with a smaller domain $D'$. 

First, the algorithm reduces domains until the instance is cycle-consistent. Then it checks irreducibility: again, if the instance is not irreducible, the algorithm can produce a reduction to some domain. The next step is to check a weaker instance that is produced from the instance by simultaneously replacing all constraints with all weaker constraints: if the solution set to such an instance is not subdirect, then some domain can be reduced. 

After these types of consistency, the algorithm checks whether some domains have a nontrivial binary absorbing subuniverse or a nontrivial center. If any of them does, the algorithm reduces the domain to the subuniverse or to the center. Then it checks whether there is a proper congruence on any domain such that its factor algebra is polynomially complete. If there is such a congruence, then the algorithm reduces the domain to an equivalence class of the congruence.

By Theorems \ref{fhfhfhydj}, \ref{llldju67yr} and \ref{fjfjhgd}, proved by Zhuk in \cite{10.1145/3402029}, if the reduced instance has no solution, then so does the initial one. 

\begin{theorem}[\cite{10.1145/3402029}]\label{fhfhfhydj} Suppose $\Theta$ is a cycle-consistent irreducible CSP instance, and $B$ is a nontrivial binary absorbing subuniverse of $D_i$. Then $\Theta$ has a solution if and only if $\Theta$ has a solution with $x_i\in B$.
\end{theorem}

\begin{theorem}[\cite{10.1145/3402029}]\label{llldju67yr} Suppose $\Theta$ is a cycle-consistent irreducible CSP instance, and $B$ is a nontrivial center of $D_i$. Then $\Theta$ has a solution if and only if $\Theta$ has a solution with $x_i\in B$.
\end{theorem}

\begin{theorem}[\cite{10.1145/3402029}]\label{fjfjhgd}
Suppose $\Theta$ is a cycle-consistent irreducible CSP instance, there does not exist a nontrivial binary absorbing subuniverse or a nontrivial center on $D_j$ for every $j$, $(D_i,\Omega)/\sigma_i$ is a polynomially complete algebra, and $E$ is an equivalence class of $\sigma_i$. Then $\Theta$ has a solution if and only if $\Theta$ has a solution with $x_i\in E$.
\end{theorem}

Finally, if the algorithm cannot reduce any domain of the CSP instance $\Theta$ any further, by Theorem \ref{4cases} it means that every domain $D_i$ of size greater than $1$ has a nontrivial affine quotient. Since we consider the special WNU operation $\Omega$, for every domain $D_i$ there exists a congruence $\sigma_i$ such that $(D_i,\Omega)/\sigma_i$ is isomorphic to $(\mathbb{Z}_{p_1}\times...\times \mathbb{Z}_{p_l},x_1+...+x_m)$ for some prime numbers $p_1,...,p_l$. The algorithm then proceeds with procedures embraced in the linear case.

\subsection{Outline of the linear case}\label{DETAILEDLINEAR}

The linear case of Zhuk's algorithm is adopted from \cite{10.1145/3402029}. Suppose that on every domain $D_i$ there exists the proper minimal linear congruence $\sigma_i$ such that $(D_i,\Omega)/\sigma_i$ is linear, i.e. isomorphic to $(\mathbb{Z}_{p_1}\times...\times \mathbb{Z}_{p_l},x_1+...+x_m)$ for some prime numbers $p_1,...,p_l$, where $m$ is the arity of $\Omega$. 

Denote each $D_i/\sigma_i$ by $L_i$ and define a new CSP instance $\Theta_L$ with domains $L_1,...,L_n$ as follows: to every constraint $(x_{i_1},...,x_{i_s};R)\in \Theta$ assign a constraint $(x'_{i_1},...,x'_{i_s};R')$, where $R'\in L_{i_1}\times...\times L_{i_s}$ and a tuple of blocks of congruences $(E_1,...,E_s)\in R'\iff (E_1\times...\times E_s)\cap R \neq \emptyset$. From now we will refer to the instance $\Theta$ as the initial instance, and to $\Theta_L$ as the factorized one. 

Since each $L_i = D_i/\sigma_i$ is isomorphic to some $\mathbb{Z}_{s_1}\times...\times \mathbb{Z}_{s_l}$, we can define a natural bijective mapping $\psi: \mathbb{Z}_{p_1}\times...\times \mathbb{Z}_{p_r} \to L_1\times...\times L_n$ and assign a variable $z_i$ to every $\mathbb{Z}_{p_i}$. By Theorem \ref{AffineSubspaces} every relation on $ \mathbb{Z}_{p_1}\times...\times \mathbb{Z}_{p_r}$ preserved by $\Omega(x_1,...,x_m)=x_1+...+x_m$ is an affine subspace, the instance $\Theta_L$ can thus be viewed as a system of linear equations over $z_1,...,z_r$. Every linear equation is an equation in $\mathbb{Z}_{p_i}$, and only variables ranging over the same field $\mathbb{Z}_{p_i}$ may appear in one equation. 

The algorithm compares two sets: the solution set to the initial instance $\Theta$ factorized by congruences (let us denote it by $S_{\Theta}/\Sigma$) and the solution set to the factorized instance, $S_{\Theta_L}$. It is known that $S_{\Theta}/\Sigma\subseteq S_{\Theta_L}$. We do not know $S_{\Theta}/\Sigma$, but we can efficiently calculate $S_{\Theta_L}$ using Gaussian Elimination (since Gaussian Elimination is strongly polynomial \cite{edmonds1967systems}). If $\Theta_L$ has no solution, then so does the initial instance. If the solution has no independent variables (i.e. there is only one solution and the dimension of the solution set is 0), the algorithm checks whether the initial instance $\Theta$ has the solution corresponding to this solution by restricting every domain $D_i$ of $\Theta$ to the corresponding congruence blocks and recursively calling the algorithm for these smaller domains. Otherwise, the algorithm arbitrarily chooses independent variables $y_1,...,y_k$ of the general solution to $\Theta_L$ (the dimension of the solution set $S_{\Theta_L}$ is $k$). 

The set $S_{\Theta_L}$ can be defined as an affine mapping $\phi:\mathbb{Z}_{q_1}\times...\times\mathbb{Z}_{q_{k}} \to L_1\times...\times L_n$. Thus, any solution to $\Theta_L$ can be obtained as $\phi(a_1,...,a_k)$ for some $(a_1,...,a_k) \in \mathbb{Z}_{q_1}\times...\times\mathbb{Z}_{q_k}$. 

The algorithm denotes an empty set of linear equations by $Eq$. The following steps will be repeated until the algorithm either finds a solution or answers that $S_{\Theta}/ \Sigma$ is empty. The idea is to add equations iteratively to the solution set $S_{\Theta_L}$ maintaining the property $S_{\Theta}/ \Sigma \subseteq S_{\Theta_L}\cup Eq$. Since the dimension of $S_{\Theta_L}$ is $k$, and at every iteration the algorithm reduces the dimension by at least one, the process will eventually stop. 

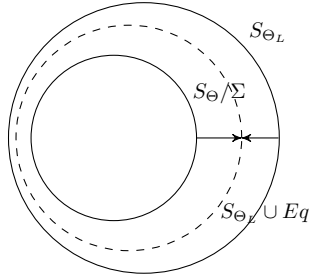
\begin{figure}[H]
\begin{center}
{ 
\begin{tikzpicture}[scale=1, every label/.style={scale=0.8}]

\filldraw[color=black, fill=white, thin](0,0) circle (1.8);
\filldraw[color=black, dashed, fill=white, thin](-0.2,0) circle (1.5);
\filldraw[color=black, fill=white, thin](-0.4,0) circle (1.1);

\coordinate   [label=right:${S_{\Theta_L}}$] (B) at (1.3,1.4);
\coordinate   [label=right:${S_{\Theta_L}\cup Eq}$] (A') at (0.92,-1);
\coordinate   [label=right:${S_{\Theta}/\Sigma}$] (A) at (0.55,0.6);

\draw[thin,->][black] (1.79,0) -- (1.3,0) node[anchor=north west] {};
\draw[thin,->][black] (0.7,0) -- (1.3,0) node[anchor=north west] {};

\end{tikzpicture}
}
\end{center}
    \caption{Solution sets. }\label{PicEveryCLone1}
  \end{figure}

First of all, the algorithm checks whether $\Theta$ has a solution corresponding to $\phi(0,...,0)$ by recursively calling the algorithm for smaller domains. If it does, the algorithm stops with a solution, if it does not, it has established the property $S_{\Theta}/ \Sigma \subsetneq S_{\Theta_L}$. Then the algorithm starts to decrease the solution set $S_{\Theta_L}$. It always starts with the initial instance $\Theta$, gradually makes it weaker and at every weakening checks whether the solution set to this new weaker instance is equal to $S_{\Theta_L}$. 

To make $\Theta$ weaker, the algorithm arbitrarily chooses a constraint $C$ and replaces it with all weaker constraints without dummy variables simultaneously. Let us denote this instance by $\Theta'$. To check whether the solution set $S_{\Theta'}/\Sigma$ to $\Theta'$ factorized by congruences is equal to $S_{\Theta_L}$, one needs to check whether $\Theta'$ has solutions corresponding to $\phi(a_1,...,a_k)$ for every $(a_1,...,a_k)\in\mathbb{Z}_{q_1}\times...\times\mathbb{Z}_{q_k}$ (using recursion for smaller domains). Since $S_{\Theta'}/\Sigma$ and $S_{\Theta_L}$ are subuniverses of $L_1\times...\times L_n$, it is enough to check the existence of solutions corresponding to $\phi(0,...,0)$ and $\phi(0,...,1,...,0)$ for any position of $1$. If the solution set to the weaker instance $\Theta'$ does not contain $S_{\Theta_L}$, the algorithm proceeds with weakening the instance $\Theta'$ step by step until it cannot make the instance weaker without obtaining  $S_{\Theta_L}\subseteq S_{\Theta'}/\Sigma$ (at this point the algorithm checks that whichever constraint it weakens, every solution to $\Theta_L$ will be a solution to $\Theta'$). It means that there exists some $(b_1,...,b_k)$ $\in \mathbb{Z}_{q_1}\times...\times\mathbb{Z}_{q_k}$ such that $\Theta'$ has no solution corresponding to $\phi(b_1,...,b_k)$. However, if we replace any constraint $C\in \Theta'$ with all weaker constraints simultaneously, then we get an instance that has a solution corresponding to $\phi(a_1,..,a_k)$ for every $(a_1,...,a_k)\in \mathbb{Z}_{q_1}\times...\times\mathbb{Z}_{q_k}$.

Finally, the algorithm finds the solution set $S_{\Theta'}/\Sigma$ to the instance $\Theta'$ factorized by congruences by finding new equations additional to the set $S_{\Theta_L}$. There are different strategies for linked and non-linked instances $\Theta'$. For linked instance, it is known that $S_{\Theta'}/\Sigma\subsetneq S_{\Theta_L}$ is of codimension $1$, so we can find only one equation and add it to $S_{\Theta_L}$. For non-linked instance $\Theta'$ we find all equations that describe $S_{\Theta'}/\Sigma$, and then intersect these equations with $S_{\Theta_L}$ (see \cite{10.1145/3402029}). After new equations are found, the algorithm adds them to the set $Eq$, solves $S_{\Theta_L}\cup Eq$ using Gaussian Elimination, and runs another iteration. 

\begin{remark}
By Theorem \ref{AffineSubspaces}, $S_{\Theta_L} \subseteq (\mathbb{Z}_{p_1})^{n_1}\times ... \times (\mathbb{Z}_{p_k})^{n_k}$ is an affine subspace. The solution set $S_{\Theta}/\Sigma$ to the initial instance factorized by congruences is also an affine subspace: the relation that describes it is a subset of $S_{\Theta_L}$, i.e. it is also preserved by $\Omega$. Moreover, when we consider the solution set $S_{\Theta'}/\Sigma$ to the weaker instance $\Theta'$ factorized by congruences, it is also an affine subspace since all weaker constraints are in $\Gamma$.
\end{remark}  

\section{Soundness of Zhuk's algorithm in a theory of bounded arithmetic}\label{SOUNDNggghESS}

To prove the soundness of Zhuk's algorithm in some theory of bounded arithmetic, it is sufficient to prove that after every step of the algorithm one does not lose all the solutions to the initial instance. Consider any relational structure $\mathcal{A}$ with at most binary relations and some negative instance $\Theta=(\mathcal{X},\ddot{\mathcal{A}})$ of CSP($\mathcal{A}$), and suppose that there is a homomorphism from $\mathcal{X}$ to $\ddot{\mathcal{A}}$. If the elected theory of bounded arithmetic proves that after every step of the algorithm the new modified instance has solutions only if the previous one does, and the algorithm terminates with no solution, then the theory proves - by its level of bounded induction - that $\mathcal{X}$ is unsatisfiable, and hence that $\neg HOM(\mathcal{X},\ddot{\mathcal{A}})$ is a tautology. 

Consider computation of the algorithm on $(\mathcal{X},\ddot{\mathcal{A}})$, $W = (W_1,W_2,...,W_k)$, where: 
    \begin{itemize}
        \item $W_1 = (\mathcal{X},\ddot{\mathcal{A}})$;
        \item $W_{i+1}=(\mathcal{X}_{i+1},\ddot{\mathcal{A}}_{i+1})$ is obtained from $W_i=(\mathcal{X}_{i},\ddot{\mathcal{A}}_{i})$ by one algorithmic step ($\mathcal{X}_{i+1}$ and $\ddot{\mathcal{A}}_{i+1}$ are some modifications of relational structures $\mathcal{X}_{i},\ddot{\mathcal{A}}_{i}$);
        \item $W_k$ has no solution.
    \end{itemize}

We need to prove, for all types of algorithmic modifications, that if $W_i$ has a solution, then $W_{i+1}$ also has a solution. This will prove that if the algorithm terminates with no solution, then there is no homomorphism from $\mathcal{X}$ to $\ddot{\mathcal{A}}$. Note that it is unnecessary to prove the opposite direction when considering soundness. Moreover, neither it is necessary to prove that the algorithm is well-defined. The transcription of the algorithm's computation can include all auxiliary necessary information.
    
In the formalization of the algorithm we will incorporate some modifications and adjustments suggested by Zhuk in his later paper \cite{DBLP:journals/mvl/Zhuk21}. We also sometimes will omit some intermediate steps and other technicalities not affecting the result. We will explicitly highlight all points that distinguish this version of the algorithm from the original one. 

In the paper we shall prove the soundness of Zhuk's algorithm in a new theory of bounded arithmetic, namely $V^1$ augmented with three universal algebra axioms, which will be defined in the next section. 

\subsection{Defining a new theory of bounded arithmetic}\label{dddddf}
In this section we will define a new theory of bounded arithmetic that will extend the theory $V^1$. Before moving to this section, we recommend that the reader recall subsections \ref{BAandT} and \ref{jjjdkyfopsy}.
\subsubsection{Arrangements before the run of the algorithm}
We will consider only algebras whose corresponding relational structures contain at most binary relations, see Theorem \ref{aakrrrfetsplgh}. The algorithm works for any finite algebra having a WNU term and uses the fact that this term and all the properties of the algebra are known in advance. From here on out we fix algebra $\mathbb{A}=(A,\Omega)$ and suppose that the only basic operation on $\mathbb{A}$ is idempotent special WNU operation $\Omega$. Algebras with richer signatures can be treated in a similar way, extending all conditions imposed on $\Omega$ to other (know in advance) basic operations. 

Since at the beginning Zhuk's algorithm adds to a constraint language $\Gamma$ all relations preserved by $\Omega$, of the arity up to the maximal arity of relations in $\Gamma$, we will consider the finite set of \emph{all} relations of arity at most $2$, invariant under $\Omega$, which we know in advance. Let us denote this set by $\Gamma_\mathcal{A}$, and the relational structure by $\mathcal{A}=(A,\Gamma_\mathcal{A})$. Any time when in formulas we claim something about this set, it means that we claim this about each relation in this set. 

A new theory of bounded arithmetic will extend the theory $V^1$. Before we introduce this theory, we need to define in $V^1$ notions from different areas of mathematics.

\subsubsection{Encoding relational structure}

We encode the finite universe $A$ of size $l$ by the set $A,\forall i<l, A(i)$, and $\Gamma_{\mathcal{A}}$ as a pair of sets $(\Gamma^1_{\mathcal{A}},\Gamma^2_{\mathcal{A}})$ where $\Gamma^1_{\mathcal{A}}$ is the set which encodes all unary relations from $\Gamma_{\mathcal{A}}$, and $\Gamma^2_{\mathcal{A}}$ encodes binary relations, 
$$
\Gamma^1_{\mathcal{A}}(j,a)\iff D^1_j(a)\text{ and } \Gamma^2_{\mathcal{A}}(i,a,b)\iff E^2_i(a,b).
$$ 
Note that in the list of $\Gamma_{\mathcal{A}}$ there are all possible subalgebras of $\mathbb{A}$ (i.e. all possible domains and strong subsets), and all possible $pp$-definitions constructed from unary and binary relations preserved by $\Omega$. When consider a subset $D$ of $A$, we will denote by $\Gamma_{\mathcal{D}}$ the set of unary and binary relations from $\Gamma_{\mathcal{A}}$ restricted to the set $D$.

Among binary relations $\Gamma^2_{\mathcal{A}}$ there are all congruences on $\mathbb{A}$ and on all its subalgebras. Let us denote this set by $\Sigma_\mathcal{A}$. Since for any subalgebra $\mathbb{D}$ any congruence of $\mathbb{A}$ is also a congruence of $\mathbb{D}$, the formula
$$D^1_j(a)\wedge D^1_j(b)\wedge\Sigma_\mathcal{A}(i,a,b)$$
defines a congruence on some $\mathbb{D}$. The number of all possible congruences on $\mathbb{A}$ is bounded by $2^{|A|^2}$.

\subsubsection{Encoding special WNU operation and polymorphism}

We can define a special WNU operation of fixed arity $m$ on some set $A$ in the theory $V^1$ in several steps. We say that a set $F$ is an $m$-ary \emph{operation} $F:A^m\rightarrow A$ on a set $A$ if it satisfies the relation
\begin{equation}
 \begin{split}
&\hspace{10pt}OP_m(F,A) \iff \forall x_0,...,x_{m-1}\in A, \exists y\in A,\,\,F(x_0,...,x_{m-1})=y\wedge
\\
&\hspace{0pt}\wedge\forall y_1,y_2 \in A\,\,(F(x_0,...,x_{m-1})=y_1\wedge F(x_0,...,x_{m-1})=y_2 \to y_1=y_2).
 \end{split}
\end{equation}
An idempotent operation $F$ is defined straightforwardly:
\begin{equation}
 \begin{split}
 &IDM_m(F,A) \iff OP_m(F,A) \wedge \forall a\in A \,\, F(a,a,...,a)=a.
 \end{split}
\end{equation}
We say that a set $\Omega$ is a WNU operation of arity $m$ on the set $A$ if it satisfies the relation
\begin{equation}
 \begin{split}
 &\hspace{0pt}wNU_m(\Omega,A) \iff OP_m(\Omega,A) \wedge \forall a,b\in A, \exists c\in A, \forall x_0,...,x_{m-1}\in A\\
 &\hspace{30pt}\bigwedge_{t<m} (x_{t}=a \wedge \forall j\neq t<m,\, x_{j}=b \to \Omega(x_0,...,x_{m-1})=c).
 \end{split}
\end{equation}
A special WNU operation is defined as follows:
\begin{equation}
 \begin{split}
 &\hspace{5pt}SwNU_m(\Omega,A) \iff wNU_m(\Omega,A) \wedge IDM_m(\Omega,A) \\
 &\hspace{0pt}\forall a,b\in A, \exists c\in A,\,\, \Omega( a,...,a,b)=c \wedge \Omega( a,...,a,c)=c.
  \end{split}
\end{equation}
Since we work with relations of arity at most $2$, we will define polymorphisms only for relations of this arity. We say that a set $F$ is an operation of arity $m$ on the set $A$ that preserves $2$-ary relation $R$ on $A$ if it satisfies the following relation
\begin{equation}
 \begin{split}
&Pol_{m,2}(F,A,R) \iff OP_m(F,A) \wedge \forall a^{0}_{1},...,a^{m-1}_{1},a^0_{2},...,a^{m-1}_{2}\in A, \\
&\hspace{60pt}\forall b_1,b_2\in A,\,\,R(a^0_{1},a^0_{2})\wedge ... \wedge R(a^{m-1}_{1},a^{m-1}_{2}) \wedge\\
&\hspace{35pt}\wedge F( a^0_{1},...,a^{m-1}_{1})=b_1\wedge F(a^0_{2},...,a^{m-1}_{2})=b_2 \rightarrow R(b_1,b_2).
 \end{split}
\end{equation}
Finally, operation $F$ preserves $1$-ary relation $R$ on $A$ if
\begin{equation}
 \begin{split}
&Pol_{m,1}(F,A,R) \iff OP_m(F,A) \wedge \forall a^0,a^1,...,a^{m-1}\in A, \\
&\hspace{60pt}\forall b\in A\,\,R(a^0)\wedge ... \wedge R(a^{m-1})\wedge \\
&\hspace{62pt}\wedge F(a^0,...,a^{m-1})=b \rightarrow R(b).
 \end{split}
\end{equation}
We will omit the second index $i$ in $Pol_{m,i}$ when we refer to the whole set of relations $\Gamma_{\mathcal{A}}$.

\subsubsection{Encoding notions from universal algebra}

A finite algebra with special WNU operation of size $l$ is a pair of sets $\mathbb{A}=(A,\Omega)$, where $|A|=l$, $A(i)$ for every $i$, and $\Omega$ is a $((m+1)l)^{2^{m+1}}$ set representing a special WNU operation on $A$. We will call this pair a Taylor algebra and denote it by $TA(A,\Omega)$. From here on out under algebra we mean Taylor algebra. We say that $\mathbb{B}=(B,\Omega)$ is a subalgebra of algebra $\mathbb{A}$ if
\begin{equation}
  STA(\mathbb{B},\mathbb{A})\iff |B|=|A|\wedge \forall i<l, B(i)\rightarrow A(i)\wedge SwNU(\Omega,B).
\end{equation}
Note that condition $SwNU(\Omega,B)$ ensures that $\mathbb{B}$ is closed under operation $\Omega$. The difference between fixed algebra $\mathbb{A}$ and all its subalgebras and factor algebras is that the size of all these objects is bounded by $l$, but since it is not necessary that for all $i<l,B(i)$, we will measure their size by census function, $\texttt{\#}B(l)$. We say that a pair of sets $\mathbb{B}=(B,\Omega)$ is a direct product of $n$ algebras $\mathbb{A}_0=(A_0,\Omega_0)$, ...,$\mathbb{A}_{n-1}=(A_{n-1},\Omega_{n-1})$ of the same type if
\begin{equation}
    \begin{split}
    &DP_{m,n}(B,\Omega,A_0,\Omega_0,...,A_{n-1},\Omega_{n-1})\iff \forall a_0\in A_0,...,a_{n-1}\in A_{n-1},\\ & \hspace{7pt}B(a_0,...,a_{n-1})\wedge \forall a^1_{0},a^2_{0},...,a^m_{0}\in A_0,...,a^1_{n-1},a^2_{n-1},...,a^m_{n-1}\in A_{n-1}\\
&\hspace{25pt}\Omega(a^1_{0},a^2_{0},...,a^m_{0},...,a^1_{n-1},a^2_{n-1},...,a^m_{n-1})=\\
&\hspace{75pt}=(\Omega_0(a^1_{0},a^2_{0},...,a^m_{0}),...,\Omega_{n-1}(a^1_{n-1},a^2_{n-1},...,a^m_{n-1})).
    \end{split}
\end{equation}
We will denote $(B,\Omega)$ by $(A_{0}\times...\times A_{n-1},\Omega)$. A subdirect $n$-ary relation $R$ on $A_0\times...\times A_{n-1}$ is encoded as follows:
\begin{equation}
\begin{split}
    &SDR_{n}(R, A_0,...,A_{n-1})\iff \bigwedge_{i<n} \forall a_{i}\in A_i, \exists a_{0}\in A_0,...,a_{i-1}\in A_{i-1},\\
    &\hspace{15pt}a_{i+1}\in A_{i+1},...,a_{n-1}\in A_{n-1}, R(a_0,...,a_{i-1},a_i,a_{i+1},...,a_{n-1}.)
\end{split}
\end{equation}
We say that a set $\sigma<l^2$ is a congruence relation on the algebra $\mathbb{A} =(A, \Omega)$ if it satisfies the following relation
\begin{equation}
 \begin{split}
&\hspace{0pt}Cong_m(A,\Omega,\sigma)\iff  Pol_{m,2}(\Omega,A,\sigma) \wedge \\
&\hspace{50pt}\forall a\in A,\, \sigma( a,a ) \wedge \forall a,b\in A, (\sigma( a,b )\leftrightarrow \sigma( b,a ))\wedge \\
&\hspace{120pt}(\forall a,b,c\in A,\,\sigma( a,b )\wedge \sigma( b,c )\rightarrow \sigma( a,c )).
 \end{split}
\end{equation}
Condition $Pol_{m,2}(\Omega,A,\sigma)$ ensures that $\sigma$ is from $Inv(Pol(\Gamma_{\mathcal{A}}))$. Recall that all congruences on $\mathbb{A}$ are listed in $\Sigma_{\mathcal{A}}$. If we additionally require that 
\begin{equation}
(\exists x,y\in D\, \neg\sigma( x,y ))\wedge (\exists x\neq y\in D\,\,\sigma( x,y )),
\end{equation}
the congruence $\sigma$ will be proper. A maximal congruence (a congruence over which there is no other congruences except the full binary relation $\nabla$) can be defined as follows:
\begin{equation}
    \begin{split}
        &MCong_m(A,\Omega,\sigma)\iff Cong_m(A,\Omega,\sigma) \wedge \exists a,b\in A,\, \neg\sigma( a,b )\wedge\\
        &\hspace{10pt}\wedge[\forall \sigma'<\langle l,l\rangle,\, (Cong_m(A,\Omega,\sigma')\wedge \exists a,b\in A,\, \neg\sigma'( a,b ))\rightarrow\\
        &\hspace{70pt}\rightarrow \exists a,b\in A,\,\sigma(a,b)\wedge \neg \sigma'(a,b)].
    \end{split}
\end{equation}
Note that this is a $\Sigma^{1,b}_1$-formula. A factor set is the set of all equivalence classes under the congruence $\sigma$ and it will be denoted by $A/\sigma$. We can represent each block of $\sigma$ by its minimal element (it exists by the Minimal principle). Therefore, we think of the factorized object $A/\sigma$ as of a set of numbers as well:
\begin{equation}
 \begin{split}
&\hspace{0pt}FS_m(A/\sigma,A,\Omega,\sigma) \iff Cong_m(A,\Omega,\sigma) \wedge \\
&\hspace{50pt}\forall a,b \in A,\, (\sigma( a,b )\wedge (a<b) \rightarrow \neg A/\sigma (b))\\
&\hspace{80pt}\wedge  (\forall a\in A(\forall a'\in A,\, \sigma( a,a') \rightarrow a\leq a') \rightarrow A/\sigma (a)).
 \end{split}
\end{equation}
We say that $a$ is a represent of the class $[a]/\sigma$ (where $[a]/\sigma$ is just a notation, it is any element of $A$) if
\begin{equation}
 \begin{split}
&Rep_m(a,[a]/\sigma,A/\sigma,A,\Omega,\sigma)\iff FS_m(A/\sigma,A,\Omega,\sigma) \wedge \\
&\hspace{70pt}\sigma(a,[a]/\sigma)\wedge A/\sigma(a).
 \end{split}
 \end{equation} 
Finally, we can define the factor algebra $\mathbb{A}/\sigma=(A/\sigma,\Omega/\sigma)$: 
\begin{equation}
 \begin{split}
&\hspace{22pt}FA_m(A/\sigma,\Omega/\sigma, A,\Omega,\sigma) \iff  FS_m(A/\sigma,A,\Omega,\sigma)\wedge\\
&\hspace{25pt}\wedge\big(\forall a_0,...,a_{m-1},c\in A,\forall [a_0]/\sigma,...,[a_{m-1}]/\sigma,[c]/\sigma\in A, \\
&\Omega([a_0]/\sigma,...,[a_{m-1}]/\sigma)=[c]/\sigma\wedge Rep_m(c,[c]/\sigma,A/\sigma,A,\Omega,\sigma)\wedge\\
&\hspace{0pt}\wedge\bigwedge_{i<m} Rep_m(a_i,[a_i]/\sigma,A/\sigma,A,\Omega,\sigma)\rightarrow \Omega/\sigma(a_0,...,a_{m-1})=c\big).
 \end{split}
\end{equation}
Thus, we define the operation $\Omega/\sigma$ on minimal elements of the congruence classes.

\subsubsection{Encoding digraphs and CSP properties} 
We will code a CSP instance on relational structures with at most binary relations in the following way. 

\begin{definition}
A \emph{directed input graph} is a pair $\mathcal{X} = (V_\mathcal{X},E_\mathcal{X})$ with $V_\mathcal{X}(i)$ for all $i<V_\mathcal{X}=n$ and $E_\mathcal{X}(i,j)$ being a binary relation on $V_\mathcal{X}$ (there is an edge from $i$ to $j$). A \emph{target digraph with domains} is an $(n+2)$-tuple of sets $\ddot{\mathcal{A}}=( V_{\ddot{\mathcal{A}}},E_{\ddot{\mathcal{A}}},D_0,...,$ $D_{n-1})$, where
\begin{itemize}
    \item $V_{\ddot{\mathcal{A}}}< \langle n,l\rangle$ is the set corresponding to the superdomain,
    \item $\forall i<n,\,D_i< l$ is the subset of length $l$ corresponding to the domain of variable $x_i$,
    \item $V_{\ddot{\mathcal{A}}}(i,a)\iff  D_i(a)$,
    \item $E_{\ddot{\mathcal{A}}}<\langle\langle n,l\rangle,\langle n,l\rangle \rangle$ is the set encoding relations $E^{ij}_{\ddot{\mathcal{A}}}(a,b)$ (there is an edge $(a,b)$ between $D_i$ and $D_j$): 
\begin{equation}
    \begin{split}
       &E_{\ddot{\mathcal{A}}}(u,v)\rightarrow \exists i,j<n\,\exists a,b<l\,\, u = \langle i,a\rangle\wedge v = \langle j,b\rangle\wedge\\
       &\hspace{80pt} D_i(a)\wedge D_j(b).
    \end{split}
\end{equation}
\end{itemize}
Sometimes we will use the notation $E^{ij}_{\ddot{\mathcal{A}}}(a,b)$ instead of $E_{\ddot{\mathcal{A}}}(\langle i,a \rangle, \langle j,b\rangle)$ for brevity sake. We will denote a pair of sets $\Theta =(\mathcal{X},\ddot{\mathcal{A}})$, satisfying all above conditions, by DG$(\Theta)$, and will call $\Theta$ an instance. This representation will allow us to construct a homomorphism from $\mathcal{X}$ to $\ddot{\mathcal{A}}$ with respect to different relations $E^{ij}_{\ddot{\mathcal{A}}}$ and different domains for all vertices $x_1,...,x_n$.
\end{definition}

\begin{definition}[Homomorphism from digraph $\mathcal{X}$ to digraph with domains $\ddot{\mathcal{A}}$] 
A map $H$ is a \emph{homomorphism between input digraph} $\mathcal{X}=(V_{\mathcal{X}},E_{\mathcal{X}})$, $V_{\mathcal{X}}=n$ \emph{and target digraph with domains} $\ddot{\mathcal{A}}=(V_{\ddot{\mathcal{A}}},E_{\ddot{\mathcal{A}}},D_0,...,D_{n-1})$, $V_\mathcal{A}<\langle n,l\rangle$ if $H$ is a homomorphism from $\mathcal{X}$ to $\ddot{\mathcal{A}}$ sending each $i\in V_\mathcal{X}$ to domain $D_i$ in $V_{\ddot{\mathcal{A}}}$.
The statement that there exists such an $H$ can be expressed by the $\Sigma^{1,b}_{1}$-formula
\begin{equation}\label{HOMINSTAG}
 \begin{split}
&\hspace{10pt}\ddot{HOM}(\mathcal{X},\ddot{\mathcal{A}}) \iff \exists H < \langle n,\langle n,l\rangle\rangle \big(MAP(V_\mathcal{X}, n,V_{\ddot{\mathcal{A}}},\langle n,l\rangle,H)\wedge \\
&\hspace{35pt}(\forall i<n,s<\langle n,l\rangle \,\, H(i)=s\rightarrow \exists a<l, s=\langle i,a\rangle\wedge D_i(a))\wedge \\
&\hspace{100pt}\forall i_1,i_2<n,\forall j_1,j_2<\langle n,l\rangle\\
&\hspace{40pt}(E_{\mathcal{X}}(i_1,i_2)\wedge H(i_1)=j_1\wedge H(i_2)=j_2\rightarrow E_{\ddot{\mathcal{A}}}(j_1,j_2)).
 \end{split}
\end{equation}
\end{definition}

Besides a homomorphism between two digraphs of different types, we will also need a classical homomorphism between digraphs of the same type. The existence of such a homomorphism between digraphs $\mathcal{G}$ and $\mathcal{H}$ with $V_{\mathcal{G}}<n$, $V_{\mathcal{G}}<m$ can be expressed by the following $\Sigma^{1,b}_1$-formula
\begin{equation}
 \begin{split}
&HOM(\mathcal{G},\mathcal{H}) \iff \exists H < \langle n,m\rangle \big(MAP(V_{\mathcal{G}},n,V_{\mathcal{H}},m,H)\wedge \\
&\hspace{80pt} \forall i_1,i_2<n, \forall j_1,j_2 < m \\&\hspace{15pt} (E_\mathcal{G}(i_1,i_2)\wedge H(i_1)=j_1 \wedge H(i_2)=j_2 \to E_\mathcal{H}(j_1,j_2))\big).
 \end{split}
\end{equation}
\begin{notation}
Sometimes we will write $\exists H< \langle n,m\rangle, HOM(\mathcal{G},\mathcal{H}, H)$ and $\exists H < \langle n,\langle n,l\rangle\rangle,$ $ \ddot{HOM}(\mathcal{X},\ddot{\mathcal{A}},H)$ to omit repetitions. 
\end{notation}
For an instance $\Theta = (\mathcal{X},  \ddot{\mathcal{A}})$ we call an instance $\Theta'=(\mathcal{X}',\ddot{\mathcal{A}})$ a subinstance of $\Theta$ if
\begin{equation}
\begin{split}
    &SINST(\mathcal{X}', \mathcal{X})\iff SS(V_{\mathcal{X'}}, V_{\mathcal{X}})\wedge SS(E_{\mathcal{X'}}, E_{\mathcal{X}})\wedge\\
    &\hspace{50pt}(E_{\mathcal{X'}}(x_1,x_2)\rightarrow x_1,x_2\in V_{\mathcal{X'}}).
\end{split}
\end{equation}
That is, the target digraph with domains $\ddot{\mathcal{A}}$ does not change, the set of vertices $V_{\mathcal{X'}}$ is a subset of $V_{\mathcal{X}}$, and the set of constraints $E_{\mathcal{X'}}$ is a subset of $E_{\mathcal{X}}$ defined only on $V_{\mathcal{X'}}$.

We need to encode three properties of a CSP instance: cycle-consistency, being a linked instance, and irreducibility. In order to certify the quantification complexity of the formulas, we will introduce them explicitly. Recall that we refer to any path or cycle with the edges not necessarily directed in the same direction as an undirected path or cycle. We say that a digraph $\mathcal{C}_t=(V_{\mathcal{C}_t},E_{\mathcal{C}_t})$ with $V_{\mathcal{C}_t}=\{0,1,...,t-1\}$ is an \emph{undirected cycle of length $t$} if it satisfies the following $\Sigma^{1,b}_0$-definable relation

\begin{equation}
 \begin{split}
& CYCLE(\mathcal{C}_t) \iff (E_{\mathcal{C}_t}(0,t-1)\vee E_{\mathcal{C}_t}(t-1,0))\wedge \\
&\hspace{10pt}\forall i<(t-1)\, (E_{\mathcal{C}_t}(i,i+1)\vee E_{\mathcal{C}_t}(i+1,i)) \wedge\\ &\hspace{0pt}\forall i,j<(t-1) (j\neq i+1 \to (\neg E_{\mathcal{C}_t}(i,j)\wedge \neg E_{\mathcal{C}_t}(j,i)).
 \end{split}
\end{equation}
We will define cycle-consistency through two homomorphisms.
\begin{definition}[Cycle-consistent instance]
An instance $\Theta = (\mathcal{X}, \ddot{\mathcal{A}})$ with $V_{\mathcal{X}}=n$, $V_{\ddot{\mathcal{A}}}<\langle n,l \rangle$ is $1$-consistent if it satisfies the following $\Sigma^{1,b}_0$-definable relation
\begin{equation}
    \begin{split}
        &\hspace{50pt}1C(\mathcal{X}, \ddot{\mathcal{A}})\iff \forall i<n, \forall a\in D_i, \forall j<n, \\
        &(E_{\mathcal{X}}(i,j) \rightarrow \exists b \in D_j, E^{ij}_{\mathcal{A}}(a,b))\wedge(E_{\mathcal{X}}(j,i) \rightarrow \exists b \in D_j, E^{ji}_{\mathcal{A}}(b,a)).
    \end{split}
\end{equation}
The instance $\Theta = (\mathcal{X}, \ddot{\mathcal{A}})$ is cycle-consistent if it is $1$-consistent and any undirected cycle $\mathcal{C}_t$ that can be homomorphically mapped into $\mathcal{X}$ with $H(0)=x_k$ can be homomorphically mapped into $\ddot{\mathcal{A}}$ for any $a\in D_k$. Cycle-consistency is expressed by the following $\Pi^{1,b}_2$-formula
\begin{equation}\label{DEFINITIONOFCYCLECONSISTENCY}
    \begin{split}
        &\hspace{20pt}CC(\mathcal{X},  \ddot{\mathcal{A}})\iff 1C(\mathcal{X}, \ddot{\mathcal{A}})\wedge \forall k<n,\forall a\in D_k,\forall t<n, \forall V_{\mathcal{C}_t}=t,\\
        & \forall E_{\mathcal{C}_t}\leq 4t^2,\forall H<\langle t,n \rangle,\big[ CYCLE(V_{\mathcal{C}_t},E_{\mathcal{C}_t})\wedge HOM(\mathcal{C}_t,\mathcal{X},H)\wedge H(0,k)\\   
        &\hspace{85pt}\rightarrow\exists H'< \langle t,\langle t,l\rangle\rangle, \ddot{HOM}(\mathcal{C}_t,\ddot{\mathcal{A}},H') \wedge\\
        &\hspace{10pt}\wedge\forall i<n,j<t\, (H(j)=i\rightarrow \exists b\in D_i,\, H'(j)=\langle i,b\rangle)\wedge H'(0)=\langle k,a\rangle\big].
    \end{split}
\end{equation}
\end{definition}

Note that for any cycle-consistent instance $\Theta=(\mathcal{X},\ddot{\mathcal{A}})$, any its subinstance $\Theta'=(\mathcal{X}',\ddot{\mathcal{A}})$ is also cycle-consistent. For any $i,j\in X'$ the constraint relations $D_i,D_j$, $E^{ij}_{\ddot{\mathcal{A}}}$ remain the same. We have just removed some vertices from $\mathcal{X}$ and have removed some edges from $E_\mathcal{X}$. This does not affect the cycle-consistency property: for any $i\in X'$, any $a\in D_i$, any \emph{existing} in $\Theta'$ path starting and ending in $i$ must connect $a$ and $a$.  

We say that a digraph $\mathcal{P}_t=(V_{\mathcal{P}_t},E_{\mathcal{P}_t})$ with $V_{\mathcal{P}_t}=\{0,1,...,t\}$ is an \emph{undirected path of length $t$} if it satisfies the $\Sigma^{1,b}_0$-definable relation
\begin{equation}
 \begin{split}
& PATH(\mathcal{P}_t) \longleftrightarrow \forall i<t\, (E_{\mathcal{P}_t}(i,i+1)\vee E_{\mathcal{P}_t}(i+1,i)) \wedge\\ &\hspace{10pt}\forall i<t,j\leq t (j\neq i+1 \to (\neg E_{\mathcal{C}_t}(i,j)\wedge \neg E_{\mathcal{C}_t}(j,i)).
 \end{split}
\end{equation}

For any two paths $\mathcal{P}_t$ and $\mathcal{P}_m$ of length $t$ and $m$ we will define the following notions. The reverse path $\mathcal{P}^{-1}_t$ is defined as: \begin{equation}
 \begin{split}
&V_{\mathcal{P}_t}=V_{\mathcal{P}^{-1}_t}=(t+1) \wedge \forall i<t,\\
&\hspace{30pt}E_{\mathcal{P}^{-1}_t}(i,i+1)\leftrightarrow E_{\mathcal{P}_t}(t-i,t-(i+1))\wedge\\
& \hspace{60pt}\wedge E_{\mathcal{P}^{-1}_t}(i+1,i)\leftrightarrow E_{\mathcal{P}_t}(t-(i+1),t-i).
 \end{split}
\end{equation}
The glued path $\mathcal{P}_t\circ \mathcal{P}_m$ is defined as:
\begin{equation}
 \begin{split}
&V_{\mathcal{P}_t\circ \mathcal{P}_m}=(t+m+1)\wedge\\
&\hspace{20pt}\wedge \forall i<t,\,E_{\mathcal{P}_t\circ \mathcal{P}_m}(i,i+1) \leftrightarrow E_{\mathcal{P}_t}(i,i+1)\wedge\\
&\hspace{40pt}\wedge E_{\mathcal{P}_t\circ \mathcal{P}_m}(i+1,i)\leftrightarrow E_{\mathcal{P}_t}(i+1,i)\\
&\hspace{60pt}\wedge \forall t\leq j<(t+m),\\
&\hspace{80pt}E_{\mathcal{P}_t\circ \mathcal{P}_m}(j,j+1)\leftrightarrow E_{\mathcal{P}_m}(j-t,j+1-t))\wedge \\
&\hspace{100pt}\wedge E_{\mathcal{P}_t\circ \mathcal{P}_m}(j+1,j)\leftrightarrow E_{\mathcal{P}_m}(j+1-t,j-t)).
 \end{split}
\end{equation}

We say that there is a path from $i$ to $j$ in the input digraph $\mathcal{X}$ if there exists a path $\mathcal{P}_t$ of some length $t$ that can be homomorphically mapped to $\mathcal{X}$ such that $H(0)=i$ and $H(t)=j$:
\begin{equation}
 \begin{split}
& Path(i,j,\mathcal{X}) \iff \exists t<n, \exists V_{\mathcal{P}_t}=t, \exists E_{\mathcal{P}_t}\leq 4t^2,\,PATH(V_{\mathcal{P}_t},E_{\mathcal{P}_t})\wedge\\ &\hspace{30pt}\wedge\exists H\leq \langle t,n \rangle, HOM(\mathcal{P}_t,\mathcal{X},H)\wedge(H(0,i)\wedge H(t,j)).
 \end{split}
\end{equation}

We say that the path $\mathcal{P}_t$ connects $i$ and $j$. Also, we can encode what it means to be linked for two elements $a\in D_i,b\in D_j$. In words, there must exist a path $\mathcal{P}_t$ of some length $t$ connecting $i,j$ with homomorphism $H$ such that there exists a homomorphism $H'$ from $\mathcal{P}_t$ to $\ddot{\mathcal{A}}$ sending $0$ to $\langle i,a\rangle$ and $t$ to $\langle j,b\rangle$, and for every element $p<t$, $H(p)=k$ implies that $H(p)=\langle k,c\rangle$ for some $c\in D_k$. We can express it by the $\Sigma^{1,b}_1$-formula
\begin{equation}\label{JJKTRE6}
\begin{split}
&\hspace{35pt}Linked(a,b,i,j,\Theta) \iff \exists t<nl, \exists V_{\mathcal{P}_t}=t, \exists E_{\mathcal{P}_t}\leq 4t^2,\\
&\exists H\leq \langle t,n \rangle,\,PATH(V_{\mathcal{P}_t},E_{\mathcal{P}_t})\wedge HOM(\mathcal{P}_t,\mathcal{X},H)\wedge(H(0,i)\wedge H(t,j))\wedge\\
&\hspace{70pt}\wedge\exists H'\leq \langle t,\langle t,l\rangle\rangle,\ddot{HOM}(\mathcal{P}_t,\ddot{\mathcal{A}},H')\wedge\\
&\hspace{40pt}\wedge(\forall k<n,p<t,\,(H(p,k)\rightarrow \exists c\in D_k,\,H'(p)=\langle k,c\rangle))\\
&\hspace{80pt}\wedge H'(0)=\langle i,a\rangle \wedge H'(t)= \langle j,b\rangle.
\end{split}
\end{equation}
\begin{notation}
Sometimes we will write 
$\exists \mathcal{P}_t< (n,4n^2), Path(i,j,\mathcal{X},\mathcal{P}_t)$, and  $\exists \mathcal{P}_t< ( nl,4(nl)^2),Linked(a,b,i,j,\Theta,\mathcal{P}_t)$ to omit repetitions. 
\end{notation}

\begin{definition}[Linked instance]
We say that an instance $\Theta = (\mathcal{X}, \ddot{\mathcal{A}})$ with $V_{\mathcal{X}}=n$, $V_{\ddot{\mathcal{A}}}<\langle n,l \rangle$ is linked if it satisfies the following $\Sigma^{1,b}_1$-relation
\begin{equation}\label{HOMINSTAG}
\begin{split}
&\hspace{10pt}LNKD(\mathcal{X}, \ddot{\mathcal{A}}) \iff\forall i<n, \forall a,b\in D_{i},\, Linked(a,b,i,i,\Theta).
\end{split}
\end{equation}
\end{definition}
To define irreducibility we need to encode a fragmented instance and a subdirect solution set. 

\begin{definition}[Fragmented instance]
A fragmented instance is an instance whose input digraph $\mathcal{X}$ is not connected. For an instance $\Theta = (\mathcal{X},  \ddot{\mathcal{A}})$ with $V_{\mathcal{X}}=n$ we define this by the following $\Sigma^{1,b}_1$-definable relation, where $PSS$ encodes a proper subset.
\begin{equation}\label{HOMINSTAG}
 \begin{split}
&\hspace{30pt}FRGM(\mathcal{X}, \ddot{\mathcal{A}}) \iff \exists V^1_\mathcal{X}, \exists V^2_\mathcal{X}, V^1_\mathcal{X}=V^2_\mathcal{X} = n\wedge \\
&\wedge PSS(V^1_\mathcal{X},V_\mathcal{X})\wedge PSS(V^2_\mathcal{X},V_\mathcal{X})\wedge (\forall i<n,\,V^1_\mathcal{X}(i)\leftrightarrow \neg V^2_\mathcal{X}(i))\wedge \\
&\hspace{40pt}\wedge \forall i\in V^1_\mathcal{X},\forall j\in V^2_\mathcal{X},\,\,\neg E_{\mathcal{X}}(i,j)\wedge \neg E_{\mathcal{X}}(j,i).
 \end{split}
\end{equation}
\end{definition}

We say that the instance $\Theta = (\mathcal{X},  \ddot{\mathcal{A}})$ has a subdirect solution set if there is a solution to the instance for all $a\in D_i$, $i\in\{0,...,n-1\}$. It can be expressed by the $\Sigma^{1,b}_1$-formula
\begin{equation}\label{HOMINSTAG}
 \begin{split}
&SSS(\mathcal{X}, \ddot{\mathcal{A}})\iff \forall i<n\forall a\in D_i, \exists H' < \langle n,\langle n,l\rangle\rangle\\
&\hspace{40pt}\ddot{HOM}(\mathcal{X},\ddot{\mathcal{A}},H)\wedge H(i)=\langle i,a \rangle.
 \end{split}
\end{equation}

Now we are ready to define irreducibility. 
\begin{definition}[Irreducible instance]
We say that an instance $\Theta = (\mathcal{X}, \ddot{\mathcal{A}})$ with $V_{\mathcal{X}}=n$, $V_{\ddot{\mathcal{A}}}<\langle n,l \rangle$ is irreducible if any its subinstance is fragmented, or linked, or its solution set is subdirect. To express it we use the $\Pi^{1,b}_2$-formula 
\begin{equation}\label{HOMINSTAG}
 \begin{split}
&\hspace{20pt}IRD(\mathcal{X}, \ddot{\mathcal{A}})\iff \forall \mathcal{X'} = (V_{\mathcal{X'}},  E_{\mathcal{X'}}), \forall V_{\mathcal{X'}}=n, \forall E_{\mathcal{X'}}<4n^2,\\ 
&\hspace{0pt}\big(SINST(\mathcal{X}',\mathcal{X})\rightarrow FRGM(\mathcal{X'}, \ddot{\mathcal{A}})\vee LNKD(\mathcal{X'}, \ddot{\mathcal{A}})\vee SSS(\mathcal{X'}, \ddot{\mathcal{A}})      \big).
 \end{split}
\end{equation}
\end{definition}

Finally, we will introduce the relation indicating that $\Theta=(\mathcal{X},\ddot{\mathcal{A}})$ is an instance of CSP$(\Gamma_{\mathcal{A}})$ for the relational structure $\mathcal{A}=(A,\Gamma_{\mathcal{A}})$. Since $\Gamma_{\mathcal{A}}$ contains at most binary relations and is closed under $pp$-definition, we indeed can identify all constraints posed on variables $x_i,x_j$ with two unary relations (domains $D_i, D_j$) and one binary relation $E^{ij}_{\ddot{\mathcal{A}}}$ from the list.
\begin{definition}
A pair of sets $\Theta=(\mathcal{X},\ddot{\mathcal{A}})$ is a CSP instance over constraint language $\Gamma_{\mathcal{A}}$ on $A$ of size $l$ if the following $\Sigma^{1,b}_0$-relation is true.
\begin{equation}
    \begin{split}
        &INST(\Theta,\Gamma_{\mathcal{A}})\iff DG(\Theta)\wedge \forall i<n, |D_i|=l\wedge\\
        &\hspace{20pt}\wedge \forall i,j<n,a,b<l, \exists s <|\Gamma_{\mathcal{A}}|, E_{\ddot{\mathcal{A}}}(\langle i,a \rangle,\langle j,b \rangle)\leftrightarrow \Gamma_{\mathcal{A}}^2(s,a,b)\wedge\\
        &\hspace{120pt}\wedge\forall i<n,a<l,\exists s<|\Gamma_{\mathcal{A}}|, D_i(a)\leftrightarrow \Gamma_{\mathcal{A}}^1(s,a).
    \end{split}
\end{equation}

\end{definition}

\subsection{Universal algebra axiom schemes}\label{66347he}
In this subsection we will encode absorbing and central subuniverses and polynomially complete algebras in $V^1$, and formulate three universal algebra axioms reflecting the "only if" implications of Theorems \ref{fhfhfhydj}, \ref{llldju67yr} and \ref{fjfjhgd} (for the soundness we do not need the "if" implication). For this subsection we will consider CSP instances alongside the corresponding algebras and suppose that any algebra is finite and has a special WNU term. 

\subsubsection{Binary absorption axiom scheme}
Consider any algebra $\mathbb{A}=(A,\Omega)$ and its subalgebra $\mathbb{B}=(B,\Omega)$, where $\Omega$ is $m$-ary basic operation. Suppose that the corresponding relational structure to $\mathbb{A}$ is $\mathcal{A}=(A,\Gamma_{\mathcal{A}})$, and $\Gamma_{\mathcal{A}}$ is a relational clone. Due to Galois correspondence, $Clone(\Omega)=Pol(\Gamma_{\mathcal{A}})$. Thus, for any binary term operation $T$ over $A$ the condition $T\in Clone(\Omega)$ can be encoded as: 
\begin{equation}
T\in Clone(\Omega) \iff Pol_{2}(T,A, \Gamma_{\mathcal{A}}).
\end{equation}
For any three sets $A,B,T$ the following $\Sigma^{1,b}_0$-definable relation indicates that the subset $B$ absorbs $A$ with binary operation $T$:
\begin{equation}
    \begin{split}
        &BAS(B,A,T)\iff SS(B,A)\wedge \forall a\in A, \forall b\in B, \exists c_1,c_2\in B,\\
        &\hspace{65pt}T(a,b)=c_1\wedge T(b,a)=c_2.
    \end{split}
\end{equation}
We will formalize the "only if" implication of Theorem \ref{fhfhfhydj} in the theory $V^1$ as Binary absorption axioms, BA-axioms. For any algebra $\mathbb{A}=(A,\Omega)$ corresponding to the constraint language $\Gamma_{\mathcal{A}}$ of a CSP instance, it is enough to consider only finitely many axioms since there are finitely many subalgebras $\mathbb{D}$ of $\mathbb{A}$ and finitely many strong subsets $B$ of $\mathbb{D}$).

\begin{definition}[BA$_\mathcal{A}$-axioms]
For any constraint language $\Gamma_{\mathcal{A}}$ over set $A$ of size $l$, fixed algebra $\mathbb{A}=(A,\Omega)$ with $\Omega$ being an $m$-ary special WNU operation, and finitely many subuniverses $D$ of $\mathbb{A}$ and binary absorbing subuniverses $B$ of $\mathbb{D}$ the \emph{binary absorption axiom scheme} is denoted by BA$_\mathcal{A}$-axioms and consists of the finitely many formulas of the following form 
\begin{equation}
\begin{split}
&\hspace{30pt}BA_{\mathcal{A},B,D}=_{def}\forall\mathcal{X}=(V_{\mathcal{X}},E_{\mathcal{X}}),\forall\ddot{\mathcal{A}}=(V_{\ddot{\mathcal{A}}},E_{\ddot{\mathcal{A}}},D_0,...,D_{n-1}),\\
&\hspace{45pt}\big(PSS(B,D)\wedge SwNU_m(\Omega,D)\wedge SwNU_m(\Omega,B)\wedge\\
&\hspace{55pt}\wedge\exists T <(3l)^{2^3}, Pol_{2}(T,D, \Gamma_{\mathcal{A}})\wedge BAS(B,D,T) \wedge\\
&\hspace{55pt}\wedge  INST(\Theta,\Gamma_{\mathcal{A}})\wedge CC(\mathcal{X},\ddot{\mathcal{A}})\wedge IRD(\mathcal{X},\ddot{\mathcal{A}})\wedge \\
&\hspace{125pt}\exists i<n, D_i=D\wedge\\ 
&\hspace{20pt}HOM(\mathcal{X},\ddot{\mathcal{A}})\big)\rightarrow HOM(\mathcal{X},\ddot{\mathcal{A}}=(V_{\ddot{\mathcal{A}}},E_{\ddot{\mathcal{A}}},D_0,...,B,...,D_{n-1})).
\end{split}
\end{equation}
\end{definition}
Variables here are an input digraph $\mathcal{X}$ with $V_{\mathcal{X}}=n$ and a target digraph with domains $\ddot{\mathcal{A}}$, $\Theta$ stands for $(\mathcal{X},\ddot{\mathcal{A}})$. The second line of the formula ensures that $B$ is a proper subset of $D$ and both $B$ and $D$ are closed under $\Omega$ (relation $SwNU_m$), i.e. both are subuniverses. The third line claims that there exists a binary operation $T$ defined on the subuniverse $D$ and compatible with all relations from $\Gamma_{\mathcal{A}}$ such that $B$ absorbs $D$ with $T$. The fourth line says that $\Theta$ is a CSP instance over constraint language $\Gamma_{\mathcal{A}}$, and this instance is cycle-consistent and irreducible. Finally, the rest of the formula says that if $D$ coincides with a domain $D_i$ for some variable $i$, all the above-mentioned conditions hold and there is a solution to the instance $\Theta$, then there is a solution to the instance $\Theta$ with $D_i$ restricted to $B$.

In strict form (with all string quantifiers occurring in front) and after regrouping them in such a way that all universal quantifiers will precede existential ones, we will eventually get the universal closure of $\Sigma^{1,b}_{2}$-formula.

\subsubsection{Central subuniverse axiom scheme}
We will formalize the "only if" implication of Theorem \ref{llldju67yr} not for a center, but for a central subuniverse. Recall that a central subuniverse has all the good properties of a center, and we will use it in the algorithm instead of the latter. To define a central subuniverse $C$ of an algebra $\mathbb{A}=(A,\Omega)$ we need to encode a set $Sg$ for the subset $X=\{\{a\}\times C,C\times \{a\}\}$ of $A^2$ for any $a\in A$. Recall that $Sg(X)$ can be constructed by the closure operator 
\begin{equation}
\begin{split}
&E(X)=X\cup\{\Omega(a_1,...,a_m): a_1,...,a_m\in X\}\\
&\hspace{5pt}\forall t\geq 0, E^0(X)=X, E^{t+1}(X) = E(E^t(X)).
\end{split}
\end{equation}
Since $\mathbb{A}$ is finite of size $l$ and $|X|=2|C|$, we do not need more than $(l^2-2|C|)$ applications of the closure operator $E$ since at every application we either add to the set at least one element or after some $t$, $E^t(X)=E^{t+r}(X)$ for any $r$. Not to depend on $C$, let us choose the value $l^2$. Thus, for any set $X\leq \langle l,l\rangle$, we will iteratively define the following set $E^{l^2}_X$ up to $l^2$
\begin{equation}
\begin{split}
&\hspace{60pt}\forall b,c<l,\,E^{0}_X(b,c)\iff X(b,c)\wedge\\
&\hspace{25pt}\wedge \forall 0<t<l^2, \forall b,c<l,\, E^{t}_X(b,c)\iff E^{t-1}_X(b,c)\vee\\
&\vee\exists b_1,...,b_m,c_1,...,c_m\in A,E^{t-1}_X(b_1,c_1)\wedge...\wedge  E^{t-1}_X(b_m,c_m)\wedge\\
&\hspace{50pt}\wedge\Omega(b_1,...,b_m)=b\wedge \Omega(c_1,...,c_m)=c.
\end{split}
\end{equation}

The existence of this set follows from $\Sigma^{1,b}_1$-induction. A central subuniverse must be an absorbing subuniverse, namely a ternary absorbing subuniverse \cite{DBLP:journals/mvl/Zhuk21}. For any three sets $A,C,S$ the following $\Pi^{1,b}_1$-definable relation expresses that the subset $C$ of $A$ is central under ternary term operation $S$.
\begin{equation}
    \begin{split}
&\hspace{10pt}CRS(C,A,S)\iff SS(C,A)\wedge\forall c_1,c_2\in C, \forall a\in A, \exists c'_1,c'_2,c'_3\in C,\\
&\hspace{30pt}S(c_1,c_2,a)=c'_1\wedge S(c_1,a,c_2)=c'_2\wedge S(a,c_1,c_2)=c'_3\wedge\\
&\hspace{0pt}\wedge\forall a\in A\backslash C, \forall X<\langle l,l\rangle,\, ((X(a,c)\wedge  X(c,a)\leftrightarrow c\in C) \rightarrow \neg E^{l^2}_X(a,a)).
    \end{split}
\end{equation}

\begin{definition}[CR$_\mathcal{A}$-axioms]
For any constraint language $\Gamma_{\mathcal{A}}$ over set $A$ of size $l$, fixed algebra $\mathbb{A}=(A,\Omega)$, with $\Omega$ being an $m$-ary special WNU operation, and finitely many subuniverses $D$ of $\mathbb{A}$ and central subuniverses $C$ of $\mathbb{D}$ we denote the \emph{central subuniverse axiom scheme} by CR$_\mathcal{A}$-axioms. The scheme embraces the finitely many formulas of the following form
\begin{equation}
\begin{split}
&\hspace{25pt}CR_{\mathcal{A},D,C}=_{def}\forall \mathcal{X}=(V_{\mathcal{X}},E_{\mathcal{X}}),\forall\ddot{\mathcal{A}}=(V_{\ddot{\mathcal{A}}},E_{\ddot{\mathcal{A}}},D_0,...,D_{n-1}),\\
&\hspace{40pt}\big(PSS(C,D)\wedge SwNU_m(\Omega,D)\wedge SwNU_m(\Omega,C)\wedge\\
&   \hspace{45pt}\ \exists S<(4l)^{2^4},\,Pol_3(S, D,\Gamma_{\mathcal{A}})\wedge CRS(C,D,S) \wedge  \\
&\hspace{50pt}\wedge  INST(\Theta,\Gamma_{\mathcal{A}})\wedge CC(\mathcal{X},\ddot{\mathcal{A}})\wedge IRD(\mathcal{X},\ddot{\mathcal{A}})\wedge \\
&\hspace{125pt}\exists i<n, D_i=D\wedge\\ 
&\hspace{20pt}HOM(\mathcal{X},\ddot{\mathcal{A}})\big)\rightarrow HOM(\mathcal{X},\ddot{\mathcal{A}}=(V_{\ddot{\mathcal{A}}},E_{\ddot{\mathcal{A}}},D_0,...,C,...,D_{n-1})).
\end{split}
\end{equation}
\end{definition}
The formula is analogous to BA$_{\mathcal{A}}$-axioms, it is again the universal closure of $\Sigma^{1,b}_2$-formula and the only line that differs is the third one: it claims that there exists a ternary term operation $S$ defined on subuniverse $D$ and compatible with all relations from $\Gamma_{\mathcal{A}}$ such that $C$ is a central subuniverse under $S$.  

\subsubsection{Polynomially complete axiom scheme}
Theorem \ref{slsl88yhdurh} claims that a finite algebra is polynomially complete if and only if it has the ternary discriminator as a polynomial operation. Consider an algebra $\mathbb{A}=(A,\Omega)$. The clone of all polynomials over $\mathbb{A}$, $Polynom(\mathbb{A})$ is defined as the clone generated by $\Omega$ and all constants on $A$, i.e. nullary operations:
\begin{equation}
Polynom(\mathbb{A}) = Clone(\Omega,a_1,...,a_{|A|}).
\end{equation}
Constants as nullary operations with constant values, composed with $0$-many $n$-ary operations are $n$-ary operations with constant values. Thus, to be preserved by all constants operations, any unary relation has to contain the whole set $A$, and any binary relation has to contain the diagonal relation $\Delta_{A}$. We can impose these conditions on the set $\Gamma_{\mathcal{A}}$. For the algebra $\mathbb{A}$ denote by $\Gamma^{diag}_{\mathcal{A}} = (\Gamma^{1,diag}_{\mathcal{A}}, \Gamma^{2,diag}_{\mathcal{A}})$ the pair of sets such that
\begin{equation}
\begin{split}
&\hspace{10pt}\Gamma^{1,diag}_{\mathcal{A}}(j,a)\iff \Gamma^{1}_{\mathcal{A}}(j,a)\wedge (\forall b\in A, \Gamma^{1}_{\mathcal{A}}(j,b))\\
&\Gamma^{2,diag}_{\mathcal{A}}(i,a,b)\iff \Gamma^{2}_{\mathcal{A}}(i,a,b)\wedge (\forall c\in A, \Gamma^{2}_{\mathcal{A}}(j,c,c)).
\end{split}
\end{equation}
An $n$-ary operation $P$ on algebra $\mathbb{A}$ is a polynomial operation if it is a polymorphism for relations from $\Gamma^{diag}_{\mathcal{A}}$, i.e.
\begin{equation}
P\in Polynom(\mathbb{A})\iff Pol_{n}(P,A,\Gamma^{diag}_{\mathcal{A}}).
\end{equation}
For any two sets $A$ and $P$ the following $\Sigma^{1,b}_0$-definable relation claims that $P$ is a ternary discriminator on $A$:
\begin{equation}
    \begin{split}
        &\hspace{40pt}PC(A,P)\iff \forall a,b,c\in A, \\
        &(a=b\wedge P(a,b,c)=c)\vee(a\neq b\wedge P(a,b,c)=a).
    \end{split}
\end{equation}
Before the formalization of the "only if" implication of Theorem \ref{fjfjhgd} as the polynomially complete axiom scheme, we need to encode one more notion. For any congruence $\sigma$ on algebra $\mathbb{A}=(A,\Omega)$, for factor algebra $\mathbb{A}/\sigma$ we will define the quotient set of relation $\Gamma_{\mathcal{A}}/\sigma$ as follows:
\begin{equation}
\begin{split}
&\hspace{0pt}\Gamma^{1}_{\mathcal{A}}/\sigma(j,a)\iff \forall [a]/\sigma\in A,\, Rep_m(a,[a]/\sigma,A/\sigma,A,\Omega,\sigma)\wedge\Gamma^{1}_{\mathcal{A}}(j,[a]/\sigma)\\
&\hspace{40pt}\Gamma^{2}_{\mathcal{A}}/\sigma(i,a,b)\iff \forall [a]/\sigma,[b]/\sigma\in A,\,\Gamma^{2}_{\mathcal{A}}(i,[a]/\sigma,[b]/\sigma)\wedge\\
&\hspace{30pt}\wedge Rep_m(a,[a]/\sigma,A/\sigma,A,\Omega,\sigma)\wedge Rep_m(b,[b]/\sigma,A/\sigma,A,\Omega,\sigma).
\end{split}
\end{equation}
The definition follows from log-space reduction from CSP($\mathbb{A}/\sigma$) to CSP($\mathbb{A}$). Note, that for some $i,j$,  $\Gamma^{1}_{\mathcal{A},j}/\sigma$ and $\Gamma^{2}_{\mathcal{A},i}/\sigma$ are empty sets, as well as $\Gamma^{1,diag}_{\mathcal{A},j}$ and $\Gamma^{2,diag}_{\mathcal{A},i}$.
\begin{definition}[PC$_\mathcal{A}$-axioms] For any constraint language $\Gamma_{\mathcal{A}}$ over set $A$ of size $l$, fixed algebra $\mathbb{A}=(A,\Omega)$ with $\Omega$ being an $m$-ary special WNU operation, and finitely many subuniverses $D$ of $\mathbb{A}$ and congruence blocks $E$ of $\mathbb{D}$ the \emph{polynomially complete axiom scheme} is denoted by PC$_\mathcal{A}$-axioms and consists of the finitely many formulas of the following form
\begin{equation}
\begin{split}
&\hspace{25pt}PC_{\mathcal{A},D,E}=_{def}\forall \mathcal{X}=(V_{\mathcal{X}},E_{\mathcal{X}}),\forall\ddot{\mathcal{A}}=(V_{\ddot{\mathcal{A}}},E_{\ddot{\mathcal{A}}},D_0,...,D_{n-1}),\\ 
&\hspace{10pt}\big([\forall j<n,\forall B<l,\forall T<(3l)^{2^3}, Pol_2(T,D_j, \Gamma_{\mathcal{A}})\rightarrow\neg BAS(B,D_j,T)\wedge\\
&\hspace{13pt}\wedge\forall j<n,\forall C<l, \forall S<(4l)^{2^4}, Pol_3(S,D_j, \Gamma_{\mathcal{A}})\rightarrow\neg CRS(C,D_j,S)]\\
&\hspace{10pt}\wedge \exists \sigma<\langle l,l\rangle, \exists D/\sigma<l,\exists\Omega/\sigma<(ml)^{2^{m+1}}, FA_m(D/\sigma,\Omega/\sigma,D,\Omega,\sigma)\wedge\\
&\hspace{40pt}\wedge \exists P<(4l)^{2^4},\, Pol_{3}(P,D/\sigma,\Gamma^{diag}_{\mathcal{D}}/\sigma)\wedge PC(D/\sigma,P)\wedge \\
&\hspace{15pt}SwNU_m(\Omega,D)\wedge PSS(E,D)\wedge (\forall a\in E,\forall b\in D, \sigma(a,b)\leftrightarrow b\in E)\wedge\\
&\hspace{50pt}\wedge  INST(\Theta,\Gamma_{\mathcal{A}})\wedge CC(\mathcal{X},\ddot{\mathcal{A}})\wedge IRD(\mathcal{X},\ddot{\mathcal{A}})\wedge \\
&\hspace{125pt}\exists i<n, D_i=D\wedge\\ 
&\hspace{20pt}HOM(\mathcal{X},\ddot{\mathcal{A}})\big)\rightarrow HOM(\mathcal{X},\ddot{\mathcal{A}}=(V_{\ddot{\mathcal{A}}},E_{\ddot{\mathcal{A}}},D_0,...,E,...,D_{n-1})).
\end{split}
\end{equation}
\end{definition}
In this $\forall\Sigma^{1,b}_2$-formula, the first and the second lines in square brackets say that for any domain $D_j$ of instance $\Theta$ there are no binary absorbing or central subuniverses. The fourth and fifth lines claim that there exists a congruence $\sigma$ on $D$ and the corresponding factor algebra $\mathbb{D}/\sigma=(D/\sigma,\Omega/\sigma)$ such that this factor algebra is polynomially complete. Note that we define a discriminator $P$ on factor set $D/\sigma$, and require that $P$ is a polymorphism for all relations from the quotient set of relation $\Gamma^{diag}_{\mathcal{D}}/\sigma$. The sixth line says that $D$ is closed under $\Omega$, $E$ is a proper subset of $D$ and $E$ is a congruence class of $\sigma$. Finally, the rest of the formula says that if $D$ coincides with a domain $D_i$ for some variable $i$, all the above-mentioned conditions hold and there is a solution to the instance $\Theta$, then there is a solution to the instance $\Theta$ with $D_i$ restricted to the congruence class $E$. 

\subsection{A new theory of bounded arithmetic}\label{aaoiisoiowie6}
For any relational structure $\mathcal{A}$ let us define a new theory of bounded arithmetic extending the theory $V^1$, as follows.
\begin{definition}[Theory $V^1_{\mathcal{A}}$]
$$V^1_{\mathcal{A}}=_{def}V^1+\{\text{BA$_{\mathcal{A}}$-axioms, CR$_{\mathcal{A}}$-axioms, PC$_{\mathcal{A}}$-axioms}\}.
$$
Each of the universal algebra axiom schemes BA$_{\mathcal{A}}$-axioms, CR$_{\mathcal{A}}$-axioms, and PC$_{\mathcal{A}}$-axioms consists of a finitely many $\forall\Sigma^{1,b}_2$-formulas for the fixed finite algebra $\mathbb{A}=(A,\Omega)$ with a special WNU operation $\Omega$ corresponding to the relational structure $\mathcal{A}=(A,\Gamma_{\mathcal{A}})$.
\end{definition}
We are going to show that for any structure $\mathcal{A}$ which leads to $p$-time solvable CSP, the theory $V^1_{\mathcal{A}}$ proves the soundness of Zhuk's algorithm.

\subsection{Consistency reductions}\label{kkspqu5465}
Consistency reductions of Zhuk's algorithm precede all other reductions and the linear case and include cycle-consistency reduction (function CheckCycleConsistency), irreducibility reduction (function CheckIrreducibility) and weaker instance reduction (function CheckWeakerInstance), see \cite{10.1145/3402029}. Consider a CSP instance $\Theta = (\mathcal{X}, \ddot{\mathcal{A}})$ with domain set $D=\{D_0,...,D_{n-1}\}$. During consistency reductions the algorithm works with some modifications of an input digraph $\mathcal{X}$ and a target digraph with domains $\ddot{\mathcal{A}}$. At the end of every procedure, the output is either "No solution" (some domain is empty after reduction), or "OK" (the algorithm cannot reduce any domain since the instance satisfies the property we are checking), or the reduction $(i, D_i')$ of the first domain in a line that we can reduce. 

At the beginning of every procedure, for simplicity we will refer to every input instance as the initial one, $\Theta = (\mathcal{X}, \ddot{\mathcal{A}})$. It makes sense: after every reduction $(i, D_i')$ we start the algorithm all from the beginning with the same input digraph (the same set of variables and the same set of constraints) but with a smaller domain set $D'=\{D_0,..., D'_i,..., D_{n-1}\}$: we remove some vertices from $\ddot{\mathcal{A}}$, which induces removing some edges. If the algorithm moves to another procedure, it means that the previous one cannot reduce any domain - so technically, we proceed with the same instance from the beginning of the current step of recursion. 

\subsubsection{Cycle-consistency}
In this section we will formalize the modification of the function CheckCycleConsistency suggested by Zhuk in his latter paper \cite{DBLP:journals/mvl/Zhuk21}. In short, the algorithm first intersects all constraints making an instance $1$-consistent, and then uses constraint propagation to ensure a type of consistency called $(2,3)$-consistency. In words, $(2,3)$-consistency means that for any variables $i,j,k$ every edge $(i,j)$ extends to a triangle by edges $(i,k)$ and $(k,j)$. These two properties taken together provide cycle-consistency. We explain the procedure in detail alongside the formalization. 

Consider a CSP instance $\Theta=(\mathcal{X},\ddot{\mathcal{A}})$. First, for any two variables $i,j$ the algorithm defines a full relation $R_{i,j}$ on domains $D_i\times D_j$. We define a new target digraph with domains $\ddot{\mathcal{R}} = (V_{\ddot{\mathcal{R}}},E_{\ddot{\mathcal{R}}},$ $D_0,...,$ $D_{n-1})$, where $V_{\ddot{\mathcal{R}}} = V_{\ddot{\mathcal{A}}}$, but while 

\begin{equation}
    \begin{split}
       &E_{\ddot{\mathcal{A}}}(u,v)\longrightarrow \exists i,j<n\,\exists a,b<l\,\, u = \langle i,a\rangle\wedge v = \langle j,b\rangle\wedge\\
       &\hspace{80pt} D_i(a)\wedge D_j(b),
    \end{split}
\end{equation}
for $E_{\ddot{\mathcal{R}}}$ we have
\begin{equation}
    \begin{split}
       &E_{\ddot{\mathcal{R}}}(u,v)\iff \exists i,j<n\,\exists a,b<l\,\, u = \langle i,a\rangle\wedge v = \langle j,b\rangle\wedge\\
       &\hspace{80pt} D_i(a)\wedge D_j(b).
    \end{split}
\end{equation}
That is, for all $i,j\in\{0,...,n-1\}$, $E^{ij}_{\ddot{\mathcal{R}}}$ is the full binary relation on $D_i\times D_j$ (even for those $i,j$, for which $\neg E_{\mathcal{X}}(i,j)$ and $\neg E_{\mathcal{X}}(j,i)$).

Then for all $i,j\in\{0,...,n-1\}$ the algorithm intersects each $E^{ij}_{\ddot{\mathcal{R}}}$ with projections of all constraints onto the variables $i,j$. In our case, for $i,j$ we have only constraints $D_i, D_j$, $E^{ij}_{\ddot{\mathcal{A}}}$, and $E^{ji}_{\ddot{\mathcal{A}}}$, i.e. we intersect $E^{ij}_{\ddot{\mathcal{R}}}$ only with $E^{ij}_{\ddot{\mathcal{A}}}$ and $E^{ji}_{\ddot{\mathcal{A}}}$. Let us denote new relations by $E^{ij}_{\ddot{\mathcal{R}}_{0}}$:
\begin{equation}\label{IINT222}
\begin{split}
&\hspace{30pt}E^{ij}_{\ddot{\mathcal{R}}_{0}}(a,b) \iff (a\in D_i\wedge b\in D_j)\wedge \\ &\wedge(E_{\mathcal{X}}(i,j) \rightarrow E^{ij}_{\ddot{\mathcal{A}}}(a,b))\wedge (E_{\mathcal{X}}(j,i) \rightarrow E^{ji}_{\ddot{\mathcal{A}}}(b,a)).
\end{split}
\end{equation}

Note that if there are no constraints $E_{\mathcal{X}}(i,j)$ and $E_{\mathcal{X}}(j,i)$, then at this point both $E^{ij}_{\ddot{\mathcal{R}}_{0}}$ and $E^{ji}_{\ddot{\mathcal{R}}_{0}}$ are still $D_i\times D_j$, $D_j\times D_i$. Then denote by $Pr_{1}(i,a)$ the intersection of the projections of all constraints $E^{ij}_{\ddot{\mathcal{R}}_{0}}$ on variable $i$:
\begin{equation}\label{IINT11111}
    \begin{split}
       &Pr_{1}(i,a)\iff a\in D_i\wedge \forall j<n, \, E_{\mathcal{X}}(i,j) \rightarrow \exists b \in D_j,\,\, E^{ij}_{\ddot{\mathcal{R}}_{0}}(a,b)\wedge\\
       &\hspace{50pt}\forall k<n, \, E_{\mathcal{X}}(k,i) \rightarrow \exists c \in D_k,\,\, E^{ki}_{\ddot{\mathcal{R}}_{0}}(c,a).
    \end{split}
\end{equation}
Let us define a new ($1$-consistent) digraph $\ddot{\mathcal{R}}_{1}$ with domains by setting
\begin{equation}\label{IINT3333}
V_{\ddot{\mathcal{R}}_1}(i,a)\iff Pr_{1}(i,a),
\end{equation}
and 
\begin{equation}\label{IINT3333}
E^{ij}_{\ddot{\mathcal{R}}_{1}}(a,b) \iff  Pr_{1}(i,a)\wedge Pr_{1}(j,b)\wedge  E^{ij}_{\ddot{\mathcal{R}}_{0}}(a,b).
\end{equation}

Then the algorithm produces iterative propagation of constraints until it cannot change any further relation. For every step of propagation $t>1$, for all $i,j\in\{0,...,n-1\}$ we define a new set $R_t$ as follows:
\begin{equation}\label{propagationlsd}
R_1(i,j,a,b)\iff E^{ij}_{\ddot{\mathcal{R}}_{1}}(a,b),
\end{equation}
and for $t>1$
\begin{equation}\label{kkofiru67}
\begin{split}
&\hspace{45pt}R_t(i,j,a,b) \iff R_{t-1}(i,j,a,b) \wedge\\
&\forall k<n \exists c<l\,\, Pr_{1}(k,c)\wedge(R_{t-1}(i,k,a,c) \wedge R_{t-1}(k,j,c,b)).
\end{split}
\end{equation}

The existence of this set is ensured by $\Sigma^{1,b}_1$-induction. For every step of propagation $t>1$, $R_t(i,j,a,b)$ corresponds to the relation $E^{ij}_{\ddot{\mathcal{R}}_{t}}$ and thus induces the next digraph with domains $\ddot{\mathcal{R}}_{t}$. The process will eventually stop since on every step $t>1$ we remove some edges from $\ddot{\mathcal{R}}_{t-1}$, and the number of edges in $\ddot{\mathcal{R}}_{1}$ is bounded by some polynomial of $n$ and $l$, $p(n,l)$. Let us prove it. 

Denote the number of edges in $\ddot{\mathcal{R}}_{1}$ by $q=\verb|#| E_{\ddot{\mathcal{R}}_{1}}$, i.e. the number of elements in $R_1(i,j,a,b)$ is $q$. For every $t\leq (q+1)$ due to definition $\forall i,j<n,\forall a,b<k,\,\,R_t(i,j,a,b) \rightarrow R_{t-1}(i,j,a,b)$. Suppose that for some $t=q'<q+1$ we have 
$$\forall i,j<n,\forall a,b<l,\,\,R_{q'}(i,j,a,b) \iff R_{q'-1}(i,j,a,b).
$$
Then it means that the part
$$\forall k<n \exists c<l,\,\, Pr_{1}(k,c)\wedge(R_{t-1}(i,k,a,c) \wedge R_{t-1}(k,j,c,b))
$$
is always true when $t=q'$. By induction on $s$ we can prove that in this case 
$$\forall i,j<n,\forall a,b<l,\,\,R_{q'+s}(i,j,a,b) \iff R_{q'-1}(i,j,a,b)
$$
since for $s=0$ it is a suggestion, and if it is true for $s=f$, then we can rewrite the definition of $R_{q'+f+1}$ using equivalent sets
\begin{equation}\label{CONSTRUCTIONCC}
\begin{split}
&\hspace{40pt}R_{q'+f+1}(i,j,a,b) \iff R_{q'-1}(i,j,a,b) \wedge\\
&\forall k<n \exists c<l,\,\, Pr_{1}(k,c)\wedge(R_{q'-1}(i,k,a,c) \wedge R_{q'-1}(k,j,c,b)).
\end{split}
\end{equation}

Now suppose that for every $1<t\leq (q+1)$, $\neg(R_{t-1}(i,j,a,b) \rightarrow R_{t}(i,j,a,b))$, i.e. for every $t$ there exist $i,j<n,a,b<l$ such that $R_{t-1}(i,j,a,b) \wedge \neg R_{t}(i,j,a,b)$, i.e. $\verb|#|R_{t}<\verb|#|R_{t-1}$. Then by induction on $t$ we can prove that $\verb|#|R_{t}\leq q - (t-1)$, therefore $\verb|#|R_{q+1}\leq 0$ (the "worst" case - we removed all edges from $\ddot{\mathcal{R}}_{1}$). In both cases we proved that for every $t>q$, $R_{t+1}(i,j,a,b) \iff R_{t}(i,j,a,b)$.

After the end of propagation, we reduce domains for the second time.
\begin{equation}\label{45CCINSTANCEPRJ}
    \begin{split}
       &Pr_{cc}(i,a)\iff Pr_{1}(i,a)\wedge \forall j<n, E_{\mathcal{X}}(i,j) \rightarrow \exists b,Pr_{1}(j,b)\wedge E^{ij}_{\ddot{\mathcal{R}}_{q+1}}(a,b)\\
       &\hspace{50pt}\wedge\forall k<n, \, E_{\mathcal{X}}(k,i) \rightarrow \exists c, \, Pr_{1}(k,c)\wedge E^{ki}_{\ddot{\mathcal{R}}_{q+1}}(c,a).
    \end{split}
\end{equation}
We denote the new (cycle-consistent) target digraph with domains by $\ddot{\mathcal{A}}_{cc}$ and set
\begin{equation}
V_{\ddot{\mathcal{A}}_{cc}}(i,a)\iff Pr_{cc}(i,a),
\end{equation}
and
\begin{equation}\label{45CCINSTANCEE}
E^{ij}_{\ddot{\mathcal{A}}_{cc}}(a,b)\iff (Pr_{cc}(i,a)\wedge Pr_{cc}(j,b))\wedge E^{ij}_{\ddot{\mathcal{R}}_{q+1}}(a,b).
\end{equation}
\begin{remark}
In Zhuk's algorithm, the original function CheckCycleConsistency in \cite{10.1145/3402029} reduces one domain $D_i$ at a time (as if in (\ref{45CCINSTANCEPRJ}) we fix some $i$), outputs the result $(x_i, D_i')$ and starts all from the beginning. The modified function CheckCC in \cite{DBLP:journals/mvl/Zhuk21} returns all reduced domains at once. Both do not return the reduced relations $E^{ij}_{\ddot{\mathcal{A}}}$: the algorithm applies the function to the initial instance again and again until it cannot produce any further reduction. Nonetheless, it does not affect the final result (we cannot produce two different cycle-consistent reductions), so we omit these technical intermediate steps here. 
\end{remark}

Now we need to prove the following two statements:
\begin{enumerate}
    \item The instance $\Theta_{cc} = (\mathcal{X},\ddot{\mathcal{A}}_{cc})$ is a cycle-consistent instance (according to definition).
    \item If the initial instance $\Theta = (\mathcal{X},\ddot{\mathcal{A}})$ has a solution, then $\Theta_{cc}$ has a solution.
\end{enumerate}
\begin{lemmach}
$V^1$ proves that if none of the domains $V_{\ddot{\mathcal{A}}_{cc}}(i),i<n$ is empty, then the instance $\Theta_{cc}=(\mathcal{X},\ddot{\mathcal{A}}_{cc})$ is cycle-consistent. 
\end{lemmach}

\begin{proof}
Due to definitions (\ref{45CCINSTANCEPRJ})-(\ref{45CCINSTANCEE}), the instance $\Theta_{cc}$ is $1$-consistent. For any $i<n$, any $a\in V_{\ddot{\mathcal{A}}_{cc}}(i)$ consider any cycle $\mathcal{C}_t$ that can be homomorphically mapped into $\mathcal{X}$ with $H(0)=i$ and define the set $H'<\langle t,\langle t,l \rangle\rangle$ such that $H'(0)=\langle i,a\rangle$ and for all $j<t,k<n, H(j)=k\rightarrow H'(j)=\langle k,b\rangle$ for some $b\in V_{\ddot{\mathcal{A}}_{cc}}(k)$ (it exists since none of the domains is empty). We need to prove that there is $b_k$ for each $k<n$ such that $H'$ is a homomorphism from $\mathcal{C}_t$ to $\ddot{\mathcal{A}}$. For this, it is enough to note that by the construction (\ref{CONSTRUCTIONCC}), the formula
\begin{equation}
    \exists b_1,b_2,...,b_{t-1}<l,\, \tilde{E}^{ik_1}_{\ddot{\mathcal{A}}_{cc}}(a,b_1)\wedge \tilde{E}^{k_1k_2}_{\ddot{\mathcal{A}}_{cc}}(b_1,b_2)\wedge...\wedge \tilde{E}^{k_{t-1}i}_{\ddot{\mathcal{A}}_{cc}}(b_{t-1},a)
\end{equation}
where $\tilde{E}^{k_ik_{i+1}}_{\ddot{\mathcal{A}}_{cc}}(b_i,b_{i+1})$ is either $E^{k_ik_{i+1}}_{\ddot{\mathcal{A}}_{cc}}(b_i,b_{i+1})$ or $E^{k_{i+1}k_i}_{\ddot{\mathcal{A}}_{cc}}(b_{i+1},b_i)$ depending on the cycle $\mathcal{C}_t$, is always true since for any $a\in V_{\ddot{\mathcal{A}}_{cc}}(i)$:
$$
E^{ii}_{\ddot{\mathcal{A}}_{cc}}(a,a)\rightarrow \exists b_{t-1}<l,\, E^{ik_{t-1}}_{\ddot{\mathcal{A}}_{cc}}(a,b_{t-1})\wedge E^{k_{t-1}i}_{\ddot{\mathcal{A}}_{cc}}(b_{t-1},a),
$$
$$...$$
$$E^{ik_3}_{\ddot{\mathcal{A}}_{cc}}(a,b_3)\rightarrow \exists b_2<l,\, E^{ik_2}_{\ddot{\mathcal{A}}_{cc}}(a,b_2)\wedge E^{k_2k_3}_{\ddot{\mathcal{A}}_{cc}}(b_2,b_3),$$
$$E^{ik_2}_{\ddot{\mathcal{A}}_{cc}}(a,b_2)\rightarrow \exists b_1<l,\, E^{ik_1}_{\ddot{\mathcal{A}}_{cc}}(a,b_1)\wedge E^{k_1k_2}_{\ddot{\mathcal{A}}_{cc}}(b_1,b_2).$$ 
Set $H'(i)=\langle k_{i},b_{i}\rangle$ for all $0<i<t$. This completes the proof. 
\end{proof}

\begin{lemmach}\label{aappppqppqpqpq}
$V^1$ proves that instance $\Theta=(\mathcal{X},\ddot{\mathcal{A}})$ has a solution if and only if $\Theta_{cc}=(\mathcal{X},\ddot{\mathcal{A}}_{cc})$ has a solution.
\end{lemmach}
\begin{proof} 
Suppose that there is a homomorphism $H$ from $\mathcal{X}$ to $\ddot{\mathcal{A}}$ and it sends edge $E_{\mathcal{X}}(i,j)$ to $E^{ij}_{\ddot{\mathcal{A}}}(a,b)$ for $a\in D_i,b\in D_j$. Due to definition of a homomorphism for both $a$ and $b$, $E^{ij}_{\ddot{\mathcal{A}}}$ must satisfy (\ref{IINT222})-(\ref{IINT3333}) and we do not lose any solution after intersection of all constraints. That is, instead of the set $\{\mathcal{X}\to \ddot{\mathcal{A}}\}$ we can consider set $\{\mathcal{X}\to \ddot{\mathcal{R}}_{1}\}$.

Consider a formula $\phi(t)$ which says that if $H$ is a homomorphism from $\mathcal{X}$ to $\mathcal{R}'_{1}$, then for every step $t$ of propagation, for all $i,j,k\in\{0,1,...,n-1\}$, all $a,b,c<l$
\begin{equation}
    \begin{split}
        &\phi(t)=\ddot{HOM}(\mathcal{X},\mathcal{R}'_{1},H)\wedge H(i)=\langle i,a\rangle \wedge H(j)=\langle j,b\rangle\wedge H(k)=\langle k,c\rangle         \longrightarrow \\
        & \hspace{80pt} (E^{ij}_{\ddot{\mathcal{R}}_{t}}(a,b)\wedge E^{ik}_{\ddot{\mathcal{R}}_{t}}(a,c)\wedge E^{kj}_{\ddot{\mathcal{R}}_{t}}(c,b)). 
    \end{split}
\end{equation} 
For $t=1$ this is true. For every constraint $E_{\mathcal{X}}(i,j)$ the implication $E^{ij}_{\ddot{\mathcal{R}}_{1}}(a,b)$ follows from the definition of a homomorphism. For any $i,j$ such that $\neg E_{\mathcal{X}}(i,j)$ the implication $E^{ij}_{\ddot{\mathcal{R}}_{1}}(a,b)$ follows from the definition of $E^{ij}_{\ddot{\mathcal{R}}_{0}}$ and (\ref{IINT11111})-(\ref{IINT3333}): we do not remove edges from $\ddot{\mathcal{R}}_{1}$ between domains not connected in a constraint without removing vertices. Thus, if there remain some vertices, there will remain all edges between these vertices as well.

If $\phi(t)$ is true for $t=s$, then it is true for $t=(s+1)$ due to construction (\ref{kkofiru67}). Hence, $\{\mathcal{X}\to \ddot{\mathcal{A}}\}\subseteq\{\mathcal{X}\to \ddot{\mathcal{A}}_{cc}\}$. The opposite inclusion is trivial. 
\end{proof}

\subsubsection{Irreducibility}
Consider a cycle-consistent instance $\Theta=(\mathcal{X},\ddot{\mathcal{A}})$ with a domain set $D=\{D_0,...,$ $D_{n-1}\}$. The algorithm chooses a variable $i$ and some maximal congruence $\sigma_i$ on $D_i$ and denotes by $I=\{i\}$ the set of the indices. Then it considers all other variables $k$ such that $k\notin I$ and for some $j\in I$ there is a projection of some constraint $C$ onto $j,k$. Since we consider at most binary relations, and the instance is cycle-consistent, it follows that the projection of any constraint $E_{\mathcal{X}}(j,k)$ (or $E_{\mathcal{X}}(k,j)$) onto $j,k$ is either the constraint relation $E^{jk}_{\ddot{\mathcal{A}}}$ (or $E^{kj}_{\ddot{\mathcal{A}}}$) or domains $D_j,D_k$. 
On domain $D_k$ of such variable $k$, the algorithm generates relation $\sigma_k$ as follows:
\begin{equation}\label{checkmaxcon}
\begin{split}
       &E_{\mathcal{X}}(j,k): \sigma_k(a,b)\iff \exists a',b'\in D_j,\,\sigma_j(a',b')\wedge E^{jk}_{\ddot{\mathcal{A}}}(a,a')\wedge E^{jk}_{\ddot{\mathcal{A}}}(b,b').
       \end{split}
\end{equation}
That is, the algorithm defines a partition on $D_k$ according to the partition on $D_j$. Since this new relation is constructed from relations compatible with $\Omega$ by $pp$-definition, it is also compatible with $\Omega$, and therefore is a congruence. If this congruence is proper, then we have the same number of equivalence classes on $D_k$ as on $D_j$, and elements from one class in $D_k$ are connected with elements only from one class in $D_j$. Otherwise, $\sigma_j$ is not maximal since we can define a new congruence on $D_j$ in an analogous way as in (\ref{checkmaxcon}). The algorithm collects all such $D_k$ with proper congruences $\sigma_k$ into the list of indices $I$, and then considers the projection $\Theta_{prX'}$ of the initial instance onto $X'=\{k|k\in I\}$. This projection can be split into instances on smaller domains (corresponding to connected classes in different domains), and these instances can be solved by recursion. 

\begin{remark}
If there is no domain $D_k$ such that $\sigma_i$ generates on it a proper congruence, the algorithm moves first to another maximal congruence $\sigma'_i$ on $D_i$ and then to $i+1\in\{0,1,...,n-1\}$.
\end{remark}
For every $k\in I$ we thus can check if the solution set to the projection $\Theta_{prX'}$ is subdirect. If not, and for some $k\in I$ there are $b_1,...,b_s$ such that there is no solution to $\Theta_{prX'}$, then the algorithm return $D'_k=D_k\backslash \{b_1,...,b_s\}$ and runs from the beginning. If for all $b\in D_k$ there is no solution to $\Theta_{prX'}$, then the algorithm returns "No solution". If the solution set to $\Theta_{prX'}$ is subdirect, then the algorithm moves to another maximal congruence on $D_i$, and then to $i+1\in\{0,1,...,n-1\}$. If the algorithm cannot reduce any domain $D_i$, and none of the domains is empty, the algorithm returns "OK".

For the formalization of the function CheckIrreducibility, for every domain $D_i$ let us denote by $\sigma_{i}(q,a,b)$ the list of all maximal congruences on $D_i$ (we know them in advance). The number of all congruences on $D_i$ is some constant $q_i\leq 2^{l^{2}}$. Then for every variable $i\in X$, and every maximal congruence $\sigma_{i}^{q}(a,b)$ on $D_{i}$ we iteratively define the following set of elements $I_{t,i,q}(j,a,b)$, where $t$ is the iteration step, $i$ is fixed domain, $q$ is fixed maximal congruence, $j$ is the considered domain and $a,b$ are elements in one congruence class:
\begin{equation}\label{fj8yr66ry5te}
    \begin{split}
        &\hspace{0pt}\forall a,b<l,\,I_{0,i,q}(i,a,b)\iff\sigma_{i}^{q}(a,b)\wedge \\
        &\hspace{40pt}\wedge\forall 0<t<n, k<n, a,b<l,\,I_{t,i,q}(k,a,b)\iff I_{t-1,i,q}(k,a,b)\vee\\
        &\hspace{0pt}\vee\exists  j<n,a',b'<l,\,I_{t-1,i,q}(j,a',b')\wedge\\
        &\hspace{40pt}\wedge(E_{\mathcal{X}}(j,k)\wedge E^{jk}_{\ddot{\mathcal{A}}}(a',a)\wedge E^{jk}_{\ddot{\mathcal{A}}}(b',b))\vee\\
        &\hspace{145pt}\vee(E_{\mathcal{X}}(k,j)\wedge E^{kj}_{\ddot{\mathcal{A}}}(a,a')\wedge E^{kj}_{\ddot{\mathcal{A}}}(b,b'))\wedge\\
        &\hspace{0pt}\wedge \neg\big[\exists c,d\in D_j,\exists e\in D_k,\,\neg I_{t-1,i,q}(j,c,d)\wedge\\
        &\hspace{50pt}\wedge(E_{\mathcal{X}}(j,k)\wedge E^{jk}_{\ddot{\mathcal{A}}}(c,e)\wedge E^{jk}_{\ddot{\mathcal{A}}}(d,e))\vee\\
        &\hspace{155pt}\vee (E_{\mathcal{X}}(k,j)\wedge E^{kj}_{\ddot{\mathcal{A}}}(e,c)\wedge E^{kj}_{\ddot{\mathcal{A}}}(e,d))\big].
    \end{split}
\end{equation} 

At step $t=0$ the set $I_{0,i,q}$ contains only index $i$ and $(a,b)$ such that $a,b\in D_i$ are in the same congruence class of $\sigma_{i}^{q}$. At each further step $t>0$ we add to $I_{t,i,q}$ all elements from $I_{t-1,i,q}$ and indices of the domains connected to elements from $I_{t-1,i,q}$ such that $\sigma_{i}^{q}$ generates proper partitions on those domains. Lines $3$-$5$ consider a connection between $j$ and $k$ and define a partition on $I_{t,i,q}(k)$, and lines $6$-$8$ in square brackets checks that this partition is proper, i.e. no elements $c,d\in D_j$ from different congruence classes connected in $D_k$. Since we cannot add more than $n$ elements to $I$, $I_{n,i,q}$ contains all wanted elements. The existence of this set is provided by induction on $t$ on $\Sigma^{1,b}_1$-formula, and the implication $t\rightarrow (t+1)$ follows from comprehension axiom scheme $\Sigma^{1,b}_0$-CA.

Suppose that the algorithm returns "OK". We will denote the new target digraph with domains after irreducibility reduction by $\ddot{\mathcal{A}}_{ir}$. Due to the algorithm, for each subinstance $\Theta'_{ir}$ of $\Theta_{ir}$, considered by the function CheckIrreducibility, the solution set to $\Theta'_{ir}$ is subdirect. It is obvious that $\Theta'_{ir}$ is not fragmented and not linked. We can formalize the properties of the instance $\Theta_{ir}=(\mathcal{X},\ddot{\mathcal{A}}_{ir})$ as follows: for every $i\in V_{\mathcal{X}}$ and every maximal congruence $\sigma_i^q$
\begin{equation}\label{fyyru8hd}
    \begin{split}
 &\forall V_{\mathcal{X'}}<n,\forall E_{\mathcal{X'}}<4n^2,\,\mathcal{X'}=(V_{\mathcal{X'}},E_{\mathcal{X'}}),\\
 &\hspace{20pt}((\forall j<n,\exists a,b<l,\, V_{\mathcal{X'}}(j)\leftrightarrow I_{n,i,q}(j,a,b))\wedge\\
 &\hspace{40pt}\wedge(\forall s,s'<n,\,E_{\mathcal{X'}}(s,s')\rightarrow s,s'\in V_{\mathcal{X'}})\wedge\\
 &\hspace{60pt}\wedge(\forall s,s'\in V_{\mathcal{X'}},\,E_{\mathcal{X'}}(s,s')\leftrightarrow E_{\mathcal{X}}(s,s')))\rightarrow\\
 &  \hspace{230pt}\rightarrow SSS(\mathcal{X'},\ddot{\mathcal{A}}_{ir}).
    \end{split}
\end{equation}

We need to prove two statements:
\begin{enumerate}
    \item The instance $\Theta_{ir}=(\mathcal{X},\ddot{\mathcal{A}}_{ir})$ is irreducible due to definition.
    \item The initial instance $\Theta = (\mathcal{X},\ddot{\mathcal{A}})$ has a solution only if $\Theta_{ir}$ has a solution.
\end{enumerate}
We start with several technical lemmas. 

\begin{lemmach}\label{djdjdjfuhk87}
$V^1$ proves that for any cycle-consistent instance $\Theta=(\mathcal{X},\ddot{\mathcal{A}})$, for any $i\in X$, relation $Linked(a,b,i,i,\Theta)$ is a congruence on $D_i$.
\end{lemmach}

\begin{proof}
Recall the definition of $Linked(a,b,i,i,\Theta)$:
\begin{equation}
\begin{split}
&\hspace{35pt}Linked(a,b,i,i,\Theta) \iff \exists t<nl, V_{\mathcal{P}_t}=t, E_{\mathcal{P}_t}\leq t^2,\\
&PATH(V_{\mathcal{P}_t},E_{\mathcal{P}_t})\wedge \exists H\leq \langle t,n \rangle, HOM(\mathcal{P}_t,\mathcal{X},H)\wedge(H(0,i)\wedge H(t,i))\wedge\\
&\hspace{70pt}\wedge\exists H'\leq \langle t,\langle t,l\rangle\rangle,\ddot{HOM}(\mathcal{P}_t,\ddot{\mathcal{A}},H')\wedge\\
&\hspace{40pt}\wedge(\forall k<n,p<t,\,(H(p,k)\rightarrow \exists c\in D_k,\,H'(p)=\langle k,c\rangle))\\
&\hspace{80pt}\wedge H'(0)=\langle i,a\rangle \wedge H'(t)= \langle i,b\rangle.
\end{split}
\end{equation}
First of all, for any $a\in D_i$ we have $Linked(a,a,i,i,\Theta)$. Indeed, since the instance is cycle-consistent, it follows that for any cycle $\mathcal{C}_t$ that can be mapped to $\mathcal{X}$ with $H(0,i)$, we will have a homomorphism $H'$ from $\mathcal{C}_t$ to $\ddot{\mathcal{A}}$ such that 
$$
\forall j<n,k<t\, (H(k,j)\rightarrow \exists b\in D_j,\, H'(i)=\langle j,b\rangle)\wedge H'(0)=\langle i,a\rangle.
$$
Instead of cycle $\mathcal{C}_t$ consider a path $\mathcal{P}_t$ such that for all $i<(t-1)$ $$E_{\mathcal{P}_t}(i,i+1)\leftrightarrow E_{\mathcal{C}_t}(i,i+1)\wedge E_{\mathcal{P}_t}(i+1,i)\leftrightarrow E_{\mathcal{C}_t}(i+1,i),
$$
and for $i=(t-1)$
$$
E_{\mathcal{P}_t}(i,i+1)\leftrightarrow E_{\mathcal{C}_t}(i,0)\wedge E_{\mathcal{P}_t}(i+1,i)\leftrightarrow E_{\mathcal{C}_t}(0,i), 
$$
and set $H(t,i), H'(t)=\langle i,a\rangle$. Thus, $Linked(a,b,i,i,\Theta)$ is indeed a relation on the whole $D_x$, and a reflexive one. To prove that the relation is symmetric, for any $a,b$ such that 
$$
\exists \mathcal{P}_t< \langle nl,(nl)^2\rangle,Linked(a,b,i,i,\Theta,\mathcal{P}_t),
$$
consider the inverse path $\mathcal{P}^{-1}_t$ and define a new homomorphisms $M,M$ such that for all $j\leq t, k<n, c<l$
$$M(j,k)\leftrightarrow H(t-j,k)\wedge M'(j)=\langle k,c\rangle \leftrightarrow  H'(t-j)=\langle k,c\rangle.
$$
Finally, if for $a,b,c\in D_i$, there are
$$
\exists \mathcal{P}_t< \langle nl,(nl)^2\rangle,Linked(a,b,i,i,\Theta,\mathcal{P}_t),
$$
$$
\exists \mathcal{P}_m< \langle nl,(nl)^2\rangle,Linked(b,c,i,i,\Theta,\mathcal{P}_m),
$$
we can consider the glued path $\mathcal{P}_t\circ\mathcal{P}_m$, and use on the first and second parts of the path homomorphisms corresponding to $\mathcal{P}_t$ and $\mathcal{P}_m$ respectively. Thus, the relation is transitive.

It remains to show that the relation is compatible with $\Omega$, i.e. $Pol_{m,2}(\Omega,D_i,$ $Linked_{[i,i,\Theta]})$. But it follows from the fact that the set of all pairs $(a,b)\in Linked_{[i,i,\Theta]}$ can be defined by a $pp$-positive formula (see \cite{10.1145/3402029}), and therefore is in the list $\Gamma^2_{\mathcal{A}}$.
\end{proof}

Note that since for every variable $i\in X$ the algorithm checks every maximal congruence on $D_i$, it follows that $Linked_{[i,i,\Theta]}$ is either contained in some maximal congruence or is a maximal congruence itself. Also, for any cycle-consistent instance $\Theta=(\mathcal{X},\ddot{\mathcal{A}})$, for any its subinstance $\Theta'=(\mathcal{X}',\ddot{\mathcal{A}})$ and any $D_i, i\in X'$
$$
Linked(a,b,i,i,\Theta')\rightarrow Linked(a,b,i,i,\Theta),
$$
i.e. the congruence relation $Linked_{[i,i,\Theta]}$ of the instance $\Theta$ contains the congruence relation $Linked_{[i,i,\Theta']}$ of any its subinstance $\Theta'$. By adding any new variable $j\in X\backslash X'$ to $X'$ with all induced edges from $\mathcal{X}$, we cannot make relation $Linked_{[i,i,\Theta']}$ smaller since when it comes down to being linked we consider the existence of a path, and for any $a,b\in D_i$ in $Linked_{[i,i,\Theta']}$ the path already exists. But we can add some new paths, making $Linked_{[i,i,\Theta']}$ larger. 

\begin{lemmach}\label{rjjjus7d}
$V^1$ proves that if an instance $\Theta = (\mathcal{X},\ddot{\mathcal{A}})$ is not fragmented, then for any $i,j\in V_{\mathcal{X}}$ there exist $t<n$ and a path $\mathcal{P}_t$ such that  
$$\exists H\leq \langle t,n \rangle,\,HOM(\mathcal{P}_t,\mathcal{X}, Z)\wedge H(0)=i\wedge H(t)=j.$$ 
\end{lemmach}
\begin{proof} Consider the formula $\theta(t)$
\begin{equation}
    \begin{split}
        &\hspace{5pt}\theta(t) =_{def} t< n,\,i\in V^1_\mathcal{X},j\in V^2_\mathcal{X}\wedge V^1_\mathcal{X}=V^2_\mathcal{X}=n\wedge\texttt{\#}  V^2_{\mathcal{X}}(n)=t\wedge\\
        &\wedge PSS(V^1_\mathcal{X},V_\mathcal{X})\wedge PSS(V^2_\mathcal{X},V_\mathcal{X})\wedge (\forall k<n,\,V^1_\mathcal{X}(k)\leftrightarrow \neg V^1_\mathcal{X}(k))\wedge\\
        &\hspace{10pt}\wedge \exists m\leq t,\exists\mathcal{P}_m,\, V_{\mathcal{P}_m}=m, E_{\mathcal{P}_m}<m^2,PATH(V_{\mathcal{P}_m},E_{\mathcal{P}_m})\wedge\\
        &\wedge \exists H\leq \langle m,n \rangle, HOM(\mathcal{P}_m,\mathcal{X},H)\wedge H(0,j)\wedge H(m,i')\wedge i'\in V^1_\mathcal{X}.
        \end{split}
\end{equation}
For $t=1$, the formula is true since $\neg FRGM(\mathcal{X},\ddot{\mathcal{A}})$. If $\theta(t)$ is true for $t=s$, then it is also true for $t=(s+1)$. Indeed, since the instance is not fragmented, it follows that for $V^2_{\mathcal{X}}$, $\texttt{\#}  V^2_{\mathcal{X}}(n)=(s+1)$ there are two elements $i'\in V^1_{\mathcal{X}}$ and $j'\in V^2_{\mathcal{X}}$ such that there is an edge $E_{\mathcal{X}}(i',j')$ or $E_{\mathcal{X}}(j',i')$. Then consider two sets $V^1_{\mathcal{X}}\cup\{j'\}$ and $V^2_{\mathcal{X}}\backslash\{j'\}$. Since $\texttt{\#} V^2_{\mathcal{X}}\backslash\{j'\} =s $, there has to be a path $\mathcal{P}_m, m\leq s$, and $H \leq \langle m,n \rangle$ with $H(0)=j, H(m)=i''$ for some $i''\in V^1_{\mathcal{X}}\cup\{j'\}$. If $i''= j'$, we get a path of length $m\leq (s+1)$ from $j$ to $i'$. If $i''\neq j'$, then there is a path of length $m\leq s$ from $j$ to some element $i''\in V^1_{\mathcal{X}}$. Finally, it also must be true for $t=n-1$.
\end{proof}

\begin{lemmach}\label{7hhfyrj}
$V^1$ proves that if a cycle-consistent instance $\Theta = (\mathcal{X},\mathcal{A})$ is not fragmented and not linked, then for all $D_i$ there exist $a,b\in D_i$ such that $\neg Linked(a,b,i,i,\Theta)$.
\end{lemmach}
\begin{proof} Since the instance $\Theta$ is not linked, by definition there exist $i\in V_{\mathcal{X}}$ and $a,b\in D_i$ such that $\neg LinkedCon(a,b,i,i,\Theta)$. Suppose that there exists $D_j$ such that for any $a',b'\in D_j$ we have $LinkedCon(a',b',j,j,\Theta)$, i.e. there exist some path $\mathcal{P}_t$ and a homomorphism $H'$ from $\mathcal{P}_t$ to $\ddot{\mathcal{A}}$ connecting $a'$ and $b'$. Since the instance is not fragmented, due to Lemma \ref{rjjjus7d} it follows that there exists a path $\mathcal{P}_s$ from $i$ to $j$. Consider the reverse path $\mathcal{P}^{-1}_s$ and define a cycle $C_{2s}$ as follows:
\begin{equation}
 \begin{split}
&\hspace{35pt}V_{\mathcal{C}_{2s}}=2m \wedge  \forall k<m,\,E_{\mathcal{C}_{2s}}(k,k+1) \leftrightarrow E_{\mathcal{P}_s}(k,k+1) \wedge\\
&\hspace{60pt}\wedge \leftrightarrow E_{\mathcal{C}_{2s}}(k+1,k)\leftrightarrow  E_{\mathcal{P}_s}(k+1,k)\wedge\\
&\hspace{30pt}\wedge\forall r<(s-1), \,E_{\mathcal{C}_{2s}}(s+r,s+r+1) \leftrightarrow E_{\mathcal{P}^{-1}_s}(r,r+1) \wedge \\ &\hspace{55pt}\wedge E_{\mathcal{C}_{2s}}(s+r+1,s+r) \leftrightarrow E_{\mathcal{P}^{-1}_s}(r+1,r) \wedge\\
&\wedge E_{\mathcal{C}_{2s}}(2s-1,0) \leftrightarrow E_{\mathcal{P}^{-1}_s}(s-1,s) \wedge E_{\mathcal{C}_{2m}}(0,2s-1) \leftrightarrow E_{\mathcal{P}^{-1}_s}(s,s-1).\\
 \end{split}
\end{equation}
That is, in $C_{2s}$ we glued together the start and the end of paths $\mathcal{P}_s$ and $\mathcal{P}^{-1}_s$ respectively, and vice versa. This cycle can obviously be mapped into $\mathcal{X}$, and due to cycle-consistency for $a,b\in D_i$ there exist homomorphisms $H'_{a},H'_{b}$ from $\mathcal{C}_{2s}$ to $\ddot{\mathcal{A}}$ such that $H'_{a}(0)=\langle i,a\rangle, H'_{b}(0)=\langle i,b\rangle$. 
Suppose that $H'_{a}(s)=\langle j,a'\rangle$ and $H'_{b}(s)=\langle j,b'\rangle$ and consider a path $\mathcal{P}_s\circ\mathcal{P}_t\circ \mathcal{P}^{-1}_{s}$. Then use homomorphism $H'_{a}$ for $\mathcal{P}_s$, $H'$ for $\mathcal{P}_t$ and $H'_{b}$ for $\mathcal{P}^{-1}_{s}$. Thus, we have a path and a new homomorphism connecting $a$ and $b$ in $D_i$. That is a contradiction. 
\end{proof}
\begin{remark}
Note that in proof of Lemma \ref{7hhfyrj} we have to use cycle-consistency. We can ensure a path from $i$ to $j$ in $\mathcal{X}$ due to the fact that the instance is not fragmented, but without cycle-consistency (or linked property) we cannot ensure that this path has proper evaluation in $\ddot{\mathcal{A}}$.
\end{remark}
\begin{lemmach}
$V^1$ proves that the instance $\Theta_{ir}=(\mathcal{X},\ddot{\mathcal{A}}_{ir})$ is irreducible.
\end{lemmach}
\begin{proof}
Suppose that there exists a subinstance $\Theta'= (\mathcal{X'}, \ddot{\mathcal{A}}_{ir})$ such that $\mathcal{X'} = (V_{\mathcal{X'}},  E_{\mathcal{X'}})$, $V_{\mathcal{X'}}<n, E_{\mathcal{X'}}<4n^2$, $V_{\mathcal{X'}}$ is a subset of $V_{\mathcal{X}}$, $E_{\mathcal{X'}}$ is a subset of $E_{\mathcal{X}}$, and $$E_{\mathcal{X'}}(x_1,x_2)\rightarrow x_1,x_2\in V_{\mathcal{X'}},$$
and this instance is not fragmented, and not linked, and its solution set is not subdirect. We need to prove that any such subinstance must be included in some subinstance generated by the algorithm (and therefore must have a subdirect solution set).
 
Due to Lemma \ref{7hhfyrj}, for any $i\in V_{\mathcal{X'}}$ there exist $a,b\in D_i$, $(a,b)\notin Linked_{[i,i,\Theta']}$, thus any such congruence is proper. Fix some $i\in X'$, and consider a maximal congruence $\sigma_i^q(a,b)$ for some $q<q_i$ on $D_i$ that contains $Linked_{[i,i,\Theta']}$. Consider subinstance $\Theta''=(\mathcal{X''},\ddot{\mathcal{A}})$, defined as:
\begin{equation}
    \begin{split}
 &\forall j<n,\exists a,b<l,\, V_{\mathcal{X''}}(j)\leftrightarrow I_{n,i,q}(j,a,b)\wedge\\
 &\hspace{50pt}\wedge\forall s,s'<n,\,E_{\mathcal{X''}}(s,s')\rightarrow s,s'\in V_{\mathcal{X''}}\wedge\\
 &\hspace{100pt}\wedge\forall s,s'\in V_{\mathcal{X''}},\,E_{\mathcal{X''}}(s,s')\leftrightarrow E_{\mathcal{X}}(s,s').
    \end{split}
\end{equation}

We need to show two points:
\begin{enumerate}
\item For every $j\in X'$ there exist $a',b'\in D_j$ such that $I_{n,i,q}(j,a',b')$ (i.e. $X'$ is a subset of $X''$).
\item For every $j\in X'$, for all $a',b'\in D_j$,
\begin{align*}
\begin{split}
&I_{n,i,q}(j,a',b')\longrightarrow \exists a,b\in D_i,\,I_{n,i,q}(i,a,b)\wedge\\
&\hspace{100pt}\wedge Linked(a,a',i,j,\Theta)\wedge Linked(b,b',i,j,\Theta),
\end{split}
\end{align*}
and for all $a,b\in D_i$, for all $j\in X'$, $a',b'\in D_j$
\begin{align*}
\begin{split}
&I_{n,i,q}(i,a,b)\wedge Linked(a,a',i,j,\Theta)\wedge Linked(b,b',i,j,\Theta)\rightarrow I_{n,i,q}(j,a',b').
\end{split}
\end{align*}
This means that in $\Theta'$ the congruence $\sigma^{q}_{i}(a,b)$ generates the same partition on each domain as in $\Theta''$.
\end{enumerate}
For the first claim, note that since the instance $\Theta'$ is not fragmented, due to Lemma \ref{rjjjus7d} it follows that $V^1$ proves that for any $j\in V_{\mathcal{X'}}$ there exist $s<n$ and a path $\mathcal{P}_s$ connecting $i$ and $j$. We go by the induction on the length of that path. For $s=0$ we have $I_{0,i,g}(i,a,b)$, for $s=1$ consider some $k$ such that $E_{\mathcal{X'}}(i,k)$ (or $E_{\mathcal{X'}}(k,i)$). Since the instance is $1$-consistent, there exist some $c,d\in D_i$, $c',d'\in D_k$ such that 
$$E_{\mathcal{X'}}(i,k)\wedge E^{ik}_{\ddot{\mathcal{A}}}(c,c')\wedge E^{ik}_{\ddot{\mathcal{A}}}(d,d'),$$
and the only thing we have to check due to defining equation (\ref{fj8yr66ry5te}) is that there are no $c,d\in D_{i}, e\in D_{k}$ such that $\neg I_{0,i,g}(i,c,d)$ and
$$E_{\mathcal{X'}}(i,k)\wedge E^{ik}_{\ddot{\mathcal{A}}}(c,e)\wedge E^{ik}_{\ddot{\mathcal{A}}}(d,e).$$
It follows immediately from the fact that if such $c,d,e$ exist, then $Linked(c,d,i,$ $i,\Theta')$ and therefore $I_{0,i,q}(i,c,d)$ (the congruence $\sigma_i^q(a,b)$ contains $Linked_{[i,i,\Theta']}$). For the implication $s=t\rightarrow s=(t+1)$, suppose that for every $k\in X'$ such that there exists a path of length $t$ connecting $i$ and $k$, there exist $j\in X'$, $c,d\in D_j$, $c',d'\in D_k$ such that $I_{t-1,i,g}(j,c,d)$, and all other conditions of (\ref{fj8yr66ry5te}) hold. Note that for $s=0,1$ we established $Linked(c,c',i,k,\Theta')\wedge Linked(d,d',i,k,\Theta')$, so we can assume that this is true for $s=t$ as well. Then use the same reasoning.

The first implication of claim $2$ follows from the above. For the second implication we again use induction on the length of a path. For $s=0,1$ it follows from the definition of $I_{n,i,q}$. For the implication $s=t\rightarrow s=(t+1)$ suppose that for every $k\in X'$ such that there exists a path of length $t$ connecting $i$ and $k$, for any $a,b\in D_i$ and any $a',b'\in D_k$ such that $I_{n,i,q}(i,a,b)\wedge Linked(a,a',i,k,\Theta)\wedge Linked(b,b',i,k,\Theta)$ we have $I_{n,i,q}(k,a',b')$. But since we can consider any path of length $(t+1)$ as glued paths of length $t$ and $1$, the implication for $s=(t+1)$ again follows straightaway from the definition of $I_{i,n,q}$. This completes the proof. 
\end{proof}

\begin{lemmach}\label{qqnhygtfrd} $V^1$ proves that $\Theta=(\mathcal{X},\ddot{\mathcal{A}})$ has a solution only if $\Theta_{ir}=(\mathcal{X},\ddot{\mathcal{A}}_{ir})$ has a solution.
\end{lemmach}
\begin{proof}
It is sufficient to show that if $\Theta$ has a solution, then $\Theta$ has a solution on domains $D_0,...,D_{j-1},D_j\backslash\{b_1,...,b_s\}, D_{j+1},$ $...,D_{n-1}$ after irreducibility reduction of one domain $D_j$. This is straightforward. Fix some $i_0$ and suppose that the maximal congruence $\sigma^q_{i_0}$ divides $D_{i_0}$ to $t$ equivalence classes. To make a reduction we consider some subgraph $\mathcal{X'}$ of digraph $\mathcal{X}$ containing vertex $i_0$ and such that it is connected and contains only vertices for which domains $D_{i_1},...,D_{i_g}$ congruence $\sigma_{i_0}$ generates proper congruences. Since instance $\Theta$ is cycle-consistent, therefore for any $s,t$ projection of $E^{st}_{\ddot{\mathcal{A}}}$ onto $D_s,D_t$ Is subdirect. Thus, we construct a subinstance $\Theta_{prX'} = (\mathcal{X'},\ddot{\mathcal{A}})$ of instance $\Theta$ with the same target digraph with domains (and the same domain set), but with another input digraph $\mathcal{X'}$. 

Suppose that there is a homomorphism from $\mathcal{X}$ to $\ddot{\mathcal{A}}$. For every $H\in$ $\{\mathcal{X}\to \ddot{\mathcal{A}}\}$ define a new homomorphism $H\restriction_{X'}$ from $\mathcal{X'}$ to $\ddot{\mathcal{A}}$ as follows: 
\begin{equation}
\forall i\in \{i_0,i_1,...,i_g\},\,H\restriction_{X'}(i)=\langle i,a\rangle \iff H(i)=\langle i,a\rangle.
\end{equation}
That $H\restriction_{X'}$ is a homomorphism follows right from the definition of $H$.  Therefore, $\{H\restriction_{X'}\}\subseteq \{\mathcal{X'}\to \ddot{\mathcal{A}}\}$. If for some $j\in \{i_0,i_1,...,i_g\}$ and some $b_1,...,b_s$ there is no homomorphism $H'\in \{\mathcal{X'}\to \ddot{\mathcal{A}}\}$ such that $H'(j)=\langle j,b_1\rangle,...,H'(j)=\langle j,b_s\rangle$, then no homomorphism from $\{\mathcal{X}\to \ddot{\mathcal{A}}\}$ sends $j$ to $\langle j,b_1\rangle,...,\langle j,b_s\rangle$.
\end{proof}

\subsubsection{Weaker instance}
When the algorithm runs the function CheckWeakerInstance it makes a copy of $\Theta=(\mathcal{X},\ddot{\mathcal{A}})$ and simultaneously replaces every constraint in the instance with all weaker constraints \emph{without dummy variables}. Then for every $i\in\{0,1,...,n-1\}$ it checks if the obtained weaker instance has a solution for $x_i=b$, for every $b\in D_i$ (by recursively calling the algorithm on a smaller domain). That is, the algorithm checks if the solution set to the weaker instance is subdirect. Suppose that the algorithm considers some $i$, set $D_i'=\emptyset$. It fixes the value $x_i=b$ and solves the weaker instance with domain set $D_0,...,D_{i-1},\{b\}, D_{i+1},...,D_{n-1}$. If there is a solution, then it adds $b$ to $D_i'$ and proceeds with another $b'\in D_i$. If there are solutions for all $b\in D_i$, the algorithm proceeds with $i+1$. If for each $b\in D_i$ there is no solution, the algorithm answers that the initial instance has no solution. If there are some $b_1,...,b_k\in D_i$ for which there is no solution to the weaker instance, the algorithm reduces domain $D_i$ to $D_i'=D\backslash\{b_1,...,b_s\}$, returns $(x_i, D_i')$ and starts from the beginning.

Consider a cycle-consistent irreducible instance $\Theta = (\mathcal{X}, \ddot{\mathcal{A}})$. Any constraint in $\Theta$ is either a domain $D_i$ for a variable $i$, or a relation $E^{ij}_{\ddot{\mathcal{A}}}$ for an edge $E_{\mathcal{X}}(i,j)$. Since $\Theta$ is cycle-consistent, projections $pr_i(E^{ij}_{\ddot{\mathcal{A}}})$ and $pr_j(E^{ij}_{\ddot{\mathcal{A}}})$ are equal to $D_i,D_j$. The algorithm never increases domains, so we weaken only binary constraints and replace each $E^{ij}_{\ddot{\mathcal{A}}}$ by two different types of weaker constraints: 
\begin{enumerate}
    \item $D_i$, $D_j$ - weaker constraints of less arity;
    \item All binary constraints from the list $\Gamma_{\mathcal{A}}$ containing $E^{ij}_{\ddot{\mathcal{A}}}$ except the full relation on $D_i\times D_j$.
\end{enumerate}

Consider the intersection of all the above weaker constraints. Note that for any $i$ we have the same domain $D_i$. We can lose some edges $(i,j)$ from $E_{\mathcal{X}}$ (when the only binary relation containing $E^{ij}_{\ddot{\mathcal{A}}}$ is the full relation on $D_i\times D_j$) and can add some edges to $\ddot{\mathcal{A}}$. Let us denote the obtained weaker instance by $\Theta_{weak}=(\mathcal{X}_{weak},\ddot{\mathcal{A}}_{weak})$.

\begin{lemmach}\label{gfhdyejr7454}
$V^1$ proves that a CSP instance $\Theta=(\mathcal{X},\ddot{\mathcal{A}})$ has a solution only if $\Theta$ has a solution after the weaker instance reduction. 
\end{lemmach}
\begin{proof}
It is obvious that if instance $\Theta$ has a solution, then $\Theta_{weak}$ has a solution (we did not remove any edge or vertex from $\ddot{\mathcal{A}}$ and probably removed some edges from $\mathcal{X}$: just take the same homomorphism). That is, $\{\mathcal{X}\to\ddot{\mathcal{A}}\}\subseteq\{\mathcal{X}_{weak}\to\ddot{\mathcal{A}}_{weak}\}$.

Suppose that for some $i$ there are $b_1,...,b_s\in D_i$ such that there is no solution to $\Theta_{weak}$, i.e. there is no homomorphism $H$ in $\{\mathcal{X}_{weak}\to\ddot{\mathcal{A}}_{weak}\}$ such that $H(i)=\langle i,b_1\rangle,...,H(i)=\langle i,b_s\rangle$. It is needed to show that if $\Theta$ has a solution, then $\Theta$ has a solution on domains $D_0,...,$ $D_{i-1},D_i\backslash\{b_1,...,b_s\},$ $ D_{i+1},$ $...,D_{n-1}$. But it is trivial.
\end{proof}

\subsection{Linear case}\label{LLLLLINEARG}
In this section we will formalize and prove the soundness of the linear case of Zhuk's algorithm in the theory $V^1$ using $\Sigma^{1,b}_{1}$-induction.

\subsubsection{Formalization of the linear case in $V^1$}\label{FORMALIZATION}
For the linear case of Zhuk's algorithm, we need to define in $V^1$ some additional notions, namely finite abelian groups and matrices over finite fields. 

To formalize the finite abelian group $\mathbb{Z}_p = \{0,1,...,p-1\}$ we define sum operation $+_{(mod\,p)}$ as follows:
\begin{equation}
    c = a+_{(mod\,p)} b \longleftrightarrow c<p\wedge c \equiv (a+b)\, (mod\,p).
\end{equation}
We define the identity element to be $0$ and the inverse element for any $a<p$, denoted $-_{(mod\,p)}a$, to be $p \dot - a$. Furthermore, for any $m\in \mathbb{N}$ and any $a\in \mathbb{Z}_p$ we can define $\cdot_{(mod\,p)}$ as follows:
\begin{equation}
c=m\cdot_{(mod\,p)}a \longleftrightarrow c<p\wedge c\equiv (ma)\,(mod\,p).
\end{equation}
For fields (when $p$ is a prime number) we can also define the multiplicative inverse for any $a\neq 0, a\in\mathbb{Z}_p$, denoted by $a^{-1}$: 
\begin{equation}
    c = a^{-1} \longleftrightarrow c<p\wedge c\neq 0\wedge c \cdot_{(mod\,p)} a = a \cdot_{(mod\,p)} c = 1.
\end{equation}
It is clear that $+_{(mod\,p)}, -_{(mod\,p)}, \cdot_{(mod\,p)}$ and $0$ can be defined in a weak subtheory of $V^1$ and satisfy all properties of a finite abelian group. A weak subtheory of $V^1$ can also define the multiplicative inverse modulo a prime and hence, in particular, $V^1$ proves that $Z_p$ is a field. In our case, primes $p_i$ are even fixed constants.
 
An $m \times n$ matrix $A$ over $\mathbb{Z}_p$ is encoded by a relation $A(i,j,a)$, we write $A_{ij}=a$ for the corresponding entry. We will denote by $MX_{m\times n,p}(A)$ a relation that $A$ is an $m \times n$ matrix over $\mathbb{Z}_p$. The sum of two $m\times n$ matrices $A$ and $B$ can be defined by a set-valued function
\begin{equation}\label{Matrixaddition}
    \begin{split}
&C = A+B \longleftrightarrow MX_{m\times n,p}(C) \wedge\forall i<m,j<n \,\,C_{ij}= A_{ij}+_{(mod\,p)}B_{ij},
    \end{split}
\end{equation}
and the scalar multiplication $b A$ of a number $b\in \mathbb{Z}_p$ and an $m\times n$ matrix $A$ can be defined as:  
\begin{equation}
    \begin{split}
&C = b A \longleftrightarrow MX_{m\times n,p}(C) \wedge\forall i<m,j<n\,\, C_{ij}= b\cdot_{(mod\,p)}A_{ij}.
    \end{split}
\end{equation}
The definability of matrix addition and scalar multiplication in $V^1$ is obvious. Finally, to define the matrix multiplication, we will use the fact that $V^1$ defines the summation of long sums, i.e. if $C$ is a function with domain $\{0,...,n-1\}$, then $V^1$ defines the sum $\sum_{i<n} C(i)$ and proves its basic properties. 

Indeed, consider $\Sigma^{1,b}_1$-induction on $t\leq n$, where $t$ is the number of elements in formula
\begin{equation}
    \begin{split}
       &\phi(i,j,t,A,B) =_{def} \exists X< \langle t,p\rangle,\,\, X_0 = A_{i0}\cdot_{(mod\,p)}B_{0j}\wedge\\
&\hspace{110pt}\forall 0<k< t\, X_k=X_{k-1} +_{(mod\,p)} A_{ik}\cdot_{(mod\,p)}B_{kj}. 
    \end{split}
\end{equation}
Here $X$ encodes the sequence of $t$ partial sums, and by $X_k$ we denote $X(k)$. For $t=1$, $\phi(i,j,t,A,B)$ is true (since $\cdot_{(mod\,p)}$ is definable in $V^1$), and $\phi(i,j,t+1,A,B)$ follows from $\phi(i,j,t,A,B)$ since $+_{(mod\,p)}$ is also definable in $V^1$. This uses $\Sigma^{1,b}_1$ induction. 

We can thus define the multiplication of an $m\times n$ matrix $A$ and an $n\times s$ matrix $B$ as follows: 
\begin{equation}
    \begin{split}
&\hspace{50pt}C = AB \longleftrightarrow MX_{m\times s,p}(C) \wedge\forall i<m,j<s \\
&C_{ij}= A_{i0}\cdot_{(mod\,p)}B_{0j}+_{(mod\,p)}...+_{(mod\,p)} A_{i(n-1)}\cdot_{(mod\,p)}B_{(n-1)j}.
\end{split}
\end{equation}
We will further use notation $+,-,\cdot$ instead of $+_{(mod\,p)}, -_{(mod\,p)}$ and $\cdot_{(mod\,p)}$ since it does not lead to confusion.

\subsubsection{Soundness of the linear case in $V^1$}
We will call an instance $\Theta = (\mathcal{X}, \ddot{\mathcal{A}})$, produced by the algorithm before the linear case, the initial instance. As the first modification of the instance, we need to define a factorized instance $\Theta_L$: at this step, we change the target digraph $\ddot{\mathcal{A}}$ and do not change instance digraph $\mathcal{X}$. The algorithm factorizes each domain separately and due to the assumption for every domain $D_i$ there is the minimal linear congruence $\sigma_i$ such that $D_i/\sigma_i$ is isomorphic to linear algebra. Denote by $\sigma<nl^2$ the set representing all congruences $\sigma_i$, $\sigma(i,a,b)\iff\sigma_i(a,b)$. The factorized target digraph with domains $\ddot{\mathcal{A}}_L$ can be represented as an $(n+2)$-tuple $(V_{\ddot{\mathcal{A}}_L}, E_{\ddot{\mathcal{A}}_L}, D_0/\sigma_0,...,D_{n-1}/\sigma_{n-1})$, where $V_{\ddot{\mathcal{A}}_L}< \langle n,l \rangle$, $V_{\ddot{\mathcal{A}}_L}(i,a)\iff \mathcal{D}/\sigma_i(a)$ and $E_{\ddot{\mathcal{A}}_L}$ such that
\begin{equation}
    \begin{split}
       &\hspace{15pt} E_{\ddot{\mathcal{A}}_L}(s,r)\iff \exists i,j<n\,\exists a,b<l,\,\, s = \langle i,a\rangle\wedge r = \langle j,b\rangle\wedge\\ 
        &D_i/\sigma_i(a)\wedge D_j/\sigma_j(b)
       \wedge (\exists c,d<l,\,\,\sigma(i,a,c)\wedge \sigma(j,b,d)\wedge E^{ij}_{\ddot{\mathcal{A}}}(c,d)).
    \end{split}
\end{equation}
In words, there is an edge between elements $a,b$ representing classes $[a]/\sigma_i$ and $[b]/\sigma_j$ in $\ddot{\mathcal{A}}_L$ any time $E^{ij}_{\ddot{\mathcal{A}}}\cap [a]/\sigma_i\times [b]/\sigma_j\neq \emptyset$.
In the factorized target digraph constructed in such a way, we actually can lose some edges (for example, when we glue all edges between elements in $[a]/\sigma_i$ and $[b]/\sigma_j$ in one edge), but we also can get new solutions (for example, when we get new cycles). We thus increase the set of solutions by simplifying the structure of the target digraph with domains. 

\begin{theorem}\label{INITIALFACTOR}
$V^1$ proves that an instance $\Theta=(\mathcal{X},\ddot{\mathcal{A}})$ has a solution only if $\Theta_L=(\mathcal{X},\ddot{\mathcal{A}}_L)$ has a solution.
\end{theorem}
\begin{proof} Consider a CSP instance  $\Theta=(\mathcal{X},\ddot{\mathcal{A}})$ with $V_\mathcal{X}=n$, $V_{\ddot{\mathcal{A}}}<\langle n,l\rangle$. Suppose that the instance has a solution, i.e. there exists a homomorphism $H$ from $\mathcal{X}$ to $\ddot{\mathcal{A}}$. Construct the factorized instance as mentioned above. 

We first construct the canonical homomorphism $H_c$ between the target digraph $\ddot{\mathcal{A}}$ and the factorized digraph $\ddot{\mathcal{A}}_L$, and then show that there is a homomorphism from $\mathcal{X}$ to $\ddot{\mathcal{A}}_L$. Define $H_c$ as follows: for every $u\in V_{\ddot{\mathcal{A}}}$, and every $v\in V_{\ddot{\mathcal{A}}_L}$
$$
Z_c(u,v)\iff \exists i<n,a,b<l,\,u = \langle i,a\rangle,\,v = \langle i,b\rangle\wedge \sigma(i,b,a) \wedge D_i/\sigma_i(b).
$$
That is, we send a vertex $a$ to a vertex $b$ in $\ddot{\mathcal{A}}_L$ in the factorized domain $D_i/\sigma_i$ if and only if $b\in D_i$, $b$ and $a$ are in the same congruence class under $\sigma_i$, and $b$ is a represent of the class $[a]/\sigma_i$ (the smallest element). This set exists due to $\Sigma^{1,b}_0$-comprehension axiom. Moreover, it satisfies the relation of being a well-defined map between two sets $V_{\ddot{\mathcal{A}}}$ and $V_{\ddot{\mathcal{A}}_L}$. The existence of $b$ is ensured by the property of congruence relation $\sigma_i$ (reflexivity), and the uniqueness by our choice of representation of the factor set by the minimal element in the class. It is left to show that
$$
 \forall u_1,u_1, v_1,v_2 < \langle n, l\rangle  (E_{\ddot{\mathcal{A}}}(u_1,u_2)\wedge Z_c( u_1,v_1) \wedge Z_c( u_2,v_2) \to E_{\ddot{\mathcal{A}}_L} (v_1,v_2)),
$$
but this follows straightforwardly from the definition of $H_c$ and $E_{\ddot{\mathcal{A}}_L}$. Finally, to construct a homomorphism from $\mathcal{X}$ to $\mathcal{A}_L$, consider set $H'<\langle n,\langle n,l\rangle\rangle$ such that
$$
H'(i)=v \iff \exists u<\langle n,l\rangle (H(i)=u\wedge H_c(u)=v.
$$
It is easy to check that set $H'$ satisfies the homomorphism relation between digraphs $\mathcal{X}$ and $\ddot{\mathcal{A}}_L$. Thus, there is a solution to the factorized instance $\Theta_L$.
\end{proof}

Suppose that there is a solution set to the instance $\Theta$, the set of homomorphisms from $\mathcal{X}$ to $\ddot{\mathcal{A}}$, denoted by $\{\mathcal{X}\to\ddot{\mathcal{A}}\}=\{H_1,H_2,...,H_s\}$. We will call the set of all homomorphisms, constructed from $H_1,...,H_s$ by canonical homomorphisms $H_c$ the solution set to $\Theta$ factorized by congruences, denoted by $\{\mathcal{X}\to\ddot{\mathcal{A}}\}/\Sigma=\{H'_1,...,H'_s\}$ (some of the homomorphisms $H'_1,...,H'_s$ can be equivalent).

By the previous theorem, we established that $\Theta_L$ has a solution only if $\Theta$ does. Now to find solutions to $\Theta_L$ we will use the translation of constraints into a system of linear equations (we suppose that this translation is included in the algorithm's transcription) and run Gaussian Elimination. We thus need to show in $V^1$ that this process does not reduce the solution set to $\Theta_L$. Let us recall that a matrix $A$ is in the \emph{row echelon form} if it is either a zero matrix or its first non-zero entry of row $i+1$ must be on the right of the first non-zero entry of row $i$, and these entries must be $1$. Consider the system of linear equations $A\bar{x}=\bar{b}$ for an $m\times n$ matrix $A$. Suppose that we have a sequence of $m\times (n+1)$ matrices $[A_0|B_0]$, $[A_1|B_1],...,[A_t|B_t]$, where $[A_0|B_0]$ is the original augmented matrix of the system of linear equations, $[A_t|B_t]$ is a matrix in the row echelon form and every next matrix is obtained from the previous one by one of the elementary row operations. Since every elementary row operation can be simulated by left multiplication by an elementary matrix, instead of defining elementary row operations, we define elementary matrices in $V^1$.

We say that an $m\times m$ matrix $E$ is elementary if $E$ satisfies one of the following three relations. The first of them corresponds to row-switching transformations
\begin{equation}
\begin{split}
&\hspace{0pt}EL^I_{m\times m,p}(E)\iff MX_{m\times m,p}(E)\wedge \exists i'\neq j'<m \forall i,j<m \\
&\hspace{10pt}(i\neq i'\wedge i\neq j' \to E_{ii}=1) \wedge (i\neq i'\wedge j\neq j'\wedge i\neq j \to E_{ij}=0)\\
&\hspace{100pt}\wedge (E_{i'i'}=0\wedge E_{j'j'}=0\wedge E_{i'j'}=1\wedge E_{j'i'}=1), 
\end{split}
\end{equation}
the second one corresponds to row-multiplying transformations
\begin{equation}
\begin{split}
&\hspace{0pt}EL^{II}_{m\times m,p}(E)\iff MX_{m\times m,p}(E)\wedge \exists a\neq 0\in \mathbb{Z}_p\exists i'<m \\
&\hspace{40pt}\forall i,j<m (i\neq i' \to E_{ii}=1)\wedge (i\neq j\to E_{ij}=0)\wedge E_{i'i'}=a,
\end{split}
\end{equation}
and the last one corresponds to row-addition transformations
\begin{equation}
\begin{split}
&\hspace{0pt}EL^{III}_{m\times m,p}(E)\iff MX_{m\times m,p}(E)\wedge \exists a\neq 0\in \mathbb{Z}_p\exists i',j'<m \forall i,j<m \\
&\hspace{60pt}(E_{ii}=1\wedge (i\neq j \wedge i\neq i'\wedge j\neq j'\to E_{ij}=0)\wedge E_{i'j'}=a.
\end{split}
\end{equation}
Let us denote these elementary matrices by $T^1$, $T^2$, $T^3$. If we consider matrix $[A|B]$, then matrices $T^1[A|B]$, $T^2[A|B]$ and $T^3[A|B]$ are matrices produced from $[A|B]$ by elementary row operations. Since $V^1$ can define long sums it is easy to show that $V^1$ proves that each of elementary row operations preserves the solution set to $A\bar{x}=\bar{b}$. 

\begin{lemmach}\label{gaussssiy643}
$V^1$ proves that for every matrix $[A|B]$ there is a row-echelon matrix $[A'|B']$ having the same solution set.
\end{lemmach}
\begin{proof}
Use $\Sigma^{1,b}_1$-induction. 
\end{proof}

Suppose now that we have established the solution set to the factorized instance $\Theta_L$, $\{\mathcal{X}\to\ddot{\mathcal{A}}_L\}$, and assume that $\{\mathcal{X}\to\ddot{\mathcal{A}}\}/\Sigma\subsetneq \{\mathcal{X}\to\ddot{\mathcal{A}}_L\}$. We will further proceed with iterative steps of the algorithm, the first iteration (see Subsection \ref{DETAILEDLINEAR}). We arbitrarily choose a constraint $E_{\mathcal{X}}(i,j)$ and replace it with all weaker constraints without dummy variables, making the initial instance weaker. It can be done either by adding some edges to the relation $E^{ij}_{\ddot{\mathcal{A}}}$ (note that new edges have to be preserved by WNU operation $\Omega$) or by removing the edge $(i,j)$ from $\mathcal{X}$ (when the only relation containing $E^{ij}_{\ddot{\mathcal{A}}}$ is the full relation on $D_i\times D_j$). Without loss of generality, suppose that we start with $\mathcal{X}$. We prove the following theorem by induction on the number of edges removed from $\mathcal{X}$. The process of removing can be interrupted by modifications of $\ddot{\mathcal{A}}$ as well, but since this interruption happens only the constant number of times (the number of edges we can add to $\ddot{\mathcal{A}}$ is a constant), we can consider the constant number of separate inductions as one from start to the end.

\begin{theorem}\label{FirstIterationOf2}
Consider two CSP instances, the initial instance $\Theta = (\mathcal{X},\ddot{\mathcal{A}})$ and the factorized instance $\Theta_L = (\mathcal{X},\ddot{\mathcal{A}}_L)$, and suppose that the solution set to the initial instance factorized by congruences is a proper subset of the solution set to the factorized instance, i.e. $\{\mathcal{X}\to\ddot{\mathcal{A}}\}/\Sigma\subsetneq \{\mathcal{X}\to\ddot{\mathcal{A}}_L\}$. 

Then $V^1$ proves that there exists a subsequence of instance digraphs $\mathcal{X}=\mathcal{X}_0,...,\mathcal{X}_t$ (and a subsequence of target digraphs with domains $\ddot{\mathcal{A}}=\ddot{\mathcal{A}}_0,...,\ddot{\mathcal{A}}_s$), where $t\leq n(n-1)$ is the number of edges removed from $\mathcal{X}$, $\{\mathcal{X}_t\to\ddot{\mathcal{A}}_s\}/\Sigma\neq \{\mathcal{X}\to\ddot{\mathcal{A}}_L\}$, and if one removes any other edge from $\mathcal{X}_t$, every solution to ${\Theta_L}$ will be a solution to $\{\mathcal{X}_{t+1}\to\ddot{\mathcal{A}}_s\}/\Sigma$.
\end{theorem}
\begin{proof}

Since $\{\mathcal{X}\to\ddot{\mathcal{A}}\}/\Sigma\subsetneq \{\mathcal{X}\to\ddot{\mathcal{A}}_L\}$, there is some point $(a_1,...,a_k)$ in free variables $y_1,...,y_k$ such that $\phi(a_1,...,a_k)$ is a solution to $\Theta_L$, but if we restrict domains $D_0,...,D_{n-1}$ of $\Theta$ to congruences blocks corresponding to $\phi(a_1,...,a_k)$, there is no solution to $\Theta$. Thus, there is some homomorphism $H_L$ from $\mathcal{X}$ to $\ddot{\mathcal{A}}_L$ such that for any well-defined map $H$ from $\mathcal{X}$ to $\ddot{\mathcal{A}}$, where every $x_i$ is mapped to the corresponding domain $D_i$ and $H_L = H\circ H_c$, there exists an edge $E_{\mathcal{X}}(i_1,i_2)$ in $\mathcal{X}$ that failed to be mapped into an edge in $\ddot{\mathcal{A}}$.
The theory $V^1$ can count the number of elements in every set. Denote by $q = \verb|#| E_{\mathcal{X}}$ the number of edges in $\mathcal{X}$, $q\leq n^2$. Consider the following formula $\theta(t)$,
\begin{align*}\label{theta1}
&\theta(t) =_{def} \exists H_L < \langle n,\langle n,l\rangle\rangle, MAP(V_\mathcal{X}, n,V_{\ddot{\mathcal{A}}_L},\langle n,l\rangle,H_L)\wedge \\
&\hspace{20pt}\wedge(\forall i<n,w<\langle n,l\rangle, \,\, H_L(i)=w\rightarrow \exists a<l, w=\langle i,a\rangle\wedge D_i/\sigma_i(a))\wedge \\
&\hspace{60pt}\wedge\forall i_1,i_2<n,\forall w_1,w_2<\langle n,l\rangle\\
&\hspace{90pt}(E_{\mathcal{X}}(i_1,i_2)\wedge H_L(i_1)=w_1\wedge H_L(i_2)=w_2\rightarrow E_{\ddot{\mathcal{A}}_L}(w_1,w_2))\\
&\hspace{0pt}\wedge \\
&\hspace{10pt}\forall i,j<n, \,\, E_{\mathcal{X}_t} (i,j) \rightarrow E_{\mathcal{X}} (i,j) \wedge (q-t)\leq \texttt{\#}  E_{\mathcal{X}_t} (i,j) \wedge \\
&\hspace{165pt}\wedge\forall u,v<\langle n,l\rangle, \,\, E_{\ddot{\mathcal{A}}} (u,v) \rightarrow E_{\ddot{\mathcal{A}}_s} (u,v) \\
&\wedge \\
&\hspace{10pt}MAP(V_\mathcal{X},n,V_{\ddot{\mathcal{A}}},\langle n,l\rangle,H)\wedge 
\forall i<n,w<\langle n,l\rangle \\
&\hspace{173pt}H(i)=w\rightarrow \exists a<l, w=\langle i,a\rangle\wedge D_i(a) \\
&\wedge \\
& \hspace{10pt}\forall i<n,v<\langle n,l\rangle,\, H_L(i)=v \longleftrightarrow \exists u<\langle n,l\rangle (H(i)=u\wedge H_c(u)=v) \\
&\Longrightarrow \\
&\hspace{10pt}\exists i_1,i_2<n,\exists w_1,w_2 <\langle n,l\rangle,\,\neg (E_{\mathcal{X}_t}(i_1,i_2)\wedge H(i_1)=w_1\wedge H(i_2)=w_2\rightarrow\\
&\hspace{270pt}\rightarrow E_{\ddot{\mathcal{A}}_s}(w_1,w_2)).
\end{align*}
The first part of the formula expresses that there is a homomorphism $H_L$ from $\mathcal{X}$ to $\ddot{\mathcal{A}}_L$.  The second part formalizes that the input digraph $\mathcal{X}_t$ is constructed from $\mathcal{X}$ by removing at least $t$ edges (and the target digraph $\ddot{\mathcal{A}}_s$ is constructed from $\ddot{\mathcal{A}}$ by adding some edges).  The third and forth parts say that there is a well-defined map $H$ from $V_{\mathcal{X}}$ to $V_{\ddot{\mathcal{A}}}$ satisfying all restrictions on domains and such that $H_L$ is a composition of $H$ and the canonical homomorphism $H_c$. And the last part expresses that if all previous conditions are true, then $H$ cannot be a homomorphism from $\mathcal{X}_t$ to $\ddot{\mathcal{A}}_s$.

In the formula $\theta(t)$ as fixed parameters we use $\mathcal{X} = (V_{\mathcal{X}}, E_{\mathcal{X}})$, $q = \verb|#| E_{\mathcal{X}}$, the target digraph with domains $\ddot{\mathcal{A}} = (V_{\ddot{\mathcal{A}}}, E_{\ddot{\mathcal{A}}})$, $V_{\ddot{\mathcal{A}}}<\langle n,l\rangle$ and $ \verb|#| E_{\ddot{\mathcal{A}}}<\langle n,l\rangle^2$, the factorized digraph with domains $\ddot{\mathcal{A}}_L = (V_{\ddot{\mathcal{A}}_L}, E_{\ddot{\mathcal{A}}_L})$ and the canonical homomorphism $H_c$. Induction goes on variables $t$ and the instance digraph $\mathcal{X}_t = (V_{\mathcal{X}}, E_{\mathcal{X}_t})$ such that $(q-t) = \verb|#| E_{\mathcal{X}_t}$. Finally, witnesses in $\Sigma^{1,b}_1$-induction corresponding to $t$ are the target digraph with domains $\ddot{\mathcal{A}}_s = (V_{\ddot{\mathcal{A}}_s},E_{\ddot{\mathcal{A}}_s})$ and the map $H$ from $V_\mathcal{X}$ to $V_{\ddot{\mathcal{A}}}$.

By assumption, the formula $\theta(t)$ is true for $t=0$. We also know that it is false for $t=q$ since for all $i_1,i_2<n$ there is $\neg\,E_{\mathcal{X}_t}(i_1,i_2)$. Since $\theta(t)$ is $\Sigma_1^{1,b}$-formula, we can use the Number maximization axiom:
$$
   \forall H\leq \langle n,\langle n,l\rangle\rangle,\forall \ddot{\mathcal{A}}_s, \big[\theta(0)\to \exists q'\leq q (\theta(q')\wedge \neg \exists q''\leq q(q'<q''\wedge \theta(q'')))\big].
$$
This completes the proof.
\end{proof}

\begin{lemmach}\label{ddkisuy}
Consider two CSP instances, the initial instance $\Theta = (\mathcal{X},\ddot{\mathcal{A}})$ and the instance $\Theta_{t,s} = (\mathcal{X}_t,\ddot{\mathcal{A}}_s)$, where $t\leq n(n-1)$ is the number of edges removed from the initial digraph $\mathcal{X}$ and $s\leq \langle n,l\rangle^2$ is the number of edges added to the target digraph $\ddot{\mathcal{A}}$. $V^1$ proves that instance $\Theta$ has a solution only if $\Theta_{t,s}$ has a solution.
\end{lemmach}
\begin{proof}
Suppose that there is a solution to the instance $\Theta$, a homomorphism $H$, and the instance $\Theta_{t,s}$ is constructed from $\Theta$ by removing $t$ arbitrary edges from $\mathcal{X}$ and adding some $s$ edges to $\ddot{\mathcal{A}}$. Then it is straightforward to check that $H$ is also a solution to $\Theta_{t,s}$. 
\end{proof}
For further iterations of Zhuk's algorithm, we will prove the following theorem. 
\begin{theorem}\label{apuyttr65e}
Consider two CSP instances, the initial instance $\Theta = (\mathcal{X},\ddot{\mathcal{A}})$ and the instance $\Theta_{t,s} = (\mathcal{X}_t,\ddot{\mathcal{A}}_s)$, where $t\leq n(n-1)$ is the number of edges removed from the initial digraph $\mathcal{X}$ and $s\leq \langle n,l\rangle^2$ is the number of edges added to the target digraph with domains $\ddot{\mathcal{A}}$. Suppose that the solution set to the initial instance factorized by congruences is a proper subset of the intersection of the solution set to the instance $\Theta_{t,s}$ factorized by congruences and the solution set to the factorized instance $\Theta_L$, i.e. $\{\mathcal{X}\to\ddot{\mathcal{A}}\}/\Sigma\subsetneq \{\mathcal{X}_t\to\ddot{\mathcal{A}}_s\}/\Sigma\cap \{\mathcal{X}\to\ddot{\mathcal{A}}_L\}$.

Then $V^1$ proves that there exists a subsequence of instance digraphs $\mathcal{X}=\mathcal{X}_0,...,\mathcal{X}_r$ (and a subsequence of target digraphs with domains $\ddot{\mathcal{A}}=\ddot{\mathcal{A}}_0,...,\ddot{\mathcal{A}}_f$), where $r\leq n(n-1)$ is the number of edges removed from $\mathcal{X}$ such that $\{\mathcal{X}_r\to\ddot{\mathcal{A}}_f\}/\Sigma\neq \{\mathcal{X}_t\to\ddot{\mathcal{A}}_s\}/\Sigma\cap\{\mathcal{X}\to\ddot{\mathcal{A}}_L\}$ and if one removes any other edge from $\mathcal{X}_r$, every solution to $\{\mathcal{X}_t\to\ddot{\mathcal{A}}_s\}/\Sigma\cap\{\mathcal{X}\to\ddot{\mathcal{A}}_L\}$ will be a solution to $\{\mathcal{X}_{r+1}\to\ddot{\mathcal{A}}_f\}/\Sigma$.
\end{theorem}

\begin{proof}
The proof is analogous to the proof of Theorem \ref{FirstIterationOf2}. Let us define a slightly modified formula $\theta'(r)$. We now consider two homomorphisms, $H_L$ from $\mathcal{X}$ to $\ddot{\mathcal{A}}_L$, and $H_{t,s}$ from $\mathcal{X}_t$ to $\ddot{\mathcal{A}}_s$ such that $H_L$ is a composition of $H_{t,s}$ and canonical homomorphism $H_c$ (it is equivalent to the condition that solutions to both instances are in $\{\mathcal{X}_t\to\ddot{\mathcal{A}}_s\}/\Sigma\cap\{\mathcal{X}\to\ddot{\mathcal{A}}_L\}$).
\begin{align*}
&\theta(r) =_{def} \exists H_L < \langle n,\langle n,l\rangle\rangle, MAP(V_\mathcal{X}, n,V_{\ddot{\mathcal{A}}_L},\langle n,l\rangle,H_L)\wedge \\
&\hspace{20pt}\wedge(\forall i<n,w<\langle n,l\rangle, \,\, H_L(i)=w\rightarrow \exists a<l, w=\langle i,a\rangle\wedge D_i/\sigma_i(a))\wedge \\
&\hspace{60pt}\wedge\forall i_1,i_2<n,\forall w_1,w_2<\langle n,l\rangle,\\
&\hspace{80pt}(E_{\mathcal{X}}(i_1,i_2)\wedge H_L(i_1)=w_1\wedge H_L(i_2)=w_2\rightarrow E_{\ddot{\mathcal{A}}_L}(w_1,w_2))\\
%%%
&\wedge \\
%%%%%
&\hspace{10pt}\exists H_{t,s} < \langle n,\langle n,l\rangle\rangle \big(MAP(V_{\mathcal{X}_t}, n,V_{\ddot{\mathcal{A}}_s},\langle n,l\rangle,H_{t,s})\wedge \\
&\hspace{20pt}\wedge(\forall i<n,w<\langle n,l\rangle \,\, H_{t,s}(i)=w\rightarrow \exists a<k, w=\langle i,a\rangle\wedge D_i(a))\wedge \\
&\hspace{60pt}\wedge\forall i_1,i_2<n,\forall w_1,w_2<\langle n,l\rangle\\
&\hspace{75pt}(E_{\mathcal{X}_t}(i_1,i_2)\wedge H_{t,s}(i_1)=w_1\wedge H_{t,s}(i_2)=w_2\rightarrow E_{\ddot{\mathcal{A}}_s}(w_1,w_2))\\
%%%
&\wedge \\
%%%%%
&\hspace{10pt}\forall i<n,v<\langle n,l\rangle\, H_L(i)=v \longleftrightarrow \exists u<\langle n,k\rangle (H_{t,s}(i)=u\wedge H_c(u)=v) \\
%%%
&\wedge \\
%%%%%
&\hspace{10pt}\forall i,j<n \,\, E_{\mathcal{X}_r} (i,j) \rightarrow E_{\mathcal{X}} (i,j) \wedge (q-r)\leq \texttt{\#}  E_{\mathcal{X}_r} (i,j)\wedge \\
&\hspace{165pt}\wedge\forall u,v<\langle n,l\rangle \,\, E_{\ddot{\mathcal{A}}} (u,v) \rightarrow E_{\ddot{\mathcal{A}}_f} (u,v) \\
%%%
&\wedge \\
%%%%%
&\hspace{10pt}MAP(V_\mathcal{X},n,V_{\ddot{\mathcal{A}}},\langle n,l\rangle,H)\wedge 
\forall i<n,w<\langle n,l\rangle \\
&\hspace{173pt}H(i)=w\rightarrow \exists a<l, w=\langle i,a\rangle\wedge D_i(a) \\
%%%
&\wedge \\
%%%%%
& \hspace{10pt}\forall i<n,v<\langle n,l\rangle\, Z_L(i)=v \longleftrightarrow \exists u<\langle n,l\rangle (H(i)=u\wedge H_c(u)=v) \\
&\Longrightarrow \\
&\hspace{10pt}\exists i_1,i_2<n,\exists w_1,w_2 <\langle n,l\rangle\,\neg (E_{\mathcal{X}_r}(i_1,i_2)\wedge H(i_1)=w_1\wedge H(i_2)=w_2\\
&\hspace{265pt}\rightarrow E_{\ddot{\mathcal{A}}_f}(w_1,w_2)).
\end{align*}
In formula $\theta'(r)$ as fixed parameters we use parameters similar to parameters in the formula $\theta(t)$, but add here $\mathcal{X}_t = (V_{\mathcal{X}}, E_{\mathcal{X}_t})$, $q-t = \verb|#| E_{\mathcal{X}_t}$ and $\mathcal{A}_s = (V_{\mathcal{A}_s}, E_{\mathcal{A}_s})$, $s <\langle n,l\rangle^2$ as well. Induction goes on variable $r$ and the instance digraph $\mathcal{X}_r = (V_{\mathcal{X}}, E_{\mathcal{X}_r})$ such that $(q-r) \leq \verb|#| E_{\mathcal{X}_r}$. Witnesses to the induction are the target digraph with domains $\ddot{\mathcal{A}}_f = (V_{\ddot{\mathcal{A}}_f},E_{\ddot{\mathcal{A}}_f})$ and the map $H$ from $V_\mathcal{X}$ to $V_{\ddot{\mathcal{A}}}$.
\end{proof}

\subsection{The main result}\label{kkaiery465r}
\begin{theorem}[The main result]
For any fixed relational structure $\mathcal{A}$ which corresponds to an algebra with WNU operation and therefore leads to $p$-time solvable CSP, the theory $V^1_{\mathcal{A}}$ proves the soundness of Zhuk's algorithm.
\end{theorem}

\begin{proof}
Consider any unsatisfiable CSP instance $\Theta=(\mathcal{X},\ddot{\mathcal{A}})$. It is sufficient to show that in the computation $W = (W_1,W_2,...,W_k)$ of the algorithm on $\mathcal{X}$, for all possible types of algorithmic modifications the theory $V^1_{\mathcal{A}}$ proves that $W_i$ has a solution only if $W_{i+1}$ has a solution.

In Section \ref{kkspqu5465} we have shown that $V^1$ proves that:
\begin{itemize}
    \item the instance $\Theta$ has a solution only if it has a solution after cycle-consistency reduction (Lemma \ref{aappppqppqpqpq});
    \item the instance $\Theta$ has a solution only if it has a solution after irreducible reduction (Lemma \ref{qqnhygtfrd});
    \item the instance $\Theta$ has a solution only if it has a solution after the weaker instance reduction (Lemma \ref{gfhdyejr7454}).
\end{itemize}

The three universal algebra axiom schemes BA$_{\mathcal{A}}$-axioms, CR$_{\mathcal{A}}$-axioms, and PC$_{\mathcal{A}}$-axioms defined in Section \ref{66347he} by $\forall\Sigma^{1,b}_2$-formulas validate universal algebra reductions of any domain $D_i$ to a binary absorbing subuniverse, central subuniverse or to an arbitrary equivalence class of polynomially complete congruence on $D_i$.

Finally, in Section \ref{LLLLLINEARG} we have shown that $V^1$ validates:

\begin{itemize}
    \item factorization of the instance by minimal linear congruences (Theorem \ref{INITIALFACTOR});
        \item Gaussian elimination (Lemma \ref{gaussssiy643});
            \item decreasing of the solution set to the factorized instance (Theorems \ref{FirstIterationOf2}, \ref{apuyttr65e}, Lemma \ref{ddkisuy}).
\end{itemize}
This completes the proof.
\end{proof}
The result implies that tautologies $\neg HOM(\mathcal{X},\mathcal{A})$ for
negative instances of any fixed $p$-time CSP have short proofs in any propositional proof system simulating Extended Resolution
and a theory that proves the three universal algebra axioms.

\section{Conclusion notes}

In the paper we investigate the proof complexity of general CSP. We proved the soundness of Zhuk's algorithm in a new theory of bounded arithmetic defined by augmenting the two-sorted theory $V^1$ with three universal algebra axioms. These axioms are designed to verify universal algebra reductions, while the soundness of consistency reductions and the linear case of the algorithm is proved directly in the theory $V^1$. 

Consistency reductions open the algorithm and represent its most technical part. Formalization of the consistency reductions uses iteratively defined sets and $\Sigma^{1,b}_1$-induction. The linear case is the last step of Zhuk's algorithm after all reductions of separate domains. However, it does not lead to linear equations straightforwardly: structures in the linear case have to be factorized first. The proof of the soundness of the linear case is based on the formalization of Gaussian elimination and linear factorization and uses $\Sigma^{1,b}_1$-induction.

In contrast, universal algebra axioms stand apart. Despite the fact that they can be defined by $\forall\Sigma^{1,b}_2$-formulas, their proof in a theory of bounded arithmetic requires the formalization of advanced notions from universal algebra and this will be a subject of further research.

Theorem \ref{aakrrrfetsplgh} allows one to consider constraint languages with at most binary relations instead of general CSP. We tested how to utilize the framework and strategy of getting short propositional proofs using bounded arithmetic in \cite{10.1093/logcom/exab028} on an elementary example of undirected graphs (the $\mathcal{H}$-coloring problem). In that case, the theory of bounded arithmetic corresponds to a weak proof system $R^*(log)$, a mild extension of resolution. 

Every theory of bounded arithmetic corresponds to some propositional proof system. The theory $V^1$ stands for polynomial time reasoning and corresponds to the Extended Frege EF proof system (equivalently Extended resolution ER). Our working hypothesis is that the soundness of Zhuk's algorithm can be established utilizing only $\Sigma^{1,b}_1$-induction. If it is true, then statements $\neg HOM(\mathcal{X},\ddot{\mathcal{A}})$ for unsatisfiable instances of polynomial time CSP($\mathcal{A}$) will have short propositional proofs in EF. The next step in our program is to investigate the boundaries of the theory $V^1$ in formalizing of universal algebra notions. 

\bigskip
\noindent
\textbf{Acknowledgements:}{ I would like to thank my supervisor Jan Krajíček for many helpful comments that resulted in many improvements to this paper. I thank Dmitriy Zhuk for answering my questions about his results. Also, I’m grateful to Michael Kompatscher for a number of discussions on universal algebra. Finally, I would like to thank Emil Jeřábek for the expert remarks on formalization in bounded arithmetic.}

\bibliographystyle{plain}  

\bibliography{bibliography}

\end{document}